\documentclass{article}[a4]
\usepackage[utf8]{inputenc}
\usepackage[margin=3cm]{geometry}

\usepackage{amsmath}
\usepackage{amsthm}
\usepackage{amsfonts}
\usepackage{amssymb}
\usepackage{mathtools}
\usepackage{bbm}
\usepackage{mathrsfs}
\usepackage{xcolor}
\usepackage{enumitem}
\usepackage[hidelinks]{hyperref}
\usepackage{listings}

\usepackage{xpatch}
\xpretocmd{\proof}{\setlength{\parindent}{0pt}}{}{}

\theoremstyle{definition}
\newtheorem{defi}{Definition}[section]

\theoremstyle{plain}
\newtheorem{thm}[defi]{Theorem}
\newtheorem*{thm*}{Theorem}

\newcommand{\thistheoremname}{}
\newtheorem{genericthm}[defi]{\thistheoremname}

\newtheorem*{genericthm*}{\thistheoremname}
\newenvironment{namedthm*}[1]{\renewcommand{\thistheoremname}{#1}%
	\begin{genericthm*}}{\end{genericthm*}}

\newtheorem{lem}[defi]{Lemma}
\newtheorem*{lem*}{Lemma}
\newtheorem{cor}[defi]{Corollary}
\newtheorem{prop}[defi]{Proposition}
\newtheorem{conj}[defi]{Conjecture}

\theoremstyle{remark}
\newtheorem{rem}[defi]{Remark}

\newcommand{\thisremname}{}
\newtheorem{genericrem}[defi]{\thisremname}

\newcommand{\mcal}{\mathcal}

\DeclareMathOperator{\cov}{Cov}
\DeclareMathOperator{\var}{Var}
\DeclareMathOperator{\dom}{dom }
\DeclareMathOperator{\opL}{L}
\DeclareMathOperator{\Fint}{I}
\DeclareMathOperator{\MalD}{D}
\DeclareMathOperator{\sgn}{sgn}
\DeclareMathOperator{\diam}{diam}
\DeclareMathOperator{\dist}{dist}

\DeclarePairedDelimiter\ceil{\lceil}{\rceil}
\DeclarePairedDelimiter\floor{\lfloor}{\rfloor}

\newcommand{\R}{\mathbb{R}}
\newcommand{\N}{\mathbb{N}}

\newcommand{\p}{\mathbb{P}}

\newcommand{\E}{\mathbb{E}}

\newcommand{\X}{\mathbb{X}}
\newcommand{\Y}{\mathbb{Y}}
\newcommand{\W}{\mathbb{W}}

\newcommand{\1}{\mathbbm{1}}

\newcommand{\Fa}{F^{(\alpha)}}
\newcommand{\fat}{F^{(\alpha)}_t}
\newcommand{\fad}{F^{(d/2)}_t}
\newcommand{\lat}{L^{(\alpha)}_t}
\newcommand{\lhat}{\hat{L}^{(\alpha)}_t}

\newcommand{\DD}{\MalD^{(2)}}
\newcommand{\asp}{\p\text{-a.s.}}
\newcommand{\yint}{\int_\Y}

\renewcommand{\geq}{\geqslant}
\renewcommand{\leq}{\leqslant}
\renewcommand{\[}{\begin{equation}}
\renewcommand{\]}{\end{equation}}

\newcommand{\ind}[1]{\mathbbm{1}_{\{#1\}}}
\newcommand{\norm}[1]{\lVert #1 \rVert}

\newcommand{\overbar}[1]{\mkern 1.5mu\overline{\mkern-1.5mu#1\mkern-1.5mu}\mkern 1.5mu}
\newcommand{\bb}{\overbar{B}}

\newcounter{savedenum}

\numberwithin{equation}{section}

\title{Quantitative CLTs on the Poisson space\\
	via Skorohod estimates and $p$-Poincaré inequalities}
\author{Tara Trauthwein\thanks{Department of Mathematics, University of Luxembourg, Luxembourg, tara.trauthwein@uni.lu. The author was supported by the Luxembourg National Research Fund (PRIDE17/1224660/GPS).}}
\date{}

\begin{document}

\maketitle

\begin{abstract}
	We establish new explicit bounds on the Gaussian approximation of Poisson functionals based on novel estimates of moments of Skorohod integrals. Combining these with the Malliavin-Stein method, we derive bounds in the Wasserstein and Kolmogorov distances whose application requires minimal moment assumptions on add-one cost operators -- thereby extending the results from (Last, Peccati and Schulte, 2016). Our applications include a CLT for the Online Nearest Neighbour graph, whose validity was conjectured in (Wade, 2009; Penrose and Wade, 2009). We also apply our techniques to derive quantitative CLTs for edge functionals of the Gilbert graph, of the $k$-Nearest Neighbour graph and of the Radial Spanning Tree, both in cases where qualitative CLTs are known and unknown.
\end{abstract}



\tableofcontents

\section{Introduction}\label{secIntro}

\subsection{Overview}\label{secIntroOverview}

The aim of this paper is to establish a new collection of probabilistic inequalities, yielding quantitative CLTs for sequences of Poisson functionals under minimal moment assumptions. We will see below that our findings substantially extend and refine the second order Poincaré inequalities proved in \cite{LPS14} (see also \cite{LRSY19,SY19,SY21}), and heavily rely on new moment inequalities for Skorohod integrals (see Theorem~\ref{thmGeneralpPoin}) that we believe to be of independent interest. As demonstrated in Section~\ref{secAppl}, our findings are specifically tailored to deriving quantitative CLTs for functionals of spatial random graphs based on Poisson inputs, in critical or near-critical regimes.

One prominent example dealt with in the present work is the \textbf{Online Nearest Neighbour Graph} (ONNG), devised by Berger et al. in \cite{BBBCR07} as a simplified version of the FKP model of the internet graph (see \cite{FKP02}). The set-up is the following: Take a Poisson measure on $\R^d \times [0,1]$, in such a way that each point of the measure has a spatial coordinate in $\R^d$ and an arrival time in $[0,1]$. Within a bounded observation window, connect each point to its nearest neighbour in space which has smaller arrival time. The resulting ONNG is a tree growing in time -- a simple model for an expanding network. Other graphs of interest include the \textbf{Gilbert graph}, where two points are connected if they are close enough to one another, or nearest-neighbour type graphs like the $k$\textbf{-Nearest Neighbour graph} and the \textbf{Radial Spanning Tree} (see Section~\ref{secAppl}).

For graphs like these, one quantity of interest is the \textbf{total edge-length}, or more generally, the $\alpha$\textbf{-power weighted total edge-length}:
\begin{equation}\label{eqIntroF}
F = \sum_{e \text{ edge}} |e|^\alpha.
\end{equation}
The main question is to understand how such a sum fluctuates as the graph expands. For the ONNG, convergence to the normal law was shown by Penrose \cite{Pen05} in the exponent range $\alpha \in \left(0,\frac{d}{4}\right)$ and conjectured in \cite{Wade2009,PW09} for the range $\alpha \in \left[\frac{d}{4},\frac{d}{2}\right]$. This conjecture has remained open until now: as an application of our abstract bounds, we settle it in this article by providing a quantitative central limit theorem for the centred and rescaled sum of power-weighted edge-lengths with powers $\alpha \in \left.\left(0,\frac{d}{2}\right.\right]$ (see Theorem~\ref{thmONNG}).

\subsection{Main contributions}
As announced in Section~\ref{secIntroOverview}, our theoretical findings refine the main bounds in \cite{LPS14}, where the authors proved second order Poincaré inequalities on configuration spaces -- thus extending the Gaussian second order Poincaré estimates established in \cite{Chatterjee2009,PNR2009,Vid20,ERTZ21} to the Poisson case. The results of \cite{LPS14} are based on the combination of \textbf{Stein's method} \cite{CGS11,PeccNourd,PeccReitz} and \textbf{Malliavin Calculus} on configuration spaces \cite{LastPoiss,LPenLectures}.

The starting point of Stein's method is the fact that a real-valued random variable $N$ is standard Gaussian if and only if
\[
\E f(N)N = \E f'(N)
\]
for a suitable collection of functions $f:\R \rightarrow \R$. This allows one to represent the probability distance between a random variable $F$ and a standard normal $N$ as
\begin{equation}\label{eqIntrod}
d(F,N) = \sup_{h \in \mathcal{H}} |\E f_h(F)F - \E f_h'(F)|,
\end{equation}
where $\mathcal{H}$ is a suitable collection of functions depending on the choice of distance and the function $f_h$ is the canonical solution to the differential equation
\[
f'_h(x) = xf_h(x) + h(x) - \E h(N).
\]
The crucial idea behind the results of \cite{LPS14} is that, for random variables $F=F(\eta)$ depending on a Poisson measure $\eta$ on a space $\X$, one can control quantities such as \eqref{eqIntrod} by using integration by parts formulae involving the \textbf{add-one cost operator}
\[
\MalD_x F(\eta) := F(\eta + \delta_x) - F(\eta),
\]
where $x \in \X$ and $\delta_x$ is the Dirac measure in $x$, as well as its iteration $\DD_{x,y} = \MalD_x \MalD_y F$. As demonstrated in \cite{LPS14,LRSY19,SY19,LachPeccYang,SY21, BSP22}) such an approach leads to flexible bounds in the Kolmogorov and Wasserstein distances, bounds that are particularly adapted for dealing with functionals displaying a form of \textbf{geometric stabilisation} -- see e.g. \cite{PY01,Pen05,PY05,LRSY19,LachPeccYang} for a discussion of this concept, as well as \cite{KL96} for the first seminal contribution on the topic.

One of the shortcomings of the bounds established in \cite{LPS14,LRSY19,SY19,LachPeccYang,BSP22}) is that their use in concrete applications typically requires one to uniformly bound over $\X$ the moments of order $(4+\epsilon)$ (with $\epsilon>0$) of $\MalD F$ and $\DD F$. Such a uniform bound is not achievable in many relevant applications, e.g. for edge functionals of the ONNG in the exponent range $\alpha \in \left[\frac{d}{4},\frac{d}{2}\right]$, the range where the central limit theorem was conjectured to hold. We substantially extend the main results from \cite{LPS14} in two ways:

\begin{enumerate}
	\item In Theorem~\ref{thmWasserNon-Cond.} we establish explicit bounds in the Kolmogorov and Wasserstein distances, whose use only requires one to uniformly bound the moments of order $2+\epsilon$ of add-one cost operators. (See also \cite{Tri19} for qualitative CLTs requiring both bounds on the moments of order $2+\epsilon$ and weak stabilisation). To motivate the reader, we will now give a simple example of a consequence of Theorem~\ref{thmWasserNon-Cond.}.
	
	Let $\nu$ be a centred probability measure on $\R$ such that
	\[
	c:=\int_\R |u|^{2+\epsilon} \nu(du) < \infty
	\]
	for some $\epsilon > 0$. Define $\sigma^2 := \int_\R |u|^{2} \nu(du)$ and let $\eta$ be a Poisson measure on $\R \times [0,\infty)$ with intensity $\nu(du) \otimes ds$. Let $T > 0$ and define
	\[
	F_T := \int_0^T \int_\R u\ \eta(du,ds).
	\]
	Then $F_T$ is equal in law to $\sum_{i=1}^{N(T)} X_i$, where $X_1,X_2,...$ are i.i.d random variables distributed according to $\nu$ and independent of $N(T)$, a Poisson distributed random variable with parameter $T$.
	
	One can compare this to the random variable
	\[
	G_n := \sum_{i=1}^{n} X_i,
	\]
	where the number of points is deterministically given by $n$.
	By an extension of the classical Berry-Esseen theorem given by Petrov in \cite[Theorem~6, p.115]{Petrov}, one has the following bound on the Kolmogorov distance between the laws of $(n\sigma^2)^{-1/2}G_n$ and a standard Gaussian $N$:
	\[
	d_K\left(\frac{G_n}{\sqrt{n}\sigma},N\right) \leq \frac{c}{\sigma^{2+\epsilon}} n^{-\epsilon/2}.
	\]
	For the functional $F_T$, our Theorem~\ref{thmWasserNon-Cond.} implies
	\begin{align}
		d_W\left(\frac{F_T}{\sqrt{T}\sigma},N\right) \leq \frac{2c}{\sigma^{2+\epsilon}} T^{-\epsilon/2}
		\intertext{and}
		d_K\left(\frac{F_T}{\sqrt{T}\sigma},N\right) \leq \frac{(4c)^{\frac{1}{1+\epsilon/2}}}{\sigma^{2}} T^{-\frac{\epsilon/2}{1+\epsilon/2}}
	\end{align}
	for the Wasserstein and Kolmogorov distances respectively. The speed of convergence in Wasserstein distance corresponds exactly to the one given by Petrov. For the Kolmogorov distance we find a slightly slower speed, which is however still converging much faster than the square root of the Wasserstein distance, which is implied by the classic estimate $d_K(.,N) \leq 2 \sqrt{d_W(.,N)}$ (see e.g. \cite[Remark~C.2.2]{PeccNourd}).
	
	In Section~\ref{secAppl}, we apply Theorem~\ref{thmWasserNon-Cond.} to edge-statistics of the form \eqref{eqIntroF} of the Online Nearest Neighbour Graph in the exponent range $\alpha \in \left(0,\frac{d}{2}\right)$ and of the Gilbert graph for exponents $\alpha \in \left(-\frac{d}{2},\infty\right)$ (extending existing results from \cite{TSRGilbert}); we also deal with the $k$-Nearest Neighbour Graph and the Radial Spanning Tree for a general class of decreasing functions $\phi:(0,\infty) \rightarrow (0,\infty)$ applied to the edge-lengths. In all our applications, the speeds we find are the same in the Wasserstein and Kolmogorov distances. Roughly speaking, a $2p$-moment bound leads to a speed of convergence of $t^{d(1/p-1)}$, where $t^d$ is the order of the variance and $p \in (1,2]$. If $p=2$, we recover the speed of order `square root of the variance', which is often presumed to be optimal and has in some contexts been shown to be optimal. If however $p<2$, the resulting speed is slower. Comparing with the above example and setting $2p=2+\epsilon$, it corresponds to $t^{-d\frac{\epsilon/2}{1+\epsilon/2}}$ in both Wasserstein and Kolmogorov distances. Whether or not this speed is optimal remains an open question.
	
	
	\item The case of an edge functional of the type \eqref{eqIntroF} with $\alpha = \frac{d}{2}$ for the ONNG is of a different nature. The variance contains an additional logarithmic factor (as conjectured and partially shown in \cite{Wade2009}, and fully established in Theorem~\ref{thmONNG}) and a $2+\epsilon$ moment bound proves too strong a condition. To deal with this particular case, we develop in Theorem~\ref{thmWasserCond} an estimate of the Wasserstein distance that depends on a time parameter. Taking $\eta$ to be a Poisson measure on a space $\X \times [0,1]$ with intensity $\lambda \otimes ds$, the estimate contains moments of $\E \big[\MalD_{(x,s)}F \big| \eta_{|\X \times [0,s)}\big]$ and $\E \big[\MalD^{(2)}_{(x,s),(y,u)}F \big| \eta_{|\X \times [0,s \vee u)}\big]$ instead of $\MalD_{(x,s)}F$ and $\MalD^{(2)}_{(x,s),(y,u)}F$. This distinction is crucial and leads to a quantitative central limit theorem in the critical case $\alpha = \frac{d}{2}$, stated in Theorem~\ref{thmONNG}.
\end{enumerate}

The underlying result making these improved estimates possible is a new inequality, stated in Theorem~\ref{thmGeneralpPoin}, providing a bound on the $p$th moment of a Skorohod integral $\delta(h)$ (see Section~\ref{secFrame} for a definition), with $p \in [1,2]$. Even more generally, a bound is provided for quantities of the type $|\E \phi(\delta(h))|$, where $\phi:\R \rightarrow \R$ is differentiable with $(p-1)$-Hölder continuous derivative for some $p \in (1,2]$. To the best of our knowledge, no comparable inequality exists to date and this result is of independent interest.
In particular, Theorem~\ref{thmGeneralpPoin} implies generalisations of the classical \textbf{Poincaré inequality} according to which for a Poisson functional $F \in L^2(\p_\eta)$, it holds that
\[
\E \left[F^2\right] - \E [F]^2 \leq \E \int \left(\MalD_x F\right)^2 \lambda(dx).
\]
As shown in Corollary~\ref{corpPoinGen} and Remark~\ref{remPPoin}, the Poincaré inequality remains true (up to a multiplying constant) if the exponent $2$ is replaced by $p$, thus for $p \in [1,2]$:
\[
\E |F|^p - |\E F|^p \leq 2^{2-p} \E \int_\X \left|\MalD_{x}F\right|^p \lambda(dx).
\]
If the functional $F$ is non-negative or centred, this inequality follows from modified log-Sobolev type inequalities shown in \cite{Chafai} and also \cite{AdamPola}. We stress, however, that the bound in Theorem~\ref{thmGeneralpPoin} is much more general and not directly deducible from \cite{Chafai,AdamPola}.

If we work over the time-augmented space $\X \times [0,1]$, it has been shown by Last and Penrose in \cite[Thm.~1.5]{LPen1} that
\[
\var(F) = \E \int_\X \int_0^1 \E [\MalD_{(x,t)}F | \eta_{|\X \times [0,t)}]^2 \lambda(dx)dt.
\]
Correspondingly, the inequality in Corollary~\ref{corpPoinGen} can be refined even further by conditioning on the information up to time $t$:
\[
\E |F|^p - |\E F|^p \leq 2^{2-p} \E \int_\X \int_0^1 \left|\E [\MalD_{(x,t)}F | \eta_{|\X \times [0,t)} ]\right|^p \lambda(dx)dt.
\]
This inequality is of great importance for improving estimates of both Wasserstein and Kolmogorov distances. Another consequence of Theorem~\ref{thmGeneralpPoin} is a more technical inequality given in Corollary~\ref{corKolgIneq}, this one being crucial when refining the bound on the Kolmogorov distance.

The proof of Theorem~\ref{thmGeneralpPoin} relies on a new version of Itô formula, shown in Theorem~\ref{thmItoFormula}. In contrast to the classical Itô formula for Poisson point processes as given in \cite[Theorem~II.5.1]{IkedaWatanabe}, our version does not assume the process to be a semi-martingale or the integrand to be predictable. In turn, we only use the term corresponding to the integral with respect to a compensated Poisson measure. A detailed discussion of differences and similarities with the classical Itô formula and comparable results in the literature is provided in Section~\ref{secIto}. We believe this result also to be of independent interest, as to the best of our knowledge no such formula for anticipative integrands and general Poisson processes exists in the literature.

\begin{rem}
	As demonstrated in the forthcoming Section~\ref{secAppl}, the principal achievement of this paper is the derivation of probabilistic bounds requiring minimal moment assumptions, that one can directly apply to a variety of models without implementing truncation or smoothing procedures. That said, it is plausible to expect that alternate bounds to some of those derived in Section~\ref{secAppl} could be derived by combining the results of \cite{LPS14} with a truncation procedure similar to the ones implemented e.g. in \cite[proof of Theorem~1.1]{Wu} or \cite[proof of Corollary~3.2]{NPY20}. In order to keep the length of this paper within reasonable limits, the comparison between the two approaches (in situations where both apply) will be discussed elsewhere.
\end{rem}

\noindent\underline{\textit{Plan of the paper:}} In Section~\ref{secFrame} we provide a short background on Poisson processes and Malliavin Calculus, with a more detailed account to be found in Appendix~\ref{secBackground}. In Section~\ref{secWasserKolg}, we discuss second order $p$-Poincaré inequalities, while Section~\ref{secSkorEstimates} contains our version of Itô formula and the new estimates for Skorohod integrals. Applications are discussed in Section~\ref{secAppl}, in particular the ONNG is dealt with in Section~\ref{secONNG}. All proofs can be found in the appendices: in Appendix~\ref{secPIto} we present the proof of Itô formula (Theorem~\ref{thmItoFormula}), Appendix~\ref{secPPPoin} contains the proofs for Theorem~\ref{thmGeneralpPoin} and its corollaries, and the second order Poincaré inequalities (Theorems~\ref{thmWasserKolg1}, \ref{thmWasserCond} and \ref{thmWasserNon-Cond.}) are shown in Appendix~\ref{secProofWK}. The ONNG is discussed in Section~\ref{secPONNG} and the proofs for the Gilbert graph, the $k$-Nearest Neighbour graphs and the Radial Spanning Tree can be found in Appendices~\ref{secPGil}, \ref{secPkNN} and \ref{secPRad} respectively.\\

\noindent\textbf{Acknowledgment.} I would like to thank my supervisor Giovanni Peccati for extensive discussions and his invaluable help on this project. I would also like to thank Pierre Perruchaud for his contribution to the computation of the constant discussed in Lemmas~\ref{lemI_2} and \ref{lemI_2'}. I am grateful to Günter Last, Matthias Schulte and Mark Podolskij for useful discussions and comments.

\section{Framework and notations}\label{secFrame}
We provide here an overview of the most relevant (in the context of this article) properties of Poisson point processes and elements of Poisson Malliavin calculus. Further definitions and properties that will be necessary for the proofs can be found in Appendix~\ref{secBackground}. We refer the reader to \cite{LastPoiss,LPenLectures} for an exhaustive discussion of the material presented below.\\

\noindent \underline{\textit{Poisson random measure.}}
Let $(\W,\mcal{W},\nu)$ be a $\sigma$-finite measure space and let $\mathbf{N}_\W$ be the set of $\N_0 \cup \{\infty\}$-valued measures on $(\W,\mcal{W})$. Define the $\sigma$-algebra $\mcal{N}_\W$ on $\mathbf{N}_\W$ as the smallest $\sigma$-algebra such that $\forall\, W \in \mcal{W}$, the map $\mathbf{N}_\W \ni \xi \mapsto \xi(W) \in \N \cup \{\infty\}$ is measurable. If it is clear from context which space we refer to, we will write $\mathbf{N}$ and $\mathcal{N}$ instead of $\mathbf{N}_\W$ and $\mathcal{N}_\W$.

A \textbf{Poisson random measure} with intensity $\nu$ is a $(\mathbf{N},\mathcal{N})$-valued random element $\chi$ defined on some probability space $(\Omega,\mathcal{F},\p)$ such that
\begin{itemize}
	\item for all $W \in \mathcal{W}$ and all $k \in \N_0$, we have $\p(\chi(W)=k)=\exp(-\nu(W)) \frac{\nu(W)^k}{k!}$ (with the convention that $\chi(W)=\infty \ \p$-a.s. if $\nu(W)=\infty$);
	\item for $W_1, ..., W_n \in \mathcal{W}$ disjoint, the random variables $\chi(W_1),...,\chi(W_n)$ are mutually independent. 
\end{itemize}
Existence and uniqueness of such a measure is shown in \cite[Chapter~3]{LPenLectures}. We denote by $\p_\chi$ the law of $\chi$ in $(\mathbf{N},\mathcal{N})$ and we say that $\chi$ is a $(\W,\nu)$-\textbf{Poisson measure}.

In view of the $\sigma$-finiteness of $(\W,\nu)$ and using \cite[Corollary~6.5]{LPenLectures} we can and will assume throughout the paper that the Poisson measure $\chi$ is \textbf{proper}, i.e. that there exist independent random elements $W_1,W_2,... \in \W$ and an independent  $\N_0\cup\{\infty\}$-valued random variable $\kappa$ such that $\p$-a.s.
\[
\chi = \sum_{n=1}^\kappa \delta_{W_n},
\]
where $\delta_w$ is the Dirac mass at the point $w \in \W$. All our results only depend on the law of $\chi$, hence this assumption has no impact on them. In this context, we will often identify $\chi$ with its support, i.e. with the random collection of points $\{W_1,W_2,...\}$.\\

\noindent \underline{\textit{Poisson functionals.}}
For $p \geq 0$, denote by $L^p(\p_\chi)$ the set of random variables $F$ such that there is a measurable function $f:\mathbf{N} \rightarrow \R$ such that $F=f(\chi)$ $\p$-a.s. and, if $p>0$, such that $\E |F|^p < \infty$. We call $F$ a \textbf{Poisson functional} and $f$ a \textbf{representative} of $F$. All results that follow do not depend on the choice of the representative $f$ and hence, throughout the article, we will use the symbol $F$ indiscriminately to represent both $f$ and $F$.
\\

\noindent \underline{\textit{Add-one cost and Malliavin derivative.}}
For a Poisson functional $F \in L^0(\p_\chi)$ and $w \in \W$, define the \textbf{add-one cost operator} of $F$ as
\[
\MalD_w F := F(\chi + \delta_w) - F(\chi),
\]
and inductively set $\MalD^{(n)}_{w_1,...,w_n}F := \MalD_{w_n}\MalD^{(n-1)}_{w_1,...,w_{n-1}} F$ for $n \geq 1$ and $w_1,...,w_n \in \W$, where $\MalD^{(0)} F = F$ and $\MalD^{(1)} F=\MalD F$. It can be shown that $\MalD^{(n)}F$ is jointly measurable in all variables and symmetric in $w_1,...,w_n$ (cf. \cite[p. 5]{LastPoiss}). We denote by $\dom \MalD$ the set of all $F \in L^2(\p_\chi)$ such that
\begin{equation}\label{eqDomDCond}
\E \int_\W (\MalD_w F)^2 \nu(dx) < \infty.
\end{equation}
The restriction of the operator $\MalD$ to $\dom \MalD$ is called the \textbf{Malliavin derivative} of $F$ (see \cite[Theorem~3]{LastPoiss}). Note that for $F \in L^1(\p_\eta)$, the LHS of \eqref{eqDomDCond} is well-defined and \eqref{eqDomDCond} is sufficient for $F$ to be in $\dom \MalD$ (as follows from the $L^1(\p_\eta)$-Poincaré inequality as stated in \cite[Cor.~1]{LastPoiss}).

\noindent For $F,G \in L^0(p_\chi)$, we have the following formula for the add-one cost of a product:
\begin{equation}\label{eqProductFormula}
	\MalD(FG) = (\MalD F) G + F (\MalD G) + (\MalD F) (\MalD G).
\end{equation}

\noindent \underline{\textit{Chaotic decomposition.}}
For a function $g \in L^2(\W^n,\nu^{(n)})$, denote by $\Fint_n(g)$ the $n$th \textbf{Wiener-Itô integral} of $g$. Then for $F \in L^2(\p_\chi)$, we have the \textbf{Wiener-Itô chaos expansion}
\begin{equation}\label{eqWienerItoChaos}
	F = \sum_{n=0}^\infty \Fint(f_n),
\end{equation}
where $f_n(w_1,...,w_n) = \frac{1}{n!} \E \MalD^{(n)}_{w_1,...,w_n}F$ and the series converges in $L^2(\p_\chi)$ (cf. \cite[Theorem~2]{LastPoiss}).\\

\noindent \underline{\textit{Mecke formula.}}
Denote by $L^p(\mathbf{N} \times \W)$ the quotient set of all measurable functions $h : \mathbf{N} \times \W \rightarrow \R$ such that, if $p>0$, one has $\E \int_W |h(\chi,w)|^p \nu(dw) < \infty$.

Next, we introduce the so-called \textbf{Mecke formula} (cf. \cite[(7)]{LastPoiss}), which holds for $h \in L^1(\mathbf{N} \times \W)$ and for $h : \mathbf{N} \times \W \rightarrow [0,\infty)$ measurable:
\begin{equation} \label{eqMecke}
	\E \int_\W h(\chi,w) \chi(dw) = \E \int_\W h(\chi + \delta_w,w) \nu(dw).
\end{equation}

In particular, combined with the fact that $\chi$ is assumed to be proper, this implies that for a function $h \in L^1(\mathbf{N} \times \W)$, the integral
\begin{equation}\label{eqIntwithEta}
	\int_\W h(\chi - \delta_w,w) \chi(dw)
\end{equation}
is well-defined.\\

\noindent \underline{\textit{Skorohod integrals.}}
If $h \in L^2(\mathbf{N} \times \W)$, then for $\nu$-a.e. $w \in \W$, we have $h(.,w) \in L^2(\p_\chi)$ and thus we can write
\begin{equation}\label{eqhExpand}
h(\chi,w) = \sum_{n=0}^\infty \Fint_n(h_n(w,.)),
\end{equation}
with $h_n(w,w_1,...,w_n) = \frac{1}{n!} \E \MalD^{(n)}_{w_1,...,w_n} h(\chi,w)$ (cf. \cite[(42)]{LastPoiss}). We say that $h \in \dom \delta$ if
\[
\sum_{n=0}^\infty (n+1)! \int_{\W^{n+1}} \tilde{h}_n^2 d\nu^{n+1} < \infty,
\]
where $\tilde{h}_n$ is the symmetrisation of $h_n$ given by
\[
\tilde{h}_n(w_1,...,w_{n+1}) = \frac{1}{n+1} \sum_{k=1}^{n+1} h_n(w_k,w_1,...,w_{k-1},w_{k+1},...,w_{n+1}).
\]
For $h\in \dom \delta$ we define the \textbf{Skorohod integral} of $h$ by
\[
\delta(h) := \sum_{n=0}^\infty \Fint_{n+1}(h_n),
\]
which converges in $L^2(\p_\chi)$. Note that, by \cite[Theorem~5]{LastPoiss}, the following condition is sufficient for $h \in L^2(\mathbf{N} \times \W)$ to be in $\dom \delta$:
\begin{equation} \label{eqDomDelCond}
	\E \int_\W \int_\W \left(\MalD_{x} h(\chi,y)\right)^2 \nu(dx) \nu(dy) < \infty.
\end{equation}
If $h \in L^1(\mathbf{N} \times \W) \cap \dom \delta$, then by \cite[Theorem~6]{LastPoiss}, we have $\p$-a.s.
\begin{equation}\label{eqPathwise}
\delta(h) = \int_\W h(\chi-\delta_w,w) \chi(dw) - \int_\W h(\chi,w) \nu(dw),
\end{equation}
where the RHS is well-defined for any $h \in L^1(\mathbf{N} \times \W)$ by \eqref{eqIntwithEta}.\\

\noindent \underline{\textit{Extension to a marked space.}}
It will often be convenient to endow the space $\W$ with marks representing time. As we are only interested in the law of the Poisson functionals in question, we can always suppose that the $(\W,\nu)$-Poisson measure $\chi$ is the marginal of a $(\W \times [0,1],\nu \otimes ds)$-Poisson measure $\eta$. Indeed, $\eta(. \times [0,1])$ has the same law on $\mathbf{N}_\W$ as $\chi$. For a functional $F \in L^0(\p_\chi)$, define
\[
G(\eta) := F(\eta(.\times [0,1])).
\]
Then $G(\eta)$ has the same law under $\p_\eta$ as $F(\chi)$ under $\p_\chi$. Moreover, for any $(x,s) \in \W \times [0,1]$,
\[
\MalD_{(x,s)}G(\eta) = \MalD_xF(\eta(.\times [0,1])),
\]
which is equal in law to $\MalD_x F(\chi)$.\\

\noindent\underline{\textit{Predictability.}}
We call a measurable function $h:\mathbf{N}_{\W \times [0,1]} \times \W \times [0,1] \rightarrow \R$ \textbf{predictable} if for all $(y,s)\in \W \times [0,1]$ and all $\nu \in \mathbf{N}_{\W \times [0,1]}$
\begin{equation}\label{eqPred}
h(\nu,y,s) = h(\nu_{|\W \times [0,s)},y,s).
\end{equation}
This definition of predictability appears e.g. in \cite[(2.5)]{LPen2}, where it is argued that this version of predictability is comparable to predictability in the classical sense (as defined e.g. in \cite[Definition~I.5.2]{IkedaWatanabe}). It is also shown in \cite[Proposition~2.4]{LPen2} that if $h \in L^2(\mathbf{N}\times\W\times[0,1])$ satisfies \eqref{eqPred}, then $h \in \dom \delta$.\\

\noindent\underline{\textit{Conditional expectations and Clark-Ocône formula.}} Let $\eta$ be a $(\W \times [0,1],\nu \otimes ds)$-Poisson measure. Using that the measures $\eta_{|\W \times [0,s)}$ and $\eta_{|\W \times [s,0]}$ are independent, one can define a version of conditional expectation for any non-negative or integrable random variable $G \in L^0(\p_\eta)$ by
\begin{equation}\label{eqCondExp}
\E[G | \eta_{|\W \times [0,s)}] := \int G(\eta_{|\W \times [0,s)} + \xi) \Pi_{s}(d\xi),
\end{equation}
where $\Pi_{s}$ is the law of $\eta_{|\W \times [s,1]}$. If it is finite, the conditional expectation $\E[G | \eta_{|\W \times [0,s)}]$ is predictable. In particular, for $F \in L^2(\p_\eta)$ the quantity $\E [\MalD_{(x,s)}F | \eta_{|\W \times [0,s)} ]$ is well-defined, finite and predictable and the following Clark-Ocône type formula is shown in \cite[Theorem~2.1]{LPen2} (see also \cite{Wu,HP02}):
\begin{equation}\label{eqClarkOcone}
	F = \E F + \delta(\E [\MalD_{(x,s)}F | \eta_{|\W \times [0,s)} ]) \qquad \p_\chi-\text{a.s.}
\end{equation}
This formula will be essential in the proof of Corollary~\ref{corpPoinGen}.\\

\noindent\underline{\textit{Generic sets.}} Let $\mu \subset \R^d$ be a finite set. We say that $\mu$ is \textbf{generic} if all pairwise distances between points are distinct.
We say that a set $\mu \subset \R^d$ is \textbf{generic with respect to points} $x,y \in \R^d$ if $x,y \notin \mu$ and $\mu \cup \{x,y\}$ is generic. Note that for compact sets $H \subset \R^d$, any $(H,dx)$-Poisson measure $\chi$ can a.s. be identified with its support and this support is a.s. generic. To simplify the presentation, we will at times adopt the notation
\begin{equation} \label{setNot}
	F(\mu) := F(\xi_\mu), \qquad \text{where } \xi_\mu = \sum_{x \in \mu} \delta_x
\end{equation}
for a finite set $\mu \in \R^d$ and a measurable functional $F : \mathbf{N_{\R^d}} \rightarrow \R$. Similar notation will be used for $\MalD F(\mu)$, $\DD F(\mu)$ etc.

\subsection{Notation}
For $x \in \R^d$ and $r >0$, we write $B^d(x,r)$ to indicate the (open) ball of centre $x$ and radius $r$. For a measurable set $A \subset \R^d$, we denote by $|A|$ the Lebesgue measure of $A$, unless $A$ is finite, in which case $|A|$ denotes the number of elements in $A$. We use $\overbar{A}$ to denote the closure of $A$. Throughout this paper, $\kappa_d = |B^d(0,1)|$. We use the symbols $\wedge$ (resp. $\vee$) to denote a minimum (resp. maximum) of two elements. We shall use LHS and RHS to denote `left hand side' and `right hand side' and use $|x|$ to denote the Euclidean norm of $x \in \R^d$. The supremum norm of a function $f:\R \rightarrow \R$ is denoted by $\norm{f}_\infty$. By $\stackrel{d}{=}$ and $\stackrel{d}{\longrightarrow}$ we mean equality and convergence in distribution respectively. We use the symbol $\simeq$ (resp. $\lesssim$) if there is equality (resp. inequality) up to multiplication by a positive constant.
\section{Second order $p$-Poincaré inequalities in Wasserstein and Kolmogorov distances}\label{secWasserKolg}

In this section we state our new bounds on the distance between the distribution of a Poisson functional and the Normal law. These bounds are called second order $p$-Poincaré inequalities, following a nomenclature coined in \cite{Chatterjee2009}, where bounds of this type were given for the first time in a Gaussian context. We make use of the well-established Malliavin-Stein method, which was pioneered in \cite{PN09} in the Wiener case, used for the first time in the Poisson case in \cite{PeccTaqq} and subsequently extended and developed in a wide range of articles -- see the references given in \cite{LPS14}, the survey \cite{APY18}, the monograph \cite{PeccReitz} and the website \cite{MallStein}. Related bounds in the Kolmogorov distance have been studied in various places \cite{TE14,S16,LPS14,LachPeccYang}.

Recall that for an integrable random variable $F$ and a standard Gaussian $N$, the Wasserstein distance between the distributions of $F$ and $N$ is given by
\begin{equation}\label{eqdW}
d_W(F,N) = \sup_{h \in \mathcal{H}} |\E h(F) - \E h(N)|
\end{equation}
where $\mathcal{H}$ is the set of Lipschitz-continuous functions $h:\R \rightarrow \R $ with Lipschitz constant $\norm{h}_L \leq 1$.
On the other hand, the Kolmogorov distance between the distributions of $F$ and $N$ is defined as
\begin{equation}\label{eqdK}
d_K(F,N) = \sup_{z \in \R} |\p(F \leq z) - \p(N \leq z)|.
\end{equation}
See e.g. \cite[Appendix~C]{PeccNourd}, and the references therein, for a discussion of the basic properties of $d_W$ and $d_K$.

\begin{rem}\label{remNotatWK}
	For the rest of this section, we fix a $\sigma$-finite measure space $(\X, \mathcal{X},\lambda)$. Before we state our main theorems, we introduce some simplified notation to improve legibility of the following results. Write $\Y := \X \times [0,1]$ and $\bar{\lambda}= \lambda \otimes dt$. We introduce a total order on $\Y$ by saying that $x<y$ if $x=(z,s)$, $y=(w,u)$ and $s<u$. In the following, $\eta$ will be a $(\Y,\bar{\lambda})$-Poisson measure and $\chi$ will be a $(\X,\lambda)$-Poisson measure. We will write $\eta_x$ for $\eta_{|\X \times [0,s)}$. Integrals with respect to $\bar{\lambda}$ are taken over $\Y$ and integrals with respect to $\lambda$ are taken over $\X$. 
\end{rem}
The next statement contains the general abstract bounds on which our analysis will rely.

\begin{thm} \label{thmWasserKolg1}
	Let $F \in L^2(\p_\eta) \cap \dom \MalD$ such that $\E F = 0$ and $\E F^2 = 1$. Then for any $q \in [1,2]$
	\begin{align}
		&d_W(F,N) \leq \sqrt{\frac{2}{\pi}}\, \E \left| 1- \yint \MalD_yF\, \E[\MalD_yF | \eta_y] \bar{\lambda}(dy) \right| + 2\, \E \yint \left| \E[\MalD_yF | \eta_y] \right| \cdot |\MalD_yF|^q \bar{\lambda}(dy)\label{eqW2}
		\intertext{and}
		&\begin{multlined}[0.9\displaywidth]
			d_K(F,N) \leq \E \left| 1- \yint \MalD_yF\, \E[\MalD_yF | \eta_y] \bar{\lambda}(dy) \right|\\ + \sup_{z \in \R} \E \yint \left|\E\big[\MalD_y F \big|\eta_y\big]\right| \MalD_y F \cdot
			\MalD_y(Ff_z(F) + \ind{F>z}) \bar{\lambda}(dy).
		\end{multlined} \label{eqK2}
	\end{align}
\end{thm}

As a next step, we derive the upper bounds we use in applications. Define the following quantities:

\begin{align*}
	&\beta_1 := \frac{2^{2/p}\sqrt{2}}{\sqrt{\pi}} \sigma^{-2} \left( \yint \left(\yint \E\left[\E\big[|\MalD_yF|\big|\eta_y\big]^{2p}\right]^{\frac{1}{2p}} \cdot \E\left[\E\big[|\DD_{x,y} F |\big| \eta_{x \vee y} \big]^{2p}\right]^{\frac{1}{2p}} \bar{\lambda}(dy) \right)^p \bar{\lambda}(dx)\right)^{1/p} \\
	&\beta_2 := \frac{2^{2/p}}{\sqrt{2\pi}} \sigma^{-2} \left( \yint \left(\yint \ind{x<y} \E\left[\E\big[|\DD_{x,y} F|\big|\eta_y\big]^{2p} \right]^{1/p} \bar{\lambda}(dy) \right)^p \bar{\lambda}(dx)\right)^{1/p} \\
	&\beta_3 := 2\sigma^{-(q+1)} \E \yint \big| \E \left[\left. \MalD_yF \right|\eta_y\right]\big|^{q+1} \Bar{\lambda}(dy),\\
	&\beta_4 := 2^{3-q} \sigma^{-(q+1)} \yint\yint\ind{y \leq x}\E\left[\E[\MalD_yF | \eta_y]^2\right]^{1/2} \cdot \E \left[ \big|\E[\DD_{x,y} F|\eta_x]\big|^{2q}\right]^{1/2} \bar{\lambda}(dx) \Bar{\lambda}(dy)
\end{align*}

The following statement is our first bound on Wasserstein distances, expressed in terms of moments of the first and second order add-one costs conditional on past behaviour.

\begin{thm}\label{thmWasserCond}
	Let $\eta$ be a $(\Y, \bar{\lambda})$-Poisson-measure and let $F \in L^2(\p_\eta) \cap \dom \MalD$. Define $\sigma:= \sqrt{\var(F)}$ and  $\hat{F}:= (F-\E F)\sigma^{-1}$. Let $p,q\in (1,2]$. Then
	\begin{equation}\label{eqWasserCond}
		d_W(\hat{F},N) \leq \beta_1 + \beta_2 + \beta_3 + \beta_4.
	\end{equation}
\end{thm}
\noindent The proof can be found in Appendix~\ref{secProofWK}.

Now define
\begin{align*}
	&\gamma_1 := \frac{2^{2/p}\sqrt{2}}{\sqrt{\pi}} \sigma^{-2} \left( \int_\X \left(\int_\X \E\left[|\MalD_yF|^{2p}\right]^{\frac{1}{2p}} \cdot \E\left[|\DD_{x,y} F|^{2p}\right]^{\frac{1}{2p}} \lambda(dy) \right)^p \lambda(dx)\right)^{1/p} \\
	&\gamma_2 := \frac{2^{2/p}}{\sqrt{2\pi}} \sigma^{-2} \left( \int_\X \left(\int_\X \E\left[|\DD_{x,y} F|^{2p} \right]^{1/p} \lambda(dy) \right)^p \lambda(dx)\right)^{1/p} \\
	&\gamma_3 := 2\sigma^{-(q+1)}  \int_\X \E\left|\MalD_yF \right|^{q+1} \lambda(dy)
\end{align*}
and
\begin{align*}
	&\gamma_4 := \sigma^{-2} \left(4  \int_\X \E\left[|\MalD_yF|^{2p}\right] \lambda(dy)\right)^{1/p} \\
	&\gamma_5 := \sigma^{-2} \left(4p \int_\X \int_\X \E\left[|\DD_{x,y} F|^{2p}\right] \lambda(dy) \lambda(dx)\right)^{1/p} \\
	&\gamma_6 := \sigma^{-2} \left(2^{2+p}p \int_\X \int_\X \E\left[|\DD_{x,y} F|^{2p}\right]^{1/2} \cdot \E\left[|\MalD_{x} F|^{2p}\right]^{1/2} \lambda(dy) \lambda(dx)\right)^{1/p} \\
	&\gamma_7 := \sigma^{-2} \left(8p \int_\X \int_\X \left(\E|\DD_{x,y} F|^{2p}\right)^{\frac{1}{2p}} \cdot  \left(\E|\MalD_xF|^{2p}\right)^{\frac{1}{2p}} \cdot \left(\E\left|\MalD_yF\right|^{2p}\right)^{1-1/p} \lambda(dy)\lambda(dx)\right)^{1/p}.
\end{align*}
Note that the quantities $\beta_1,...,\beta_4$ and $\gamma_1,...,\gamma_7$ only contain expressions related to $\MalD F$ and $\DD F$.

The next statement contains our main estimates on Wasserstein and Kolmogorov distances, given in terms of moments of first and second order add-one costs (without conditioning).
\begin{thm}\label{thmWasserNon-Cond.}
	Let $\chi$ be a $(\X,\lambda)$-Poisson measure and let $F \in L^2(\p_\chi) \cap \dom \MalD$. Define $\sigma:= \sqrt{\var F}$ and $\hat{F}:= (F- \E F)\sigma^{-1}$. Then
	\begin{equation}\label{eqWasserNCond}
		d_W(\hat{F},N) \leq \gamma_1 + \gamma_2 + \gamma_3
	\end{equation}
	and
	\begin{equation}\label{eqKolg}
		d_K(\hat{F},N) \leq \sqrt{\tfrac{\pi}{2}} \gamma_1 + \sqrt{\tfrac{\pi}{2}}\gamma_2 + \gamma_4 + \gamma_5 + \gamma_6 + \gamma_7.
	\end{equation}
\end{thm}


\begin{rem}\label{remGamma7}
	Using Hölder's inequality, one can replace the term $\gamma_7$ by the slightly larger but simpler bound
	\[
	\sigma^{-2} \left(\int_\X\int_\X \left(\E|\DD_{x,y} F|^{2p}\right)^{\frac{1}{2p}} \cdot  \left(\E|\MalD_xF|^{2p}\right)^{1-\frac{1}{2p}} \lambda(dx) \lambda(dy) \right)^{1/p}.
	\]
	We will use this bound in the proof of Theorem~\ref{thmRST} in the context of the Radial Spanning Tree.
\end{rem}

\begin{rem}[Discussion of literature]
	Our results in this section are a substantial extension of \cite{LPS14}. The bounds given in \cite[Theorems~1.1 and 1.2]{LPS14} contain moments of first and second order add-one costs with exponent $4$ (or even $4+\epsilon$, see \cite[Proposition~1.4]{LPS14}). While this is a very powerful tool for showing asymptotically Gaussian behaviour, a finite $4$th moment is too strong a condition for some applications, most notably for the ONNG discussed in Section~\ref{secONNG}. Our Theorem~\ref{thmWasserNon-Cond.} reduces this condition to finite $2p$ moments, where $p\in (1,2]$, while retaining similar bounds in the case $p=2$. In particular, \cite[Theorem~6.1 and Proposition~1.4]{LPS14} follow from our Theorem~\ref{thmWasserNon-Cond.}. (See also \cite{Tri19} for qualitative results requiring bounds on moments of order $2p$ under weak stabilisation assumptions).

	The proofs of Theorems~\ref{thmWasserKolg1}, \ref{thmWasserCond} and \ref{thmWasserNon-Cond.} follow in spirit the ideas from \cite{LPS14} and \cite[Theorem~1.12]{LachPeccYang} (for the Kolmogorov distance). However, we work on a space $\X \times [0,1]$ extended by a time component and systematically replace the operator $L^{-1}$ by the conditional expectation $\E[\MalD_{(x,t)} . |\eta_{|\X \times [0,t)}]$ (see \cite{PT13} for a similar approach for Poisson measures on the real line). Moreover, we apply the inequalities established in Section~\ref{secSkorEstimates} to achieve the improvement in the exponent. For the Wasserstein distance, we also use an improvement due to \cite{BOPT} to obtain the terms $\beta_3,\beta_4,\gamma_3$. For the Kolmogorov distance, our bound in Theorem~\ref{thmWasserNon-Cond.} makes use of an improvement implemented in \cite{LachPeccYang}, but we remove a strong condition on $F$. The resulting bound is close in spirit to the one given in \cite[Theorem~1.2]{LPS14}, but with an improvement from $4$th moments to $2p$th moments. Moreover, our bound does not need the term corresponding to \cite[term $\gamma_3$, p. 670]{LPS14} and replaces the term corresponding to \cite[term $\gamma_4$, p. 671]{LPS14} by a term depending only on the add-one cost operators of $F$ instead of $\E F^4$.
	
	In Theorem~\ref{thmWasserCond}, we do not take moments of the first and second order add-one costs of our functionals, but of their expectation conditional on `past behaviour'. A bound of this type is new and the distinction is crucial to solve the critical case of the ONNG (see Theorem~\ref{thmONNG}). As of now, such a bound is only available in the Wasserstein distance.
\end{rem}

\section{Ancillary results: new estimates for Skorohod integrals}\label{secSkorEstimates}

\subsection{A version of Itô formula}\label{secIto}

We start this section by giving a version of Itô formula for Poisson integrals with anticipative integrands. This is a crucial ingredient for the proof of the new estimates given in Theorem~\ref{thmGeneralpPoin}.
In the following, we will take $\eta$ to be a $(\X \times [0,1], \mathcal{X} \otimes \mathcal{B}([0,1]),\lambda(dx)\otimes ds)$-Poisson measure, where $(\X,\mathcal{X},\lambda)$ is a $\sigma$-finite measure space.

\begin{thm}[Itô formula for non-adapted integrands]\label{thmItoFormula}
	Let $h \in L^1(\mathbf{N} \times \X \times [0,1])$ be bounded and let $X_0 \in \R$. For $t\in[0,1]$, define
	\begin{equation}\label{eqItoProcess}
	X_t(\eta) := X_0 + \int_{\X \times [0,t]} h(\eta-\delta_{(y,s)},y,s) \eta(dy,ds)
	- \int_{\X}\int_0^t  h(\eta,y,s) \lambda(dy)ds.
	\end{equation}
	Then the process $(X_t)_{t \in [0,1]}$ is well-defined and $\p$-a.s. càdlàg. Let $\phi \in \mathcal{C}^1(\R)$. Then, $\forall\, t \in [0,1]$,
	\begin{multline}
		\phi(X_t) = \phi(X_0) + \int_{\X \times [0,t]} \left(\phi\left(X_{s-}+h(\eta-\delta_{(y,s)},y,s)\right)-\phi(X_{s-})\right) \eta(dy,ds)\\ - \int_\X\int_0^t \phi'(X_s)h(\eta,y,s) \lambda(dy)ds \qquad \p\text{-a.s.} \label{eqIto}
	\end{multline}
	and the quantities in \eqref{eqIto} are well-defined.
\end{thm}

In the next three items we compare our version of Itô formula with the classical one given in \cite[Theorem~II.5.1]{IkedaWatanabe}.
\begin{enumerate}
	\item The main difference between \eqref{eqIto} and \cite[Thm.~II.5.1]{IkedaWatanabe} consists in the fact that we do not assume the integrand $h$ to be predictable. There exist Itô formulae for anticipative integrands in various settings, e.g. in the Wiener case in \cite{AlosNualart1996} and \cite{Nualart1988}) and for pure jump and general Lévy processes in \cite{Nunno2005} and \cite{Alos2008} respectively. To the best of our knowledge, our setting of a general Poisson point process is new.	 
	 \item Assume $h$ to be predictable in the sense of \eqref{eqPred}. It follows that $h(\eta-\delta_{(y,s)},y,s) = h(\eta,y,s)$ for all $(y,s) \in \X \times [0,1]$. For $ \phi \in \mathcal{C}^2(\R)$, formula \eqref{eqIto} is now roughly equivalent to the Itô formula given in \cite[Theorem~II.5.1]{IkedaWatanabe} in the special case where the semi-martingale in the statement of \cite[Theorem~II.5.1]{IkedaWatanabe} has the following properties:
	 \begin{itemize}
	 	\item the point process in question is a Poisson point process;
	 	\item the only non-zero part is the one with respect to the compensated Poisson measure;
	 	\item the integrand $h$ is both in $L^2$ and in $L^1$.
	 \end{itemize}
	 
	 \item Our setting is thus both more general ($\phi \in \mathcal{C}^1(\R)$ and $h$ anticipative) and more restrictive ($h \in L^1 \cap L^2$ instead of $h \in L^{2,\text{loc}}$ and the Gaussian, finite variation and non-compensated Poisson terms are zero) than the one given by Ikeda and Watanabe. The proof of our result relies however on the same ideas as the proof of \cite[Theorem~II.5.1]{IkedaWatanabe}.
\end{enumerate}

\subsection{Moment Inequalities}
In this section, we present a number of functional inequalities that are of independent interest and also crucial to the improved bounds on Wasserstein and Kolmogorov distances presented in earlier sections.

To the best of our knowledge, Theorem~\ref{thmGeneralpPoin} is the first bound of its kind on functionals of general Poisson-Skorohod integrals. Partial results are known in the particular case where $h$ is predictable, see Corollary~\ref{corpPoinGen} and the discussion thereafter. In particular, Theorem~\ref{thmGeneralpPoin} below contains the first general estimate in terms of add-one costs for $p$-moments of the Skorohod integral, where $p \in [1,2]$, the cases $p=1$ and $p=2$ being the only ones known. See also \cite{LMS22}.

In the special case $\phi(x)=x^2$, the theorem below follows immediately from the isometry relation reported in formula \eqref{eqIsomDelta} of Appendix~\ref{secBackground}.

\begin{thm}\label{thmGeneralpPoin}
	Let $h\in L^2(\mathbf{N}\times\X\times[0,1])$ satisfy \eqref{eqDomDelCond}. Let $\phi:\R \rightarrow \R$ be a differentiable function with $(p-1)$-Hölder continuous derivative, for some $p \in (1,2]$ and assume that $\phi(0)=0$. Then
	\begin{align}
		\left|\E \phi(\delta(h))\right| &\leq \frac{c_\phi}{p} \E \int_\X\lambda(dy) \int_0^1ds\ |h(\eta,y,s)|^p  \notag\\
		&+ c_\phi \E\int_\X\lambda(dy) \int_0^1ds \int_\X \lambda(dx) \int_0^s dt\ \left|\MalD_{(y,s)}h(\eta,x,t)\right| \cdot \left|\MalD_{(x,t)}h(\eta,y,s)\right|^{p-1} \notag\\
		&+ 2c_\phi \E\int_\X\lambda(dy) \int_0^1ds \int_\X \lambda(dx) \int_0^s dt\ \left|\MalD_{(y,s)}h(\eta,x,t)\right| \cdot \left|h(\eta,y,s)\right|^{p-1}, \label{eqGeneralPPoin}
	\end{align}
	where $c_\phi$ is the Hölder constant of $\phi'$. In particular, this inequality holds with $\phi(x)=|x|^p$ and $c_\phi = p2^{2-p}$, for $p\in[1,2]$.
\end{thm}

Our first corollary is a version of the above inequality for predictable functions $h$ and contains a generalisation of the classical Poincaré inequality.

\begin{cor} \label{corpPoinGen}
	Let $h \in L^2(\mathbf{N}\times\X\times[0,1])$ be predictable in the sense of \eqref{eqPred}. Let $\phi:\R \rightarrow \R$ be a differentiable function with $(p-1)$-Hölder continuous derivative, for some $p \in (1,2]$. Assume $\phi(0)=0$. Then
	\begin{equation} \label{eqPhiDelta}
	|\E \phi(\delta(h))| \leq \frac{c_\phi}{p} \E \int_\X \lambda(dy) \int_0^1 ds |h(\eta,y,s)|^p.
	\end{equation}
	Moreover, for $F\in L^2(\p_\eta)$ and $p \in [1,2]$,
	\begin{equation}\label{eqPPoin2}
		\E |F|^p - |\E F|^p \leq 2^{2-p} \E \int_\X \int_0^1 \left|\E [\MalD_{(x,t)}F | \eta_{\X \times [0,t)} ]\right|^p \lambda(dx)dt.
	\end{equation}
\end{cor}
\begin{rem}\label{remPPoin}
		\begin{enumerate}
		\item We can extend inequality \eqref{eqPPoin2} to $F \in L^1(\p_\eta)$ at the cost of introducing an additional absolute value on the RHS:
		\begin{equation}\label{eqPPoin1}
			\E |F|^p - |\E F|^p \leq 2^{2-p} \E \int_\X \int_0^1 \E [|\MalD_{(x,t)}F| | \eta_{\X \times [0,t)} ]^p \lambda(dx)dt.
		\end{equation}
		This can be seen easily by approximating $F$ by $F_n := (F \wedge n) \vee (-n)$ and using monotone and dominated convergence.
		\item When removing the conditional expectation in \eqref{eqPPoin2} using Jensen's inequality, the inequality can be extended to functionals $G \in L^1(\p_\chi)$, where $\chi$ is a $(\X,\lambda)$-Poisson measure without time component. Indeed, as discussed in Section~\ref{secFrame}, the marginal $\eta(.\times [0,1])$ has the same law as $\chi$, which means one can see $G$ as a functional on $\mathbf{N}_{\X \times [0,1]}$. We have then
		\begin{equation}\label{eqPPoinNoTime}
			\E |G|^p - |\E G|^p \leq 2^{2-p} \E \int_\X \left|\MalD_{x}G\right|^p \lambda(dx).
		\end{equation}
	\end{enumerate}
\end{rem}
\begin{rem}[Literature review]
	The proof of Theorem~\ref{thmGeneralpPoin} relies on a combination of the Clark-Ocône type representation result~\eqref{eqClarkOcone} and the version of Itô formula given in Theorem~\ref{thmItoFormula}. This method of combining a Clark-Ocône result with Itô formulae to deduce functional inequalities has been applied before in various settings, e.g. in \cite{Wu} and \cite{Chafai}, where it was used to deduce a modified log-Sobolev inequality and $\Phi$-Sobolev inequalities respectively.
	
	Inequalities \eqref{eqPhiDelta}, \eqref{eqPPoin2} and \eqref{eqPPoinNoTime} can be seen as part of a larger family of functional inequalities on the Poisson space. The first to mention is the classical Poincaré inequality, given e.g. in \cite[Theorem~10]{LastPoiss} (see also \cite[Cor.~4.4]{HPA95} for a very early appearance of this inequality). Our inequality extends the classical one, which is \eqref{eqPPoinNoTime} in the case $p=2$. Another well-known inequality is the modified log-Sobolev inequality shown in \cite{Wu} (see also \cite{AneLedoux}). It is extended in \cite[(5.10)]{Chafai} to the so-called $\Phi$-Sobolev inequalities, which in the case $\Phi(x)=x^p$, imply \eqref{eqPPoinNoTime} when $F\geq 0$ or $\E F=0$. Similarly, the Beckner type $p$ inequalities discussed in \cite[Section 4.6]{AdamPola} imply \eqref{eqPPoinNoTime} when $F\geq 0$ or $\E F=0$, albeit with a worse constant. \cite[Theorem 3.3.2]{ZhuThesis} gives a version of \eqref{eqPhiDelta} for $p$-norms in martingale type $p$ Banach spaces. Although we did not check the details, it is reasonable to assume that one can deduce \eqref{eqPhiDelta} in the case $\phi(x)=|x|^p$ from such a result when applied to $\R^d$.
	
\end{rem}
\begin{rem}[Comparison with the Gaussian case and extensions when $p\geq 2$] \phantom{luckluckluckluckluck}
	\begin{enumerate}
		\item Inequality \eqref{eqPPoinNoTime} for $p<2$ does not hold for functionals of Gaussian random measures, as can be seen by taking $G=W_t$ and letting $t \rightarrow 0$, with $W$ a standard Brownian motion. This is in contrast with the classical Poincaré inequality ($p=2$) which holds in both Gaussian and Poisson settings.
		\item Inequality \eqref{eqPPoinNoTime} (and hence also \eqref{eqPPoin2}) is false in general for $p>2$. Indeed, consider $(\X,\lambda) = (\R^d,dx)$ and $G = \chi(A) - \lambda(A)$ for some measurable $A \subset \R^d$. Then $\E G =0$, $\E G^2 = \lambda(A)$ and $\MalD_x G = \1_A(x)$. On the LHS we have $\E |G|^p \geq (\E G^2)^{p/2} = \lambda(A)^{p/2}$ by Jensen's inequality and on the RHS
		\[
		\int_\X \E |\MalD_x G|^p \lambda(dx) = \lambda(A).
		\]
		However, since $p>2$, we have $\lambda(A)^{p/2} \gg \lambda(A)$ for $\lambda(A)$ large enough. Hence the inequality fails for any multiplying constant.
		\item Moment estimates for $p\geq 2$ are given in \cite[Theorem~4.1]{GusSamTh} and \cite[Proposition~4.20]{AdamPola}. The RHSs of these inequalities involve related, but different quantities.
	\end{enumerate}
\end{rem}

The versatility of Theorem~\ref{thmGeneralpPoin} can be appreciated when considering the following corollary, which will be crucial in finding a bound on the Kolmogorov distance.
\begin{cor}\label{corKolgIneq}
	Let $h \in L^1(\mathbf{N} \times \X \times [0,1])$ and $G \in L^0(\p_\eta)$ bounded by a constant $c_G>0$. Then for any $p \in [1,2]$,
	\begin{align}
		&\left|\E \int_\X \int_0^1 h(\eta,x,s) \MalD_{(x,s)}G \lambda(dx)ds\right|\notag\\
		&\leq c_G \left( 2^{2-p} \E \int_\X\int_0^1 |h(\eta,x,s)|^p \lambda(dx)ds \right. \notag\\
		&+ p2^{2-p} \E \int_\X\int_0^1 \int_\X\int_0^1 |\MalD_{(x,s)}h(\eta,y,u)|^p \lambda(dx)ds \lambda(dy)du \notag\\
		&\left.+ p2^{3-p} \int_\X\int_0^1 \int_\X\int_0^1 \ind{s<u} \E [|\MalD_{(y,u)} h(\eta,x,s)|^p]^{1/p} \E [|h(\eta,x,s)|^p]^{1-1/p} \lambda(dx)ds\lambda(dy)du \right)^{1/p}. \label{eqCorKolg}
	\end{align}
\end{cor}
\begin{rem}
	Provided that we upper bound the indicator in the third term on the RHS of \eqref{eqCorKolg} by 1, this inequality can be extended to a space $\X$ without time component.
\end{rem}

\section{Applications}\label{secAppl}
In this section, we look at four types of graphs built on Poisson measures and assess the speeds of convergence to Normality of $\alpha$-power-weighted edge-lengths such as \eqref{eqIntroF}. As was found in previous work \cite{LPS14,ST14rst}, we find for certain ranges of exponents $\alpha$ that the speed is given by $t^{-d/2}$, which corresponds to the order of the square root of the variance. This is the presumably optimal speed corresponding to the one in the classical Berry-Esseen theorem (see e.g. \cite[Theorem~4, p.111]{Petrov}). Beyond a certain threshold, we find a slower speed of convergence that depends on $\alpha$. Generally speaking, a $2p$th moment integrability of the first and second order add-one costs of the functionals leads to a speed of convergence of $t^{-d(1-1/p)}$. Whether this speed is optimal or not is an open question.

\subsection{Online Nearest Neighbour Graph} \label{secONNG}
Let $\mu \subset \R^d \times [0,1]$ be a finite set such that the projection of $\mu$ onto $\R^d$ is generic and does not contain any multiplicities and the projections onto $[0,1]$ are distinct. The ONNG on $\mu$ is an (undirected) graph in $\R^d$ constructed as follows:
\begin{itemize}
	\item Vertices are given by $\{x \in \R^d : (x,s) \in \mu\}$
	\item Let $(x,s) \in \mu$. If $\mu \cap \left(\R^d \times [0,s)\right)$ is non-empty, then the online nearest neighbour of $(x,s)$ is given by the point $(z,u) \in \mu \cap \left(\R^d \times [0,s)\right)$ which minimises $|x-z|$. In this case there is an edge from $x$ to $z$ and we denote this event by $\{(x,s) \rightarrow (z,u) \text{ in } \mu\}$.
\end{itemize}
For a point $(x,s) \in \mu$, the coordinate $s$ can be seen as the arrival time of the point $x \in \R^d$, or its mark. Any point $(x,s) \in \mu$ has exactly one online nearest neighbour, except for the point in $\mu$ whose mark is minimal, which has none. Even though the graph is undirected, we think of edges going from a point to its nearest neighbour, as this simplifies the discussion.

For $(x,s) \in \mu$, let
\[ e(x,s,\mu) := %
\begin{cases*}
	\inf \{|x-z| : (z,u) \in \mu \cap \left(\R^d \times [0,s)\right)\}, & if  $\mu \cap \left(\R^d \times [0,s)\right) \neq \emptyset $\\
	0, & otherwise.
\end{cases*}
\]
This is the length of the edge from $x$ to its online nearest neighbour if there is one, and zero otherwise. Note that one can find a unique online nearest neighbour in $\mu$ for any point $(x,s) \in \R^d \times [0,1]$ such that the time coordinate $s$ and the position $x$ do not occur in $\mu$. For convenience, we shall extend the above definitions to any such $(x,s) \in \R^d \times [0,1]$ and tacitly adopt the corresponding notation.

We will be studying the sums of power-weighted edge-lengths defined as follows: for $\alpha > 0$, let
\[
\Fa(\mu) := \sum_{(x,s) \in \mu} e(x,s,\mu)^\alpha.
\]
Note that here we make use of the convention explained in Section~\ref{secFrame} to identify a set of points $\mu$ with the point measure whose support is given by $\mu$.

Let $\eta$ be a Poisson measure on $\R^d \times [0,1]$ with Lebesgue intensity. Let $H \subset \R^d$ be a convex body. For $t \geq 1$, define
\[
\fat := \Fa(\eta_{|tH \times [0,1]}).
\]

\begin{figure}
	\centering
	\includegraphics[width=.5\textwidth]{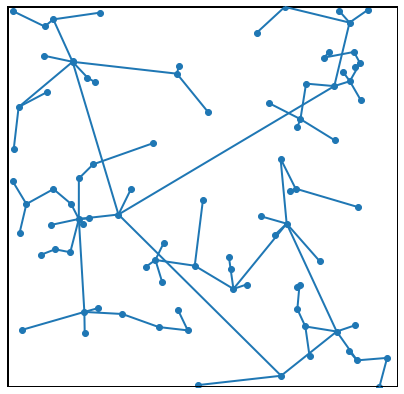}
	\caption{Realisation of the Online Nearest Neighbour graph}
\end{figure}

The ONNG is a relatively simple model for networks growing in time. Already mentioned in \cite{Steele89}, the Online Nearest Neighbour graph came to general attention in \cite{BBBCR07}, where is was presented as a simplified version of the FKP model developed in \cite{FKP02}, used to model the internet graph. The name of the graph was coined in \cite{Pen05}, where the martingale method is used to show central limit theorems for stabilising random systems satisfying a 4th moment condition. In particular, it is shown in the case of the ONNG:
\begin{thm}[{\cite[Theorem 3.6]{Pen05}}]\label{thmPenrose}
	For $0 \leq \alpha < \frac{d}{4}$, there is a constant $\sigma_{\alpha,d} > 0$ such that as $t \rightarrow \infty$,
	\begin{equation}\label{eqWade}
	 t^{-d} \var\left(\fat\right) \longrightarrow \sigma_{\alpha,d} \quad \text{and} \quad \frac{\fat - \E \fat}{t^{d/2}} \stackrel{d}{\longrightarrow} \mathcal{N}(0,\sigma_{\alpha,d}).
	\end{equation}
\end{thm}
A quantitative counterpart to this result is shown in \cite{LachPeccYang}. Our Theorem~\ref{thmONNG} provides a speed of convergence that is faster than the one given in \cite{LachPeccYang}.

In \cite{Pen05,PenWade08,Wade2009}, results similar to Theorem~\ref{thmPenrose} were conjectured to hold for $\alpha \in \left[\frac{d}{4},\frac{d}{2}\right]$. In particular, part of the Conjectures~2.1 and 2.2. in \cite{Wade2009} states
\begin{conj}[\cite{Wade2009}]
	For $\alpha \in \left[\left.\frac{d}{4}, \frac{d}{2}\right)\right.$, there is a constant $\sigma_{\alpha,d}>0$ such that \eqref{eqWade} holds.
	
	For $\alpha=\frac{d}{2}$, there is a constant $\sigma_d > 0$ such that
	\[
	\log(t)^{-1}t^{-d} \var\left(\fat\right) \longrightarrow \sigma_{d} \quad  \text{and} \quad \frac{\fat - \E \fat}{\log(t)^{1/2}t^{d/2}} \stackrel{d}{\longrightarrow} \mathcal{N}(0,\sigma_{d}).
	\]
\end{conj}
Our forthcoming Theorem~\ref{thmONNG} confirms this conjecture by giving quantitative central limit theorems for $\alpha \in \left(\left.0,\frac{d}{2}\right]\right.$ and upper and lower bounds for the variances that match the conjectured orders. Upper bounds of the conjectured orders were already given in \cite[Theorem~2.1]{Wade2009} for the variances involved. They are shown for an ONNG built on $n$ uniformly distributed random variables and the corresponding result for the Poisson version follows by Poissonisation. For the sake of completeness, we will give purely Poissonian proofs of the upper bounds, following however a similar strategy as in \cite{Wade2009}. A law of large numbers is shown in \cite{Wade07} and the case $\alpha>\frac{d}{2}$ is discussed in \cite{PenWade08} (especially for $d=1$) and in \cite{Wade2009}, where is is shown that a limit exists in this case, but is non-Gaussian for $\alpha>d$. For more related results we refer to the survey \cite{PW09} (for results up to 2010) and to the paper \cite{LM21}.

\begin{thm}\label{thmONNG}
	For $0 < \alpha < \frac{d}{2}$, and for every $1<p<\frac{d}{2\alpha}$ such that $p \leq 2$, there is a constant $c_1>0$ such that for all $t \geq 1$ large enough
	\[
	\max\left\{d_W\left(\frac{\fat - \E \fat}{\sqrt{\var\big(\fat\big)}}, N\right),d_K\left(\frac{\fat - \E \fat}{\sqrt{\var\big(\fat\big)}}, N\right)\right\} \leq c_1 t^{-d\left(1-\frac{1}{p}\right)},
	\]
	where $N$ denotes a standard normal random variable.
	Moreover, there are constants $c_2,C_2>0$ such that for all $t \geq 1$ large enough
	\begin{equation}\label{eqONNGVarNC}
	c_2 t^d < \var(\fat) < C_2 t^d.
	\end{equation}
	For $\alpha = \frac{d}{2}$, there is a constant $c_3>0$ such that for all $t \geq 1$ large enough
	\[
	d_W\left(\frac{\fad - \E \fad}{\sqrt{\var\big(\fad\big)}}, N\right) \leq c_3 \log(t)^{-1}.
	\]
	Moreover, there are constants $c_4,C_4>0$ such that for all $t \geq 1$ large enough
	\begin{equation}\label{eqONNGVarC}
	c_4 t^d\log(t^d) < \var(\fad) < C_4 t^d\log(t^d).
	\end{equation}
	The constants $c_1,c_2, C_2,c_3,c_4, C_4$ may depend on $H,\ \alpha,\ d$ and $p$.
\end{thm}
Note that, in the special case $0<\alpha<\frac{d}{4}$, we find a speed of convergence of $t^{-d/2}$, which corresponds to the square root of the order of the variance.

\subsection{Gilbert Graph} \label{secGil}

For a finite set $\mu \subset \R^d$ and a real number $\epsilon > 0$, the Gilbert graph $G(\mu,\epsilon)$ has vertex set $\mu$ and an edge between $x,y \in \mu$, $x \neq y$ if and only if $|x-y|<\epsilon$. To construct our functional of interest, we consider
\begin{itemize}
	\item $W \subset \R^d$ a convex body;
	\item for every $t>0$, we take $\eta^t$ a $(W,t\, dx)$-Poisson measure;
	\item $(\epsilon_t)_{t>0}$ a sequence of positive real numbers s.t. $\epsilon_t \rightarrow 0$ as $t \rightarrow \infty$.
\end{itemize}
Then for $\alpha \in \R$, define
\[
\lat := \sum_{e \in G(\eta^t,\epsilon_t)} |e|^{\alpha} = \frac{1}{2} \sum_{x,y \in \eta^t, x\neq y} \ind{|x-y| \leq \epsilon_t} |x-y|^{\alpha},
\]
where $e$ denote the edges of the graph and $|e|$ their length.

Define
\[
\lhat := \frac{\lat - \E \lat}{\sqrt{\var\big(\lat\big)}}.
\]

\begin{figure}
	\centering
	\includegraphics[width=.5\textwidth]{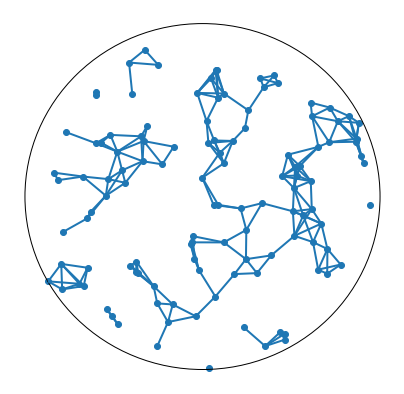}
	\caption{Realisation of the Gilbert graph}
\end{figure}

\begin{rem}
	For the sake of continuity with the article \cite{TSRGilbert}, we use the convention that the intensity $t\, dx$ of $\eta$ grows and the observation window $W$ stays constant. In the other applications presented in this article, we keep the intensity constant and instead let the observation window grow as $tW$. Note that one can pass from one setting to the other by a simple rescaling. Indeed, consider $\tilde{\eta}$ a Poisson measure on $\R^d$ and Lebesgue intensity and for $s \geq 1$, construct  a Gilbert graph on $\tilde{\eta}_{|sH}$ by connecting two points $x \neq y \in \tilde{\eta}_{|sH}$ if and only if $|x-y| < \tilde{\epsilon}_s$. For $\alpha \in \R$, let
	\[
	F_s^{(\alpha)} := \frac{1}{2} \sum_{x,y \in \tilde{\eta}_{|sW},x \neq y} \ind{|x-y| < \tilde{\epsilon}_s} |x-y|^\alpha.
	\]
	Then $\fat$ is equal in law to $s^\alpha L_{s^d}^{(\alpha)}$ with $\epsilon_{s^d}=s^{-1}\tilde{\epsilon}_s$. The central limit theorem for $\fat$ can be deduced from the one for $\lat$.
\end{rem}
The first mention of the Gilbert graph was by Gilbert in \cite{Gil61}, in dimension $d=2$. It has been treated in many works under various names: geometric or proximity graph, interval graph (when $d=1$) or disk graph (when $d=2$). The book \cite{Pen03} provides a vast background and literature review and we also refer to \cite{LRP13I,LRP13II} for central limit theorems of generalisations of the Gilbert graph and \cite{RS13} for a quantitative CLT on a sum of weighted edge-lengths. For a comprehensive overview of the Gilbert graph in the context of $U$-statistics, see \cite{LRRinPR}, especially Section~4.3. See also \cite{McD03,Mul06,GM09,BP14,DST16,GT20}.
In \cite{TSRGilbert}, the authors give a complete picture of the asymptotic behaviour of $\lhat$ for $\alpha \in \R$. In particular, they show that for $\alpha>-\frac{d}{2}$, the quantity $\lhat$ converges in distribution to a standard Gaussian as $t \rightarrow \infty$, provided that $t^2\epsilon_t^2 \rightarrow \infty$. They also give a quantitative bound on the speed of convergence in Kolmogorov distance in the case $\alpha>-\frac{d}{4}$. As an application of our estimates, we recover this speed of convergence below and extend to the case $-\frac{d}{2} < \alpha \leq -\frac{d}{4}$. The authors of \cite{TSRGilbert} show that CLTs hold also for $-d <\alpha \leq -\frac{d}{2}$ with different rescalings; however, establishing corresponding speeds of convergence in this range is still an open problem.

\begin{thm}\label{thmGilbert}
	Let $\alpha>-\frac{d}{2}$ and assume that $t^2\epsilon_t^d \rightarrow \infty$ as $t \rightarrow \infty$. Then for $t\geq 1$ large enough
	\begin{itemize}
		\item if $\alpha>-\frac{d}{4}$, there is a constant $c_1>0$ such that
		\begin{equation}\label{eqGil1}
		\max\left\{ d_W\left(\hat{L}^{(\alpha)}_t,N\right), d_K\left(\hat{L}^{(\alpha)}_t,N\right) \right\} \leq c_1 \left(t^{-1/2} \vee (t^2\epsilon_t^d)^{-1/2}\right).
		\end{equation}
		\item if $-\frac{d}{2} < \alpha \leq -\frac{d}{4}$, then for any $1<p<-\frac{d}{2 \alpha}$, there is a constant $c_2>0$ such that
		\begin{equation}\label{eqGil2}
		\max\left\{ d_W\left(\hat{L}^{(\alpha)}_t,N\right), d_K\left(\hat{L}^{(\alpha)}_t,N\right) \right\} \leq c_2 \left(t^{-1+1/p} \vee (t^2\epsilon_t^d)^{-1+1/p}\right).
		\end{equation}
	\end{itemize}
\end{thm}
\begin{rem}
	A careful inspection of the bounds applied to $\gamma_3$ in the proof of Theorem~\ref{thmGilbert} reveals that in the sparse regime ($t\epsilon_t^d \rightarrow 0$) when $-\frac{d}{2}<\alpha\leq-\frac{d}{4}$, a slightly improved rate can be found for the Wasserstein distance. Indeed, for any $1<p<-\frac{d}{2 \alpha}$ and any $0<r<-\frac{d}{\alpha}-2$, there is a constant $c>0$ such that
	\[
	d_W(\hat{L}^{(\alpha)}_t,N) \leq c \left(t^{-1+1/p} \vee (t^2\epsilon_t^d)^{-r/2}\right).
	\]
	Since $\frac{r}{2} \in \left(0,-\frac{d}{2\alpha}-1\right)$ and $1-\frac{1}{p} \in \left(0,1+\frac{2\alpha}{d}\right)$ and $1+\frac{2\alpha}{d} < -\frac{d}{2\alpha}-1$, one can choose $\frac{r}{2} > 1-\frac{1}{p}$. This then gives a slightly faster convergence rate. As an illustrating example, consider the case where $\epsilon_t^d = t^{-\theta}$ with $1<\theta<2$. Then $t\epsilon_t^d \rightarrow 0$ and $t^2\epsilon_t^d \rightarrow \infty$. Theorem~\ref{thmGilbert} provides the rate of convergence $t^{(-1+\frac{1}{p})(2-\theta)}$, and by following this strategy it can be improved to $t^{-1+\frac{1}{p}} \vee t^{-\frac{r}{2}(2-\theta)}$ with $\frac{r}{2}>1-\frac{1}{p}$.
\end{rem}

\subsection{$k$-Nearest Neighbour graphs}
For a finite generic set $\mu \subset \R^d$ and a positive integer $k \in \N$, the $k$\textbf{-Nearest Neighbour graph} has vertex set $\mu$ and an edge between $x,y \in \mu$ if and only if $y$ is one of the $k$ nearest points to $x$ or vice-versa.

For our functional of interest, consider the following framework:
\begin{itemize}
	\item $H \subset \R^d$ is a convex body;
	\item $\eta$ is an $(\R^d,dx)$-Poisson measure;
	\item $\phi:(0,\infty) \rightarrow (0,\infty)$ is a decreasing function such that there is an $r>2$ verifying
	\begin{equation}\label{eqkNNC1}
		\int_0^1 \phi(s)^{r} s^{d-1} ds < \infty.
	\end{equation}
\end{itemize}
\noindent For any finite generic set $\mu \subset \R^d$, define
\[
F(\mu) := \frac{1}{2} \sum_{x,y\in\mu,x\neq y} \ind{x \in N(y,\mu) \text{ or } y \in N(x,\mu)} \phi(|x-y|),
\]
where $N(x,\mu)$ is the set of $k$-nearest neighbours of $x$ in $\mu$. For $t\geq 1$, define $F_t:=F(\eta_{|tH})$ and set $\hat{F}_t:= (F_t-\E F_t)\var(F_t)^{-1/2}$.	
\begin{figure}
	\centering
	\includegraphics[width=.5\textwidth]{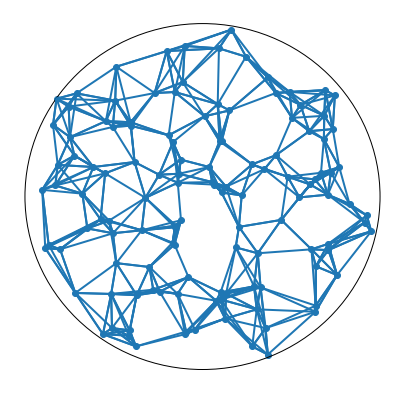}
	\caption{Realisation of a 6-Nearest Neighbour graph}
\end{figure}

	The $k$-nearest neighbour graph is a model frequently used in e.g. social sciences or geography, see \cite{Wade07} for a discussion of applications. Quantitative central limit theorems for edge-related quantities were shown in \cite{AB93,PY05} and subsequently improved in \cite{LPS14}. For a discussion of the literature, we refer to \cite{LPS14}.
	
	In \cite{LPS14}, the authors give a quantitative central limit theorem for the sum of power-weighted edge-lengths with powers $\alpha \geq 0$, at a speed of convergence of $t^{-d/2}$. We complement this result by dealing with the case $\alpha \in \left(-\frac{d}{2},0\right)$. Note that in this case, the CLT is new even in its qualitative version. In the regime $\alpha \in \left(-\frac{d}{4},0\right)$, we also recover the same, presumably optimal, speed of convergence of $t^{-d/2}$ as in \cite{LPS14}, whereas in the case $\alpha \in \left(\left.-\frac{d}{2},-\frac{d}{4}\right]\right.$, we find a speed of convergence that decreases as $\alpha$ approaches $-\frac{d}{2}$. It is natural to ask what happens when $\alpha \leq -\frac{d}{2}$. We consider this a separate issue and leave it open for further research.

\begin{thm}\label{thmkNN}
	Under the conditions stated above, for any $p \in (1,2]$ such that $p<\frac{r}{2}$, there is a constant $c>0$ such that, for $t\geq 1$,
	\[
	\max\left\{d_W\big(\hat{F}_t,N\big),d_K\big(\hat{F}_t,N\big)\right\} \leq c t^{d(1/p-1)}.
	\]
	This inequality holds in particular for the function $\phi(x)=x^{-\alpha}$ with $0<\alpha<\frac{d}{2}$, for any $p \in (1,2]$ such that $p<\frac{d}{2\alpha}$. 
\end{thm}

\subsection{Radial Spanning Tree}
Let $\mu \subset \R^d\setminus \{0\}$ be a finite set, generic with respect to the point $0$. The radial spanning tree on $\mu$, in short $RST(\mu)$, is constructed as follows:
\begin{itemize}
	\item The set of vertices is given by $\mu \cup \{0\}$;
	\item for every $x \in \mu$, we add exactly one edge to the point $z \in \mu \cup \{0\} \cap B^d(0,|x|)$ which minimises $|x-z|$. We call $z$ the radial nearest neighbour of $x$ and say `$x$ connects to $z$', denoted by `$x \rightarrow z$ in $\mu$'. We denote the length $|x-z|$ by $g(x,\mu)$.
\end{itemize}
In order to define our functional of interest, consider the following setting:
\begin{itemize}
	\item $H \subset \R^d$ a convex body such that $B^d(0,\epsilon) \subset H$ for some $\epsilon>0$;
	\item $\eta$ is an $(\R^d,dx)$-Poisson measure;
	\item $\phi:(0,\infty) \rightarrow (0,\infty)$ is a decreasing function such that there is an $r>2$ satisfying
	\begin{equation}\label{eqRSTcond}
		\int_0^1 \phi(s)^{r} s^{d-1} ds < \infty.
	\end{equation}
\end{itemize}

For any finite set $\mu \subset \R^d$ generic with respect to $0$, define
\[
F(\mu) := \sum_{x\in\mu} \phi(g(x,\mu))
\]
and for $t \geq 1$, define $F_t := F(\eta_{|tH})$. Set $\hat{F}_t:= (F_t-\E F_t)\var(F_t)^{-1/2}$.

\begin{figure}
	\centering
	\includegraphics[width=.5\textwidth]{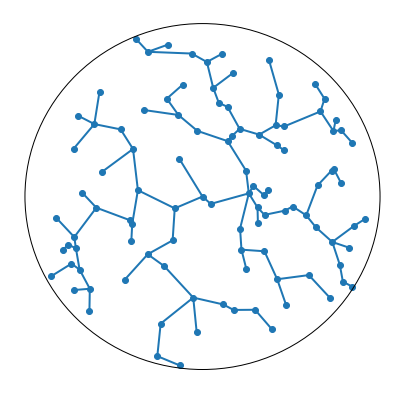}
	\caption{Realisation of the Radial Spanning Tree}
\end{figure}

The radial spanning tree was developed in \cite{BB07} as a model related to the minimal directed spanning tree and to Poisson forests. The paper also discusses various applications, most notably in communication networks. Further work on the radial spanning tree has been done in \cite{PW09,BCT13,ST14rst}. In \cite{ST14rst}, the authors give a quantitative central limit theorem for sums of power-weighted edge-lengths of the radial spanning tree for powers $\alpha  \geq 0$. The framework is one where the intensity of the Poisson measure increases while the observation window stays constant. After rescaling to our framework of a constant intensity and a growing window, one obtains by \cite[Theorem~1.2]{ST14rst} a speed of convergence of $t^{-d/2}$. We add quantitative central limit theorems for $\alpha \in \left(-\frac{d}{2},0\right)$, recovering the same speed of $t^{-d/2}$ for $\alpha \in \left(-\frac{d}{4},0\right)$. Note that this CLT is new even in its qualitative version. As for the $k$-Nearest Neighbour graph, the case $\alpha \leq -\frac{d}{2}$ will be the object of further research.
	
\begin{thm}\label{thmRST}
Under the conditions stated above, for any $p \in (1,2]$ such that $p<\frac{r}{2}$, there is a constant $c>0$ such that for $t\geq 1$,
\begin{equation}\label{eqRST}
\max\left\{d_W\big(\hat{F}_t,N\big),d_K\big(\hat{F}_t,N\big)\right\} \leq c t^{d(1/p-1)}.
\end{equation}
This inequality holds in particular for the function $\phi(x)=x^{-\alpha}$ with $0<\alpha<\frac{d}{2}$, for any $p \in (1,2]$ such that $p<\frac{d}{2\alpha}$. 
\end{thm}

The rest of the paper is devoted to providing the proofs of the results in Sections~\ref{secWasserKolg} to \ref{secAppl}.

\appendix
\section{Background on Malliavin Calculus} \label{secBackground}
In this section, we present several useful notions related to Malliavin calculus. Unless otherwise indicated, these results are explained in \cite{LastPoiss}. We work in the setting of Section~\ref{secFrame}: in particular, $\chi$ indicates a $(\W,\nu)$-Poisson measure.

We start with three useful \textbf{isometry} relations. Let $f \in L^2(\W^n,\nu^{(n)})$ and $g \in L^2(\W^m,\nu^{(m)})$. Then
\begin{equation}\label{eqIsom1}
	\E \Fint_n(f)\Fint_m(g) = \ind{m=n} n! \int_{\W^n} \tilde{f}(x)\tilde{g}(x)\, \nu^{(n)}(dx),
\end{equation}
where $\tilde{f}$ and $\tilde{g}$ are the \textbf{symmetrisations} of $f$ and $g$ defined by
\[
\tilde{f}(x_1,...,x_n) = \frac{1}{n!} \sum_{\sigma \in \Sigma_n} f(x_{\sigma(1)},...,x_{\sigma(n)}),
\]
the set $\Sigma_n$ being the set of all permutations of $\{1,...,n\}$. See \cite[Lemma~4]{LastPoiss} and the remark thereafter on page~10 for a proof.

Relation \eqref{eqIsom1} implies that for $F,G \in L^2(\p_\chi)$ having an expansion \eqref{eqWienerItoChaos} with kernels $f_n$ and $g_n$ respectively,
\begin{equation}\label{eqIsom}
	\E FG = \E F \E G + \sum_{n=1}^\infty n! \int_{\W^n} f_n g_n \, d\nu^{(n)}.
\end{equation}
By \cite[Theorem~5]{LastPoiss}, if $h \in L^2(\mathbf{N} \times \W)$ satisfies \eqref{eqDomDelCond}, then
\begin{equation}\label{eqIsomDelta}
	\E \delta(h)^2 = \E \int_\W h(\chi,w)^2 \nu(dw) + \E \int_\W \int_\W \MalD_z h(\chi,w) \MalD_w h(\chi,z) \, \nu(dz)\nu(dw).
\end{equation}

A well-known relation in Malliavin calculus is the so-called \textbf{integration by parts} formula: for $F \in \dom \MalD$ and $h \in \dom \delta$, we have $\E \int_\W h(\chi,w)\MalD_w F \nu(dw) = \E[F\delta(h)]$ (cf. \cite[Theorem~4]{LastPoiss}). The condition on $F$ is however suboptimal in our context, which is why we need a version of integration by parts under slightly different assumptions.

\begin{lem}\label{lemNewIBP}
	Let $h \in \dom \delta \cap L^1(\mathbf{N} \times \W)$ and $F \in L^0(\p_\chi)$ bounded. Then
	\[
	\E \int_\W h(\chi,w)\MalD_w F \nu(dw) = \E[F\delta(h)].
	\]
\end{lem}
\begin{proof}
	Since $h \in \dom \delta \cap L^1(\mathbf{N} \times \W)$, it is easy to check that the expectations appearing in the statement are well-defined and finite. Note that
	\begin{align}
		\E \int_\X \MalD_w F\, h(\chi,w) \nu(dw) &= \E \int_\X  (F(\chi + \delta_w)-F(\chi)) h(\chi,w) \nu(dw) \notag\\
		&= \E \int_\X F(\chi + \delta_w) h(\chi,w) \nu(dw) - \E \int_\X F(\chi) h(\chi,w) \nu(dw), \label{iP100eq1}
	\end{align}
	where the last line is justified by the fact that $F$ is bounded and $h \in L^1(\mathbf{N} \times \W)$, so both integrals are well-defined. We now apply Mecke formula \eqref{eqMecke} to deduce that \eqref{iP100eq1} equals
	\begin{multline}
		\E \int_\X F(\chi) h(\chi - \delta_w,w) \chi(dw) - \E \int_\X F(\chi) h(\chi,w) \nu(dw) \\= \E F(\chi) \left(\int_\X  h(\chi - \delta_w,w) \chi(dw) - \int_\X h(\chi,w) \nu(dw) \right).
	\end{multline}
	Since $h \in \dom \delta \cap L^1(\mathbf{N} \times \X)$,
	\[
	\int_\X  h(\chi - \delta_w,w) \chi(dw) - \int_\X h(\chi,w) \nu(dw) = \delta(h) \qquad \asp
	\]
	The result follows.
\end{proof}

Next, we introduce the \textbf{Ornstein-Uhlenbeck} operator $P_\tau$. For $F \in L^1(\p_\chi)$ and $\tau \in [0,1]$, we define 
\begin{equation}\label{eqOrnAlt}
	P_\tau F = \int \E \big[F(\chi^{\tau} + \xi) \big| \chi \big] \Pi_{\tau}(d\xi),
\end{equation}
where $\chi^{\tau}$ is a $\tau$-thinning of $\chi$ (see \cite[p. 24]{LastPoiss} and the reference given therein) and $\Pi_\tau$ is the law of an independent Poisson measure with intensity measure $(1-\tau)\nu$. It follows by Jensen's inequality that for all $p \geq 1$, one has
\begin{equation}\label{eqOrnx2}
	\E |P_\tau F|^p \leq \E |F|^p.
\end{equation}
By \cite[Lemma~6]{LastPoiss}, for all $F \in L^2(\p_\chi)$ and all $\tau \in [0,1]$, for $\nu^{(n)}$-a.e. $w_1,...,w_n \in \W$ it holds $\p$-a.s. that
\begin{equation}\label{eqOrnx1}
	\MalD^{(n)}_{w_1,...,w_n} (P_\tau F) = \tau^n P_\tau \MalD^{(n)}_{w_1,...,w_n} F.
\end{equation}
This implies that for $F \in L^2(\p_\chi)$, the following expansion holds (see also \cite[(79)]{LastPoiss}):
\begin{equation}\label{eqOrnUhlDef}
	P_\tau F = \E F + \sum_{n=1}^\infty \tau^n \Fint_n(f_n).
\end{equation}

The following lemma summarises some useful approximation properties of the Ornstein-Uhlenbeck operator.
\begin{lem}\label{lemOrnTrick}
	Let $h \in L^2(\mathbf{N}_\W \times \W)$ and let $\tau \in (0,1)$. Then $P_\tau h$ satisfies condition \eqref{eqDomDelCond} and $P_\tau h \rightarrow h$ in $L^2(\mathbf{N}_\W \times \W)$ as $\tau \rightarrow 1$. Moreover, for $w,z \in \W$, and all $p \geq 1$,
	\begin{align}
		&\E |P_\tau h(\chi,w)|^p \leq \E |h(\chi,w)|^p \label{eqOrnTrick1}\\
		\intertext{and}
		&\E |\MalD_z P_\tau h(\chi,w)|^p \leq \E |\MalD_z h(\chi,w)|^p. \label{eqOrnTrick2}
	\end{align}
	Under the additional assumption that $h \in \dom \delta$, it holds that $\delta(P_\tau h) \rightarrow \delta(h)$ in $L^2(\p_\chi)$ as $\tau \rightarrow 1$.
\end{lem}
\begin{proof}
	By the isometry property \eqref{eqIsom} and the expansions \eqref{eqOrnUhlDef} and \eqref{eqhExpand}, we infer that
	\[
	\E \int_\W \int_\W \left(\MalD_w P_\tau h(\chi,z)\right)^2 \nu(dw) \nu(dz) = \sum_{n=0}^\infty n \cdot n!\, \tau^{2n} \norm{h_n}_{n+1}^2,
	\]
	where $\norm{.}_{n}$ is the norm in $L^2(\W^{n},\nu^{(n)})$. Now note that $\sup_{n \geq 1} n\tau^{2n} < \infty$ and
	\[
	\sum_{n=0}^\infty n!\, \norm{h_n}_{n+1}^2 = \int_\W \E h(w)^2 \nu(dw) < \infty,
	\]
	hence $P_\tau h$ satisfies \eqref{eqDomDelCond}. Similarly using the expansions, we deduce that
	\[
	\E \int_\W \left(P_\tau h(\chi,w) - h(\chi,w)\right)^2 \nu(dw) = \sum_{n=1}^\infty n! \, (\tau^n-1)^2 \norm{h_n}_{n+1}^2.
	\]
	By dominated convergence, this expression tends to $0$ as $\tau \rightarrow 1$. Properties \eqref{eqOrnTrick1} and \eqref{eqOrnTrick2} follow immediately from \eqref{eqOrnx2} and \eqref{eqOrnx1}. For the last point, note that
	\[
	\delta(P_\tau h)-\delta(h) = \sum_{n=0}^\infty \Fint_{n+1}((\tau^n-1)h_n)
	\]
	and
	\[
	\E \left(\delta(P_\tau h)-\delta(h)\right)^2 = \sum_{n=0}^\infty (n+1)! (1-\tau^n)^{2} \norm{\tilde{h}_n}_{n+1}^2,
	\]
	which converges to $0$ as $\tau \rightarrow 1$ by dominated convergence since $h \in \dom \delta$.
\end{proof}

The following lemma is used on several occasions:
\begin{lem}[{\cite[Theorem~1.5]{LPen1}}] \label{lemCondCov}
	Let $\eta$ be a $(\W \times [0,1],\nu \otimes ds)$-Poisson measure and let $F,G \in L^2(\p_\eta)$. Then
	\[
	\E \int_{\W}\int_0^1 \E[\MalD_{(y,s)}F|\eta_{|\W \times [0,1]}]^2 \lambda(dy)ds < \infty,
	\]
	and an analogous estimate holds for $G$. Moreover,
	\[
	\cov (F,G) = \E \int_{\W}\int_0^1 \E[\MalD_{(y,s)}F|\eta_{|\W \times [0,1]}]\E[\MalD_{(y,s)}G|\eta_{|\W \times [0,1]}] \lambda(dy)ds.
	\]
\end{lem}

\section{Proof of Theorem~\ref{thmItoFormula}}\label{secPIto}

	Each summand on the RHS of \eqref{eqItoProcess} is well defined by virtue of Mecke formula~\eqref{eqMecke} and the discussion thereafter. By Mecke formula \eqref{eqMecke} it can be seen that $h(\eta-\delta_{(y,s)},y,s)$ is almost surely integrable with respect to the measure $\eta(dy,ds)$. By assumption, $h$ is also integrable with respect to $\lambda(dy)ds$. It now follows by dominated convergence that the process $(X_t)$ is càdlàg.
	
	As a next step, we show that the integrals on the RHS of \eqref{eqIto} are well-defined. For this, note that $(X_t)_{t \in [0,1]}$ (and $(X_{t-})_{t \in [0,1]}$) are a.s. bounded on $[0,1]$. Indeed,
	
	\begin{align}
			\E \sup_{t \in [0,1]} |X_t| &\leq \E \int_{\X \times [0,1]} |h(\eta-\delta_{(y,s)},y,s)| \eta(dy,ds) + \E \int_{\X}\int_0^1 |h(\eta,y,s)| \lambda(dy)ds \notag\\
										&= 2 \E \int_{\X}\int_0^1 |h(\eta,y,s)| \lambda(dy)ds < \infty, \label{iPsup}
		\end{align}
	where the second inequality follows by Mecke formula \eqref{eqMecke}. Now write
	\begin{multline}\label{iPc1}
				\int_{\X \times [0,t]} \left|\phi(X_{s-}+h(\eta-\delta_{(y,s)},y,s))-\phi(X_{s-})\right| \eta(dy,ds) \\ \leq
			\int_{\X \times [0,1]} \left| \int_0^1 \phi'\left(X_{s-}+uh(\eta-\delta_{(y,s)},y,s)\right)h(\eta-\delta_{(y,s)},y,s) du\right|  \eta(dy,ds).
		\end{multline}
	By the boundedness of $h$, we infer from \eqref{iPsup} that $X_{s-}+uh(\eta-\delta_{(y,s)},y,s)$ almost surely takes values in a compact interval. Since the function $\phi'$ is continuous, this entails
	\[
	\sup_{s,u\in[0,1],y\in\X}|\phi'(X_{s-}+uh(\eta-\delta_{(y,s)},y,s))| < \infty \qquad \p\text{-a.s.}
	\]
	Hence the RHS of \eqref{iPc1} is bounded by
	\[
		\left(\sup_{s,u\in[0,1],y\in\X}|\phi'(X_{s-}+uh(\eta-\delta_{(y,s)},y,s))|\right) \int_{\X \times [0,1]} |h(\eta-\delta_{(y,s)},y,s)| \eta(dy,ds) < \infty \qquad \p\text{-a.s.}
	\]
	which implies that the first integral on the RHS of \eqref{eqIto} is well-defined. Similarly,
	\[
	\sup_{s \in [0,1]} |\phi'(X_s)| < \infty \qquad \p\text{-a.s.}
	\]
	and hence
	\[
		\int_\X\int_0^1 |\phi'(X_s)h(\eta,x,s)| \lambda(dx)ds < \infty \qquad \p\text{-a.s.}
	\]
	This concludes the proof that all terms in \eqref{eqIto} are well-defined.
	
	To show \eqref{eqIto}, we start by showing it for an approximation of $X_t$. Let $(U_m)_{m\in\N}\subset\X$ s.t. $\bigcup_mU_m = \X$ and $\forall\, m\in\N$, $\lambda(U_m)<\infty$ and $U_m \subset U_{m+1}$. Define
	\[
	X_t^{(m)}(\eta) := X_0 + \int_{U_m \times [0,t]} h(\eta-\delta_{(y,s)},y,s) \eta(dy,ds)
	- \int_{U_m}\int_0^t  h(\eta,y,s) \lambda(dy)ds.
	\]
	Define the event $\Omega_0:=\{\eta(U_m \times [0,1])<\infty, m\geq 1\}$. Then $\p(\Omega_0)=1$ and, since $\eta$ is proper, $\forall \omega \in \Omega_0$ and all $m \geq 1$ there exists a finite collection of points $(y_1,s_1),...,(y_{n_m},s_{n_m}) \in U_m\times[0,1]$ (all depending on $\omega$) s.t.
	\[
	\eta_{|U_m \times [0,1]} = \sum_{i=1}^{n_m} \delta_{(y_i,s_i)}.
	\]
	W.l.o.g. we can assume that $0<s_1<s_2<...<s_{n_m}<1$ and set $s_0:=0$ and $s_{n_m+1}:=1$. Now the process $X^{(m)}$ can be written as
	\[
	X^{(m)}_t = \sum_{i=1}^{n_m} \1_{\{s_i\leq t\}} h(\eta-\delta_{(y_i,s_i)},y_i,s_i) - \int_{U_m} \int_0^t h(\eta,y,s) \lambda(dy) ds,
	\]
	and one has the telescopic sums:
	
	\begin{align}
		\phi\left(X^{(m)}_t\right) - \phi(X_0)%
		&= \sum_{i=1}^{n_m+1} \left(\phi\left(X^{(m)}_{s_i \wedge t}\right) - \phi\left(X^{(m)}_{s_{i-1} \wedge t}\right)\right) \notag\\
		&= \sum_{i=1}^{n_m} \left(\phi\left(X^{(m)}_{s_i \wedge t}\right) - \phi\left(X^{(m)}_{(s_{i}-) \wedge t}\right)\right) + \sum_{j=1}^{n_m+1} \left(\phi\left(X^{(m)}_{(s_j-) \wedge t}\right) - \phi\left(X^{(m)}_{s_{j-1} \wedge t}\right)\right) \notag\\
		&= I_1(t) + I_2(t),
	\end{align}
	where
	\[X^{(m)}_{(s-) \wedge t} = %
	\begin{cases*}
		X^{(m)}_{s-} & if $s \leq t$ \\
		X^{(m)}_t & if $s > t$.
	\end{cases*}
	\]
	The sum $I_1(t)$ represents what is happening at jump times, whereas $I_2(t)$ shows what happens in between jump times. The index $i=n_m+1$ does not appear in the sum $I_1(t)$ because $\p$-a.s. there is no jump at time $t=1$.
	
	\noindent We first study $I_1(t)$.
	\begin{align}
		I_1(t) &= \sum_{i=1}^{n_m} \left(\phi\left(X^{(m)}_{s_i \wedge t}\right) - \phi\left(X^{(m)}_{(s_{i}-) \wedge t}\right)\right) \notag\\
		&= \sum_{i=1}^{n_m} \1_{\{s_i \leq t\}} \left(\phi\left(X^{(m)}_{s_i-}+ h(\eta-\delta_{(y_i,s_i)},y_i,s_i)\right) - \phi\left(X^{(m)}_{s_{i}-}\right)\right) \notag\\
		&= \int_{U_m \times [0,t]} \left(\phi\left(X^{(m)}_{s-}+ h(\eta-\delta_{(y,s)},y,s)\right) - \phi\left(X^{(m)}_{s-}\right)\right) \eta(dy,ds).
	\end{align}
	Now consider $I_2(t)$. For $s\in [s_{i-1},s_i)$,
	\[
	X^{(m)}_s = \sum_{j=1}^{i-1} h(\eta-\delta_{(y_j,s_j)}) - \int_{U_m}\int_0^s h(\eta,y,u) \lambda(dy)du
	\]
	and so for $s\in (s_{i-1},s_i)$
	\[
	\frac{d}{ds} X^{(m)}_s = - \int_{U_m} h(\eta,y,s) \lambda(dy).
	\]
	This implies that
	\[
	\phi\left(X^{(m)}_{(s_i-) \wedge t}\right) - \phi\left(X^{(m)}_{s_{i-1} \wedge t}\right) = -\int_{s_{i-1} \wedge t}^{(s_i-)\wedge t} \phi'\left(X^{(m)}_s\right) \int_{U_m} h(\eta,y,s) \lambda(dy) ds.
	\]
	We conclude that
	\begin{align}
		I_2(t) &= \sum_{i=1}^{n_m+1} \phi\left(X^{(m)}_{(s_i-) \wedge t}\right) - \phi\left(X^{(m)}_{s_{i-1} \wedge t}\right) \notag\\
		&= - \sum_{i=1}^{n_m+1} \int_{s_{i-1} \wedge t}^{(s_i-)\wedge t}\int_{U_m} \phi'\left(X^{(m)}_s\right)  h(\eta,y,s) \lambda(dy) ds \notag\\
		&= - \int_0^t \int_{U_m} \phi'\left(X^{(m)}_s\right)  h(\eta,y,s) \lambda(dy) ds.
	\end{align}
	We have shown until now that
	\begin{multline}
		\phi\left(X^{(m)}_t\right) = \phi(X_0) + \int_{U_m \times [0,t]} \left(\phi\left(X^{(m)}_{s-}+h(\eta-\delta_{(y,s)},y,s)\right)-\phi\left(X^{(m)}_{s-}\right)\right) \eta(dy,ds)\\ - \int_{U_m}\int_0^t \phi'\left(X^{(m)}_s\right)h(\eta,x,s) \lambda(dx)ds \qquad \p\text{-a.s.} \label{iP1Ito}
	\end{multline}
	Our aim is now to let $m \rightarrow \infty$. By dominated convergence, $X^{(m)}_t \rightarrow X_t$ a.s. for fixed $t \in [0,1]$. We would like to use dominated convergence for both the second and third terms on the RHS of \eqref{iP1Ito}. Start by noting that $(X_t^{(m)})_{t \in [0,1]}$ (as well as $(X_{t-}^{(m)})_{t \in [0,1]}$ and $(X_{t})_{t \in [0,1]}$) are $\p$-a.s. uniformly bounded on $[0,1]$ and in $m$. Indeed,
	\begin{multline}
		\sup_{\substack{m\in\N \\ t \in [0,1]}} \max\left\{ \left|X_t^{(m)}\right|, \left|X_{t-}^{(m)}\right|, \left|X_{t}\right|  \right\} \\ \leq \int_{\X \times [0,1]} |h(\eta-\delta_{(y,s)},y,s)| \eta(dy,ds) + \int_{\X}\int_0^1 |h(\eta,y,s)| \lambda(dy)ds < \infty \qquad \p\text{-a.s.} \label{iP1Bound}
	\end{multline}
	We start with the second term on the RHS of \eqref{iP1Ito} and write
	\begin{multline}\label{iP1step}
		\int_{U_m \times [0,t]} \left|\phi\left(X^{(m)}_{s-}+h(\eta-\delta_{(y,s)},y,s)\right)-\phi\left(X^{(m)}_{s-}\right)\right| \eta(dy,ds) \\ \leq
		\int_{\X \times [0,1]} \left| \int_0^1 \phi'\left(X^{(m)}_{s-}+uh(\eta-\delta_{(y,s)},y,s)\right)h(\eta-\delta_{(y,s)},y,s) du\right|  \eta(dy,ds).
	\end{multline}
	By boundedness of $h$, we get that $X^{(m)}_{s-}+uh(\eta-\delta_{(y,s)},y,s)$ almost surely takes values in a compact interval independent of $m$. The function $\phi'$ being continuous, we deduce
	\[
	\sup_{\substack{m\in\N \\ s,u\in[0,1] \\ y\in\X}} \left|\phi'\left(X^{(m)}_{s-}+uh(\eta-\delta_{(y,s)},y,s)\right)\right| < \infty \qquad \p\text{-a.s.}
	\]
	Hence the RHS of \eqref{iP1step} is bounded by
	\[
	\bigg(\sup_{\substack{m\in\N \\ s,u\in[0,1] \\ y\in\X}}\left|\phi'\left(X^{(m)}_{s-}+uh(\eta-\delta_{(y,s)},y,s)\right)\right|\bigg) \int_{\X \times [0,1]} |h(\eta-\delta_{(y,s)},y,s)| \eta(dy,ds) < \infty \qquad \p\text{-a.s.}
	\]
	We can thus apply dominated convergence and deduce that
	\begin{multline}
		\int_{U_m \times [0,t]} \left(\phi\left(X^{(m)}_{s-}+h(\eta-\delta_{(y,s)},y,s)\right)-\phi\left(X^{(m)}_{s-}\right)\right) \eta(dy,ds) \\ \stackrel{m\rightarrow \infty}{\longrightarrow} \int_{\X \times [0,t]} \left(\phi\left(X_{s-}+h(\eta-\delta_{(y,s)},y,s)\right)-\phi\left(X_{s-}\right)\right) \eta(dy,ds) \qquad \p\text{-a.s.}
	\end{multline}
	Now, similarly,
	\begin{multline}
	\int_{U_m}\int_0^t \left|\phi'\left(X^{(m)}_s\right)h(\eta,x,s)\right| \lambda(dx)ds \\ \leq \left(\sup_{\substack{m\in\N \\ s\in[0,1]}} \left|\phi'\left(X^{(m)}_s\right)\right|\right)\int_{\X}\int_0^t \left|h(\eta,x,s)\right| \lambda(dx)ds < \infty \qquad \p\text{-a.s.}
	\end{multline}
	and by dominated convergence
	\[
	\int_{U_m}\int_0^t \phi'\left(X^{(m)}_s\right)h(\eta,x,s) \lambda(dx)ds \\ \stackrel{m\rightarrow \infty}{\longrightarrow} \int_{\X}\int_0^t \phi'\left(X_s\right)h(\eta,x,s) \lambda(dx)ds \qquad \p\text{-a.s.}
	\]
	The fact that $\phi\left(X^{(m)}_t\right) \rightarrow \phi\left(X_t\right)$ a.s. follows by continuity of $\phi$. This concludes the proof. \qed

\section{Proofs of Theorem~\ref{thmGeneralpPoin} and Corollaries~\ref{corpPoinGen} and \ref{corKolgIneq}}\label{secPPPoin}
\begin{lem} \label{lempHolder}
	For $p\in (1,2]$, the function $\phi:\R \rightarrow \R:x \mapsto |x|^p$ is continuously differentiable and its derivative $\phi'$ is $(p-1)$-Hölder continuous with Hölder constant $c_\phi=p2^{2-p}$.
\end{lem}
\begin{proof}
	We want to show that
	\[
	c_\phi := \sup_{\substack{a \neq b\\a,b \in \R}} \frac{|\phi'(a)-\phi'(b)|}{|a-b|^{p-1}} = p2^{2-p}.
	\]
	First we observe that $\phi'(x) = p \sgn(x) |x|^{p-1}$, with the convention that $sgn(0)=1$. Let $a \neq b$ and assume without loss of generality that $|a|\geq |b| \geq 0$. Then $a \neq 0$ and
	\[
	\frac{\left|\sgn(a)|a|^{p-1}-\sgn(b)|b|^{p-1}\right|}{|a-b|^{p-1}} = \frac{\left|1-\sgn\left(\frac{b}{a}\right)\left|\frac{b}{a}\right|^{p-1}\right|}{\left|1-\frac{b}{a}\right|^{p-1}}.
	\]
	It follows that
	\[
	c_\phi = p \sup_{x\in [-1,1)} \frac{1-\sgn(x)\left|x\right|^{p-1}}{(1-x)^{p-1}} =: p \sup_{x\in [-1,1)} f(x)
	\]
	The function $f$ is differentiable on $[-1,1)$ with derivative $f'(x)=(p-1)(1-x)^{-p}(1-|x|^{p-2}) \leq 0$, so $f$ is decreasing and therefore
	\[
	f(x) \leq f(-1) = 2^{2-p}.
	\]
	We conclude that $c_\phi = p 2^{2-p}$.
\end{proof}

\begin{proof}[Proof of Theorem~\ref{thmGeneralpPoin}]
	We start by showing the result for an approximation of $h$. Let $(U_m)_{m\in\N}\subset\X$ s.t. $\bigcup_mU_m = \X$ and $\forall\, m\in\N$, $\lambda(U_m)<\infty$ and $U_m \subset U_{m+1}$. Define for $\mu \in \mathbf{N}$, $(y,s)\in \X\times[0,1]$:
	\begin{equation}\label{iP2approx}
	h_m(\mu,y,s):= \left[(h(\mu,y,s) \wedge m) \vee (-m)\right] \ind{y \in U_m}.
	\end{equation}
	Now $h_m \in L^1(\mathbf{N}\times\X\times[0,1]) \cap L^2(\mathbf{N}\times\X\times[0,1])$ and $h_m \rightarrow h$ in $L^2$ as $m \rightarrow \infty$. Moreover, $|h_m| \leq |h|$ and $|\MalD_{(x,t)}h_m(\eta,y,s)| \leq |\MalD_{(x,t)}h(\eta,y,s)|$ for all $(x,t),(y,s) \in \X\times[0,1]$. This implies that $h_m$ satisfies \eqref{eqDomDelCond} and hence $h_m \in \dom \delta$.
	
	\noindent Finally, $h_m$ is also bounded. In particular, for any $\mu\in\mathbf{N}$, $(y,s) \in \X\times[0,1]$,
	\begin{equation} \label{iP2Ptauhm}
		| h_m(\mu,y,s)| \leq m\ind{y \in U_m}.
	\end{equation}
	We conclude that $\delta( h_m)$ is well-defined and has by \eqref{eqPathwise} the pathwise expression
	\[
	\delta( h_m) = \int_{\X \times [0,1]}  h_m(\eta-\delta_{(y,s)},y,s) \eta(dy,ds) - \int_{\X} \int_0^1  h_m (\eta,y,s) \lambda(dy) ds.
	\]
	Define $X_t$ as in Theorem~\ref{thmItoFormula} with $h= h_m$ and $X_0=0$. Then $\delta( h_m)=X_1$ and we infer from \eqref{eqIto} that
	\begin{multline}\label{iP2Poin1}
		\phi(\delta( h_m)) - \phi(0) = \int_{\X \times [0,1]} \left(\phi(X_{s-}+ h_m(\eta-\delta_{(y,s)},y,s)) - \phi(X_{s-}) \right) \eta(dy,ds)\\%
		- \int_\X \int_0^1 \phi'(X_s) h_m(\eta,y,s) \lambda(dy) ds \qquad \p\text{-a.s.}
	\end{multline}
	We would now like to take expectations on both sides of \eqref{iP2Poin1}, but in order to do so, we must first show that the terms in question are in $L^1(\p_\eta)$. Start with a simple estimate for $\phi$ that uses Hölder continuity of $\phi'$. Let $a,b \in \R$. Then
	\begin{align}
		|\phi(a+b)-\phi(a)| &= \left| \int_0^1 \phi'(a+ub) b\, du \right| \notag\\
		&\leq |\phi'(a)b| + \int_0^1 |b|\cdot \left|\phi'(a+ub)-\phi'(a)\right| du \notag\\
		&\leq |\phi'(0)b| + |\phi'(a)-\phi'(0)| \cdot |b| + |b| \int_0^1 c_\phi |ub|^{p-1} du \notag\\
		&\leq |\phi'(0)|\cdot|b| + c_\phi |a|^{p-1} \cdot |b| + \frac{c_\phi}{p} |b|^p. \label{iP2phiBound}
	\end{align}
	We are also going to need a bound on $\sup_{s\in[0,1]} \left\{\max{\left\{|X_s|,|X_{s-}(\eta+\delta_{(y,s)})|\right\}}\right\}$. Using \eqref{iP2Ptauhm}, we find the following:
	\begin{align}
		\max{\left\{|X_s|,|X_{s-}(\eta+\delta_{(y,s)})|\right\}}
		&\leq \int_{\X \times [0,1]} m \ind{x \in U_m} \eta(dx,dt) + \int_\X \int_0^1 m \ind{x \in U_m} \lambda(dx)dt \notag\\
		&\leq m \left(\eta(U_m \times [0,1]) + \lambda(U_m)\right) \label{iP2XBound}
	\end{align}
	Using \eqref{iP2phiBound} with $a=0$ and $b=\delta(h_m)$, we get for the LHS in \eqref{iP2Poin1}
	\[
	\E \left|\phi(\delta( h_m))\right| \leq |\phi'(0)|\cdot\E|\delta( h_m)| + \frac{c_\phi}{p} \E|\delta( h_m)|^p < \infty
	\]
	which is finite since $\delta(h_m) \in L^2(\p_\eta)$. To show that the first term on the RHS in \eqref{iP2Poin1} is integrable, we first apply Mecke formula \eqref{eqMecke} to get
	\begin{multline}\label{Rainbow9}
		\E \left|\int_{\X \times [0,1]} \left(\phi(X_{s-}+ h_m(\eta-\delta_{(y,s)},y,s)) - \phi(X_{s-})\right) \eta(dy,ds)\right| \\ \leq
		\E \int_{\X} \int_0^1 \left|\phi(X_{s-}(\eta + \delta_{(y,s)})+ h_m(\eta,y,s)) - \phi(X_{s-}(\eta + \delta_{(y,s)}))\right| \lambda(dy)ds.
	\end{multline}
	Now we combine \eqref{iP2phiBound}, \eqref{iP2XBound} and \eqref{iP2Ptauhm} to get that the RHS of \eqref{Rainbow9} is bounded by
	\begin{align}
		&\E \int_\X \int_0^1 \bigg( |\phi'(0)|\cdot m \ind{y \in U_m} + c_\phi |m \left(\eta(U_m \times [0,1]) + \lambda(U_m)\right)|^{p-1} \cdot m \ind{y \in U_m}\notag\\
		&\hspace{.6\textwidth} + \frac{c_\phi}{p} m^p \ind{y \in U_m} \bigg) \lambda(dy)ds \notag\\
		&\leq m \lambda(U_m) \left(|\phi'(0)| + c_\phi m^{p-1} \E\left[ \left(\eta(U_m \times [0,1]) + \lambda(U_m)\right)^{p-1}\right] + \frac{c_\phi m^{p-1}}{p} \right) < \infty
	\end{align}
	which is finite since $\eta(U_m \times [0,1])$ is a Poisson random variable with parameter $\lambda(U_m)$ and thus all its moments are finite.	This also shows that $\phi(X_{s-}(\eta + \delta_{(y,s)})+ h_m(\eta,y,s)) - \phi(X_{s-}(\eta + \delta_{(y,s)})) \in L^1(\mathbf{N} \times \X \times [0,1])$. The second term on the RHS can be treated by the same method.

	We can now take expectations on both sides of \eqref{iP2Poin1} and apply Mecke formula \eqref{eqMecke} to get:
	\begin{multline}
		\E\phi(\delta( h_m)) = \E \int_{\X}\int_0^1 \phi(X_{s-}(\eta + \delta_{(y,s)})+ h_m(\eta,y,s)) - \phi(X_{s-}(\eta + \delta_{(y,s)})) \lambda(dy)ds\\%
		- \E\int_\X \int_0^1 \phi'(X_s(\eta)) h_m(\eta,y,s) \lambda(dy) ds \label{iP2Poin2}
	\end{multline}
	Now add and subtract the integral of $\phi(X_{s-}(\eta)+ h_m(\eta,y,s)) - \phi(X_{s-}(\eta))$ (which can be shown to be in $L^1(\mathbf{N}\times \X \times [0,1])$ by the same methods as above). We obtain:
	\begin{align}\label{iP2Poin3}
		\E\phi(\delta(h_m)) &= \E \int_{\X}\int_0^1 \phi(X_{s-}(\eta + \delta_{(y,s)})+ h_m(\eta,y,s)) - \phi(X_{s-}(\eta + \delta_{(y,s)})) \notag\\
		&\hspace{.3\textwidth}  - \phi(X_{s-}(\eta)+ h_m(\eta,y,s)) + \phi(X_{s-}(\eta))\lambda(dy)ds \notag\\
		&+\E \int_{\X}\int_0^1 \phi(X_{s-}(\eta)+ h_m(\eta,y,s)) - \phi(X_{s-}(\eta)) - \phi'(X_s(\eta)) h_m(\eta,y,s)\lambda(dy)ds \notag\\
		&=: I_1 + I_2.
	\end{align}
	To deal with $I_1$, note that for $a,b,c \in \R$
	\begin{align}
		|\phi(a+c)-\phi(a)-\phi(b+c)+\phi(b)| &= \left| \int_a^b \phi'(u+c)-\phi'(u) du \right| \\
		&\leq \int_{a \wedge b}^{a \vee b} c_\phi |c|^{p-1} du \\
		&= c_\phi |a-b| \cdot |c|^{p-1}.
	\end{align}
	Hence
	\begin{equation}\label{Rainbow10}
	|I_1| \leq c_\phi \E \int_\X \int_0^1 \left|X_{s-}(\eta + \delta_{(y,s)})-X_{s-}(\eta)\right| \cdot | h_m(\eta,y,s)|^{p-1} \lambda(dy)ds.
	\end{equation}
	Now we bound and rewrite part of the integrand on the RHS of \eqref{Rainbow10} to find
	
	\begin{align}
		&\left|X_{s-}(\eta + \delta_{(y,s)})-X_{s-}(\eta)\right| \notag\\
		&= \left| \int_{\X \times [0,s)}  h_m(\eta+\delta_{(y,s)}-\delta_{(x,t)},x,t) \eta(dx,dt) - \int_\X \int_0^s  h_m(\eta + \delta_{(y,s)},x,t) \lambda(dx)dt \right. \notag\\
		&\phantom{=} \left. -\int_{\X \times [0,s)}  h_m(\eta-\delta_{(x,t)},x,t) \eta(dx,dt) + \int_\X \int_0^s  h_m(\eta,x,t) \lambda(dx)dt \right| \notag\\
		&= \left| \int_{\X \times [0,s)} \MalD_{(y,s)} h_m(\eta-\delta_{(x,t)},x,t) \eta(dx,dt) - \int_\X \int_0^s \MalD_{(y,s)} h_m(\eta,x,t) \lambda(dx)dt \right| \notag\\
		&\leq \int_{\X \times [0,s)} \left|\MalD_{(y,s)} h_m(\eta-\delta_{(x,t)},x,t)\right| \eta(dx,dt) + \int_\X \int_0^s \left|\MalD_{(y,s)} h_m(\eta,x,t)\right| \lambda(dx)dt
	\end{align}
	Multiplying this by $| h_m(\eta,x,t)|^{p-1}$ and taking expectations, after an application of Mecke formula \eqref{eqMecke} we deduce that
	\begin{multline}
		|I_1| \leq c_\phi\E \int_\X \lambda(dy) \int_0^1 ds \int_\X \lambda(dx) \int_0^s dt\, \left|\MalD_{(y,s)} h_m(\eta,x,t)\right| \\ \left( | h_m(\eta + \delta_{(x,t)},y,s)|^{p-1} + | h_m(\eta,y,s)|^{p-1} \right) \lambda(dx)dt\lambda(dy)ds.
	\end{multline}
	Since $|a-b|^{p-1} \leq |a|^{p-1} + |b|^{p-1}$ for all $a,b \in\R$, one has that
	\[
	| h_m(\eta + \delta_{(x,t)},y,s)|^{p-1} \leq |\MalD_{(x,t)} h_m(\eta,y,s)|^{p-1} + | h_m(\eta,y,s)|^{p-1}
	\]
	This implies that
	\begin{align}
		|I_1| &\leq c_\phi\E \int_\X \lambda(dy) \int_0^1 ds \int_\X \lambda(dx) \int_0^s dt \left|\MalD_{(y,s)} h_m(\eta,x,t)\right| \cdot |\MalD_{(x,t)} h_m(\eta,y,s)|^{p-1} \notag\\
		&\phantom{\leq} + 2 c_\phi\E \int_\X \lambda(dy) \int_0^1 ds \int_\X \lambda(dx) \int_0^s dt\, \left|\MalD_{(y,s)} h_m(\eta,x,t)\right| \cdot | h_m(\eta,y,s)|^{p-1}.
	\end{align}
	Now we consider $|I_2|$. For this, note that for $a,b \in \R$,
	\begin{align}
		|\phi(a+b)-\phi(a)-\phi'(a)b| &= \left| \int_0^b \phi'(a+u)-\phi'(a) du \right| \notag\\
		&\leq c_\phi \int_0^{|b|} u^{p-1} du \notag\\
		&= \frac{c_\phi}{p} |b|^{p}.
	\end{align} 
	Applying this to $a=X_s(\eta)$ and $h=h_m(\eta,y,s)$ yields
	\[
	|I_2| \leq \frac{c_\phi}{p} \E \int_\X \int_0^1 | h_m(\eta,y,s)|^{p} \lambda(dy)ds.
	\]
	\sloppy Observe that in the previous computation we implicitly used the fact that, $\p$-a.s., the set $\left\{ s \in [0,1]: X_s(\eta) \neq X_{s-}(\eta)\right\}$ has zero Lebesgue measure.
	We have therefore shown the following inequality:
	\begin{align}
		|\E\phi(\delta( h_m))| &\leq \frac{c_\phi}{p} \E \int_\X \lambda(dy) \int_0^1 ds\, | h_m(\eta,y,s)|^{p} \notag\\
		&\phantom{\leq} + c_\phi\E \int_\X \lambda(dy) \int_0^1 ds \int_\X \lambda(dx) \int_0^s dt\, \left|\MalD_{(y,s)} h_m(\eta,x,t)\right| \cdot |\MalD_{(x,t)} h_m(\eta,y,s)|^{p-1} \notag\\
		&\phantom{\leq} + 2 c_\phi\E \int_\X \lambda(dy) \int_0^1 ds \int_\X \lambda(dx) \int_0^s dt\, \left|\MalD_{(y,s)} h_m(\eta,x,t)\right| \cdot | h_m(\eta,y,s)|^{p-1}.
	\end{align}
	By the construction of $h_m$, the RHS of this inequality is upper bounded by
	\begin{multline}
		\frac{c_\phi}{p} \E \int_\X \lambda(dy) \int_0^1 ds\, | h(\eta,y,s)|^{p} \\
		+ c_\phi\E \int_\X \lambda(dx) \int_0^1 ds \int_\X \lambda(dx) \int_0^s dt \left|\MalD_{(y,s)} h(\eta,x,t)\right| \cdot |\MalD_{(x,t)} h(\eta,y,s)|^{p-1} \\
		+ 2 c_\phi\E \int_\X \lambda(dy) \int_0^1 ds \int_\X \lambda(dx) \int_0^s dt\, \left|\MalD_{(y,s)} h(\eta,x,t)\right| \cdot | h(\eta,y,s)|^{p-1}.
	\end{multline}
	In order to conclude the proof, it remains to show that $\E\phi(\delta( h_m))\rightarrow \E\phi(\delta( h))$, as $m\rightarrow \infty$. For this, we use \eqref{iP2phiBound} with $a=\delta(h)$ and $b=\delta(h_m)-\delta(h)$ to get
	\begin{align}
		\E|\phi(\delta(h_m))-\phi(\delta(h))| &\leq \E |\delta(h_m)-\delta( h)| \left(|\phi'(0)| + c_\phi |\delta(h)|^{p-1} + \frac{c_\phi}{p} |\delta( h_m)-\delta( h)|^{p-1}\right).
	\end{align}
	Using, in order, the Cauchy-Schwarz, Minkowski and Jensen inequalities, we obtain
	\begin{align}
		&\E|\phi(\delta( h_m))-\phi(\delta(h))| \notag\\
		&\hspace{.1\textwidth}\leq \E \left[\left(\delta( h_m)-\delta( h)\right)^2\right]^{1/2} \E\left[\left(|\phi'(0)| + c_\phi |\delta(h)|^{p-1} + \frac{c_\phi}{p} |\delta( h_m)-\delta( h)|^{p-1}\right)^2\right]^{1/2} \notag\\
		&\hspace{.1\textwidth}\leq \E \left[\delta(h_m-h)^2\right]^{1/2} \left(|\phi'(0)| + c_\phi \E\left[ |\delta(h)|^{2} \right]^{(p-1)/2} + \frac{c_\phi}{p} \E\left[|\delta(h_m-h)|^{2}\right]^{(p-1)/2}\right). \label{iP2ineq}
	\end{align}
	By the isometry property \eqref{eqIsomDelta} and the Cauchy-Schwarz inequality, we infer that
	\begin{align}
		&\E\left[ |\delta(h)|^{2} \right] \leq \E \int_{\X}\int_0^1 h(\eta,y,s)^2 \lambda(dy) ds + \E \int_\X \int_0^1 \int_\X \int_0^1 \left(\MalD_{(x,t)} h(\eta,y,s)\right)^2 \lambda(dy) ds \lambda(dx) ds \label{eqWow} \\
		\intertext{and}
		&\E\left[ |\delta(h_m-h)|^{2} \right] \leq \E \int_{\X}\int_0^1 (h_m(\eta,y,s)-h(\eta,y,s))^2 \lambda(dy) ds \notag\\
		&\hspace{.2\textwidth}+ \E \int_\X \int_0^1 \int_\X \int_0^1 \left(\MalD_{(x,t)} h_m(\eta,y,s)-\MalD_{(x,t)} h(\eta,y,s)\right)^2 \lambda(dy) ds \lambda(dx) ds. \label{eqWow2}
	\end{align}
	The RHS of \eqref{eqWow} is finite because $h$ satisfies \eqref{eqDomDelCond}, and the RHS of \eqref{eqWow2} tends to $0$ as $m \rightarrow \infty$ by dominated convergence and the assumption \eqref{eqDomDelCond}. This implies convergence to $0$ on the RHS of \eqref{iP2ineq}.
	
	\noindent To show that the inequality holds in particular for $\phi(x)=|x|^p$ with $p \in (1,2]$, it suffices to note that by Lemma~\ref{lempHolder}, this function satisfies the conditions of the theorem. For $\phi(x)=|x|$, we define $h_m$ as in \eqref{iP2approx}. Then $h_m \in L^1(\mathbf{N}\times \X \times [0,1]) \cap L^2(\mathbf{N}\times \X \times [0,1]) \cap \dom \delta$ and the pathwise representation \eqref{eqPathwise} holds. By the triangle inequality, we have
	\[
	\E |\delta(h_m)| \leq \E \int_{\X \times [0,1]} |h_m(\eta-\delta_{y,s},y,s)| \eta(dy,ds) + \E \int_\X \int_0^1 |h_m(\eta,y,s)| \lambda(dy)ds.
	\]
	By Mecke formula \eqref{eqMecke} and the fact that $|h_m| \leq |h|$, inequality \eqref{eqGeneralPPoin} follows with $h_m$ on the LHS instead of $h$, and the second and third term on the RHS being zero. As we have shown before, $\delta(h_m) \rightarrow \delta(h)$ in $L^2(\p_\eta)$ as $m \rightarrow \infty$, hence the inequality follows for $h$.
\end{proof}

\begin{proof}[Proof of Corollary~\ref{corpPoinGen}]
	\textbf{Step 1.} We first prove the following slightly modified version of \eqref{eqGeneralPPoin} which holds under the weaker assumption $h \in \dom \delta$:
	\begin{align}
		&\left|\E \phi(\delta(h))\right| \notag\\
		&\leq \frac{c_\phi}{p} \E \int_\X\lambda(dy) \int_0^1ds |h(\eta,y,s)|^p  \notag\\
		&+ c_\phi \int_\X\lambda(dy) \int_0^1ds \int_\X \lambda(dx) \int_0^s dt \left(\E\left|\MalD_{(y,s)}h(\eta,x,t)\right|^p\right)^{1/p} \cdot \left(\E\left|\MalD_{(x,t)}h(\eta,y,s)\right|^{p}\right)^{1-1/p} \notag\\
		&+ 2c_\phi \int_\X\lambda(dy) \int_0^1ds \int_\X \lambda(dx) \int_0^s dt \left(\E\left|\MalD_{(y,s)}h(\eta,x,t)\right|^p\right)^{1/p} \cdot \left(\E\left|h(\eta,y,s)\right|^{p}\right)^{1-1/p}. \label{iP0eq1}
	\end{align}
	For $h$ satisfying \eqref{eqDomDelCond}, this is an immediate consequence of Theorem \ref{thmGeneralpPoin} by applying Hölder inequality.
	
	\noindent Now let $h \in L^2(\mathbf{N}\times\X\times[0,1])$ and for $\tau \in (0,1)$ define $P_\tau h$. By Lemma~\ref{lemOrnTrick}, we have that $P_\tau h$ satisfies \eqref{eqDomDelCond}, and so inequality \eqref{iP0eq1} holds when $h$ is replaced by $P_\tau h$. Using \eqref{eqOrnTrick1} and \eqref{eqOrnTrick2}, we deduce
	\begin{align}
		&\left|\E \phi(\delta(P_\tau h)) - \phi(0)\right|\\
		&\leq \frac{c_\phi}{p} \E \int_\X\lambda(dy) \int_0^1ds |h(\eta,y,s)|^p  \notag\\
		&+ c_\phi \int_\X\lambda(dy) \int_0^1ds \int_\X \lambda(dx) \int_0^s dt \left(\E\left|\MalD_{(y,s)}h(\eta,x,t)\right|^p\right)^{1/p} \cdot \left(\E\left|\MalD_{(x,t)}h(\eta,y,s)\right|^{p}\right)^{1-1/p} \notag\\
		&+ 2c_\phi \int_\X\lambda(dy) \int_0^1ds \int_\X \lambda(dx) \int_0^s dt \left(\E\left|\MalD_{(y,s)}h(\eta,x,t)\right|^p\right)^{1/p} \cdot \left(\E\left|h(\eta,y,s)\right|^{p}\right)^{1-1/p}.
	\end{align}
	
	\noindent It remains to show that $\phi(\delta(P_\tau h)) \rightarrow \phi(\delta(h))$ in $L^1(\eta)$ as $\tau \rightarrow 1$. Using arguments similar to the ones in the proof of Theorem \ref{thmGeneralpPoin}, one shows that
	\begin{multline}
	\E |\phi(\delta(P_\tau h)) - \phi(\delta(h))| \\ \leq \norm{\delta(P_\tau h - h)}_{L^2(\p_\eta)} \left(|\phi'(0)| + c_\phi \norm{\delta(h)}_{L^2(\p_\eta)}^{p-1} + \frac{c_\phi}{p} \norm{\delta(P_\tau h - h)}_{L^2(\p_\eta)}^{p-1}\right).
	\end{multline}
	By the isometry property \eqref{eqIsom1}, the second moment $\E \delta(h)^2$ is finite since $h \in \dom \delta$ and Lemma~\ref{lemOrnTrick} implies that $\norm{\delta(P_\tau h - h)}_{L^2(\p_\eta)} \rightarrow 0$ as $\tau \rightarrow 1$.
	
	\noindent\textbf{Step 2.} Let $\in L^2(\mathbf{N}\times\X\times[0,1])$ be predictable. Then $h \in \dom \delta$ and \eqref{iP0eq1} holds. If $t<s$, then by predictability of $h$ one has $h(\eta + \delta_{(y,s)},x,t) = h(\eta,x,t)$ and hence $\MalD_{(y,s)}h(\eta,x,t)=0$. It follows that the second and third terms on the RHS of \eqref{iP0eq1} are 0. Inequality \eqref{eqPhiDelta} ensues.
	
	\noindent\textbf{Step 3.} To prove \eqref{eqPPoin2} for $p \in (1,2]$, let $h(\eta,x,s) := \E [\MalD_{(x,s)}F | \eta_{\X \times [0,s)} ]$. This is predictable in the sense of \eqref{eqPred} (see Section~\ref{secFrame}) and by the Clark-Ocône representation \eqref{eqClarkOcone}, one has that
	\[
	F = \E F + \delta(h) \qquad \asp
	\]
	Define $\phi(x):= |x+\E F|^p-|\E F|^p$. Then $\phi(0)=0$ and the derivative $\phi'$ is $(p-1)$-Hölder continuous with Hölder constant $2^{2-p}$ by Lemma~\ref{lempHolder}. Moreover, it holds that
	\[
	\phi(\delta(h)) = |F|^p - |\E F |^p \qquad \asp
	\]
	We can now apply \eqref{eqPhiDelta} for this choice of $\phi$ and $h$ and inequality \eqref{eqPPoin2} follows.
	
	\noindent For $p=1$, assume the RHS of the inequality \eqref{eqPPoin2} to be finite (else there is nothing to prove). Then $h(\eta,x,s) = \E [\MalD_{(x,s)}F | \eta_{\X \times [0,s)} ] \in \dom \delta \cap L^1(\N \times \X \times [0,1])$ and hence
	\[
	\delta(h) = \int_{\X \times [0,1]} h(\eta-\delta_{(x,s)},x,s) \eta(dx,ds) - \int_\X \int_0^1 h(\eta,x,s) \lambda(dx)ds.
	\]
	Using triangle inequality and Mecke formula, we deduce the chain of inequalities
	\begin{align}
		\E |F| - |\E F| &\leq \E |F-\E F| \notag\\
		&= \E |\delta(h)| \\
		&\leq \E \int_{\X \times [0,1]} |h(\eta-\delta_{(x,t)},x,t)| \eta(dx,dt) + \E \int_{\X} \int_0^1 |h(\eta,x,t)| \lambda(dx)dt \notag\\
		&\leq 2 \E \int_{\X} \int_0^1 |h(\eta,x,t)| \lambda(dx)dt,
	\end{align}
	which concludes the proof.
\end{proof}

\begin{proof}[Proof of Corollary~\ref{corKolgIneq}]
	Let $(U_n)_{n\geq 1} \subset \X$ be an increasing sequence of subsets such that $\bigcup_{n\geq 1}U_n = \X$ and $\lambda(U_n)<\infty$. Define for $n \in \N$
	\[
	h_n(\eta,y,s) := \ind{y \in U_n} [(h(\eta,y,s) \wedge n) \vee (-n)].
	\]
	Clearly $h_n \in L^1(\mathbf{N} \times \X \times [0,1]) \cap L^2(\mathbf{N} \times \X \times [0,1])$ and we can define $P_\tau h_n$ for $\tau \in (0,1)$. We will start by showing \eqref{eqCorKolg} for $P_\tau h_n$ instead of $h$. Using the definition \eqref{eqOrnAlt}, one easily checks that $P_\tau h_n \in L^1(\mathbf{N} \times \X \times [0,1])$. By Lemma~\ref{lemOrnTrick}, it also follows that $P_\tau h_n \in \dom \delta$ and hence we can apply Lemma~\ref{lemNewIBP} and deduce that
	\[
	\E \int_\X \int_0^1 P_\tau h_n(\eta,x,s) \MalD_{(x,s)} G \lambda(dx) ds = \E \left[ G \delta(P_\tau h_n) \right].
	\]
	This implies by Jensen's inequality that for any $p \in [1,2]$
	\begin{equation}\label{iP1001}
	\left|\E \int_\X \int_0^1 P_\tau h_n(\eta,x,s) \MalD_{(x,s)} G \lambda(dx) ds\right| \leq c_G \left(\E |\delta(P_\tau h_n)|^p\right)^{1/p}.
	\end{equation}
	As by Lemma~\ref{lemOrnTrick} the quantity $P_\tau h_n$ also satisfies \eqref{eqDomDelCond}, we can apply Theorem~\ref{thmGeneralpPoin} to $\E |\delta(P_\tau h_n)|^p$ with $\phi(x) = |x|^p$ (which satisfies the required conditions by Lemma~\ref{lempHolder}). After a further application of Hölder inequality, this yields for $p \in (1,2]$ that
	\begin{align}
		&\E |\delta(P_\tau h_n)|^p \notag\\
		&\leq 2^{2-p} \E \int_\X\lambda(dy) \int_0^1ds |P_\tau h_n(\eta,y,s)|^p \notag\\
		&+ p2^{2-p} \int_\X\lambda(dy) \int_0^1ds \int_\X \lambda(dx) \int_0^s dt \E\left[\left|\MalD_{(y,s)}P_\tau h_n(\eta,x,t)\right|^p\right]^{1/p} \cdot \E\left[\left|\MalD_{(x,t)}P_\tau h_n(\eta,y,s)\right|^p\right]^{1-1/p} \notag\\
		&+ p2^{3-p} \int_\X\lambda(dy) \int_0^1ds \int_\X \lambda(dx) \int_0^s dt \E\left[\left|\MalD_{(y,s)}P_\tau h_n(\eta,x,t)\right|^p\right]^{1/p} \cdot \E\left[\left|P_\tau h_n(\eta,y,s)\right|^p\right]^{1-1/p}. \label{eqRainbow}
	\end{align}
	Using Lemma~\ref{lemOrnTrick} and the definition of $h_n$, we find that
	\begin{align}
		\E |\MalD_{(x,s)} P_\tau h_n(\eta,y,u)|^p \leq \E |\MalD_{(x,s)} h(\eta,y,u)|^p \\
		\intertext{and}
		\E |P_\tau h_n(\eta,y,u)|^p \leq \E |h(\eta,y,u)|^p
	\end{align}
	hence \eqref{eqRainbow} is upper bounded by the RHS of \eqref{eqCorKolg}. For $p=1$, we reason as in the proof of Corollary~\ref{corpPoinGen} to deduce that
	\[
	\E |\delta(P_\tau h_n)| \leq 2 \E \int_\X \int_0^1 |P_\tau h_n(\eta,x,s)| \lambda(dx) ds,
	\]
	which is again upper bounded by the RHS of \eqref{eqCorKolg}.
	
	\noindent It remains to take $\tau \rightarrow 1$ and $n \rightarrow \infty$ in the LHS of \eqref{iP1001}. By \eqref{eqOrnAlt}, one sees that $P_\tau h_n(\eta,x,s) = \ind{x \in U_n} P_\tau h_n(\eta,x,s)$, and hence by the Cauchy-Schwarz inequality
	\[
	\E \int_\X \int_0^1 \left|P_\tau h_n(\eta,x,s) - h_n(\eta,x,s)\right| \cdot |\MalD_{(x,s)} G| \lambda(dx) ds \leq 2c_G \lambda(U_n) \norm{P_\tau h_n - h_n}_{L^2(\mathbf{N} \times \X \times [0,1])},
	\]
	which converges to zero as $\tau \rightarrow 1$, by Lemma~\ref{lemOrnTrick}. Therefore inequality \eqref{eqCorKolg} holds with $h_n$ instead of $h$ on the LHS. By dominated convergence, $h_n \MalD G \rightarrow h \MalD G$ in $L^1(\mathbf{N} \times \X \times [0,1])$, thus inequality \eqref{eqCorKolg} holds for $h$.
\end{proof}

\section{Proofs of Theorems~\ref{thmWasserKolg1}, \ref{thmWasserCond} and \ref{thmWasserNon-Cond.}} \label{secProofWK}

Throughout this section, we work with the simplified notation adopted in Remark~\ref{remNotatWK}; recall also the definitions of the distances $d_W$ and $d_K$ given in \eqref{eqdW} and \eqref{eqdK}, and write $\mathcal{H}$ to denote the set of Lipschitz-continuous functions $h:\R \rightarrow \R $ with Lipschitz constant $\norm{h}_L \leq 1$. Let $N \sim \mathcal{N}(0,1)$.

Given $h \in \mathcal{H}$ and $\phi(z) = \p(N \leq z)$, Stein's equation
\[
f'(x) = xf(x) + h(x) - \int_\R h(y) \phi'(y)dy, \qquad x\in \R,
\]
admits a \textbf{canonical solution} $f_h$ satisfying the two properties: (a) $f_h \in \mathcal{C}^1(\R)$ and $\norm{f_h'}_\infty \leq \sqrt{\frac{2}{\pi}}$ and (b) $f_h'$ is Lipschitz-continuous with $\norm{f_h'}_L \leq 2$, see e.g. \cite[Theorem~3, Chapter~6]{PeccReitz} and the references therein. In particular, one has that
\begin{equation} \label{eqW1}
	d_W(F,N) = \sup_{h \in \mathcal{H}} |f_h'(F)-Ff_h(F)|.
\end{equation}
Similarly, for fixed $z \in \R$, the canonical solution $f_z$ to Stein's equation
\[
f_z'(x) = xf_z(x) + \1_{(-\infty,z]}(x) - \phi(z), \qquad x \in \R,
\]
is differentiable everywhere except at $z$, where it is customary to define $f_z'(z) = zf_z(z) + 1 - \phi(z)$. One has that
\begin{itemize}
	\item $\norm{f_z}_\infty \leq \frac{\sqrt{2\pi}}{4}$ and $\norm{f_z'}_\infty \leq 1$
	\item $|xf_z(x)| \leq 1$ for all $x \in \R$ and the function $x \mapsto xf_z(x)$ is non-decreasing for all $z \in \R$,
\end{itemize}
(see e.g. \cite[Lemma~2.3]{CGS11}). We consequently have that
\begin{equation} \label{eqK1}
	d_K(F,N) = \sup_{z \in \R} |f_z'(F)-Ff_z(F)|.
\end{equation}

\begin{proof}[Proof of Thm. \ref{thmWasserKolg1}] 
	We will show the upper bounds in \eqref{eqW2} and \eqref{eqK2} by exploiting the representations \eqref{eqW1} and \eqref{eqK1}. The proof is divided into three steps. First, we are going to derive the first terms on the RHSs of \eqref{eqW2} and \eqref{eqK2} for both Wasserstein and Kolmogorov distances at the same time. Second, we deduce the second term on the RHS of \eqref{eqW2} and, as a last step, we find the second term on the RHS of \eqref{eqK2}. Throughout the proof, we fix $h \in \mathcal{H}$ and $z \in \R$ and consider the corresponding canonical solutions $f_h$ and $f_z$.
	
	\noindent\textbf{Step 1.} Write $f$ for either $f_h$ or $f_z$. Then, $f$ is Lipschitz and there is a version of $f'$ which is bounded. Since $F \in \dom \MalD$ and $f'$ is bounded, the expression
	\[
	\E f'(F)\yint \MalD_yF \ \E[\MalD_yF|\eta_y] \bar{\lambda}(dy)
	\]
	is well-defined. Add and subtract this term to $\E f'(F)- \E Ff(F)$ and bound the resulting first term as follows:
	\begin{align}
		& \left| \E f'(F) - \E f'(F)\yint \MalD_yF\ \E[\MalD_yF|\eta_y] \bar{\lambda}(dy) \right| \notag\\
		&\leq  \E |f'(F)| \cdot \left| 1- \yint \MalD_yF\ \E[\MalD_yF | \eta_y] \bar{\lambda}(dy) \right|\notag\\
		&\leq \norm{f'}_\infty\, \E \left| 1- \yint \MalD_yF\ \E[\MalD_yF | \eta_y] \bar{\lambda}(dy) \right|.
	\end{align}
	As $\norm{f_h'}_\infty \leq \sqrt{\frac{2}{\pi}}$ and $\norm{f_z'}_\infty \leq 1$, the bounds follow.
	
	\noindent We are left to deal with
	\[
	\left|\E F f(F) - \E f'(F)\yint \MalD_yF \ \E[\MalD_yF|\eta_y] \bar{\lambda}(dy)\right|.
	\]
	Since $f$ is Lipschitz and $F\in L^2(\p_\eta)$, it follows that $f(F)\in L^2(\p_\eta)$. Hence by Lemma \ref{lemCondCov}
	\[
	\E F f(F) = \cov(F,f(F)) = \E \yint \E[\MalD_yF|\eta_y]\ \E[\MalD_yf(F)|\eta_y] \bar{\lambda}(dy).
	\]
	Again by Lipschitzianity of $f$, it follows that $|\MalD_yf(F)| \leq |\MalD_yF|$ and hence an application of Cauchy-Schwarz inequality justifies that
	\[
	\E F f(F) = \E \yint \E[\MalD_yF|\eta_y] \MalD_yf(F) \bar{\lambda}(dy).
	\]
	Therefore we are left to bound
	\begin{equation}\label{iP3Rest}
		\E \yint \left|\E[\MalD_yF|\eta_y]\right| \cdot \left| f'(F)\MalD_yF - \MalD_yf(F) \right| \bar{\lambda}(dy) .
	\end{equation}
	\textbf{Step 2.} To bound \eqref{iP3Rest} for the Wasserstein distance, we use an argument borrowed from the proof of \cite[Theorem~3.1]{BOPT} to upper bound $|f_h(b)-f_h(a)-f_h'(a)(b-a)|$. Let $a,b \in \R$. Then by the properties stated above, $f_h$ is Lipschitz and hence
	\[
	|f_h(b)-f_h(a)-f_h'(a)(b-a)| \leq |f_h(b)-f_h(a)| + |f_h'(a)(b-a)| \leq 2\sqrt{\tfrac{2}{\pi}}|b-a|.
	\]
	But at the same time by Lipschitzianity of $f_h'$,
	\[
	|f_h(b)-f_h(a)-f_h'(a)(b-a)| = \left|\int_a^b f_h'(x)-f_h'(a)dx\right| \leq 2 (b-a)^2.
	\]
	We deduce that for any $q \in [1,2]$
	\begin{equation}
		|f_h(b)-f_h(a)-f_h'(a)(b-a)| \leq 2 \min\{|b-a|,(b-a)^2\} \leq 2 |b-a|^q.
	\end{equation}
	It follows that
	\begin{align}
		\left| f_h'(F)\MalD_yF - \MalD_yf_h(F) \right| &= |f_h(F(\eta + \delta_y)) - f_h(F(\eta)) - f_h'(F(\eta))(F(\eta + \delta_y)-F(\eta))| \notag\\
		&\leq 2 |\MalD_yF|^q.
	\end{align}
	and therefore \eqref{iP3Rest} is bounded by
	\[
	2\, \E \yint \left|\E[\MalD_yF|\eta_y]\right| |\MalD_yF|^q \bar{\lambda}(dy).
	\]
	The required bound in the Wasserstein distance follows suit.
	
	\noindent\textbf{Step 3.} We reason as in the proof of \cite[Theorem~1.12]{LachPeccYang} to conclude
	\[
	\left| f_z'(F)\MalD_yF - \MalD_yf_z(F) \right| \leq \MalD_yF \cdot \MalD_y (Ff_z(F) + \1_{\{F > z\}})
	\]
	Thus \eqref{iP3Rest} is upper bounded by
	\[
	\E \yint \left|\E[\MalD_yF|\eta_y]\right| \MalD_yF \cdot \MalD_y \left(Ff_z(F) + \1_{\{F > z\}}\right) \bar{\lambda}(dy).
	\]
	The desired bound now follows by taking the supremum over all $z \in \R$.
\end{proof}

\begin{proof}[Proof of Theorem~\ref{thmWasserCond}]
	The proof will be split into two steps.
	
	\noindent\textbf{Step 1.} We start by showing that under the conditions of the theorem,
	
	\begin{equation}
		\E \left| 1- \yint \MalD_yF\, \E[\MalD_yF | \eta_y] \bar{\lambda}(dy) \right| \leq \beta_1 + \beta_2.
	\end{equation}
	
	\noindent We can assume that $\E F = 0$ and $\sigma =1$ (indeed, the result then follows since $\MalD \hat{F} = \sigma^{-1} \MalD F$).
	
	\noindent For ease of notation, define
	\[
	G := \yint \MalD_yF\, \E[\MalD_yF | \eta_y] \bar{\lambda}(dy) - 1.
	\]
	As a first step, we show that $\E G = 0$.
	As $F \in \dom \MalD$, we can use Fubini's theorem and Lemma~\ref{lemCondCov} to deduce
	\[
	\E \yint \MalD_yF\, \E[\MalD_yF | \eta_y] \bar{\lambda}(dy) = \E \yint \E[\MalD_yF | \eta_y]^2 \bar{\lambda}(dy) = \var(F) = 1.
	\]
	Now by the modification \eqref{eqPPoin1} of Corollary~\ref{corpPoinGen} given in Remark~\ref{remPPoin} and since $G \in L^1(\p_\eta)$, we have for $p \in [1,2]$,
	\[
	\E |G| \leq \left(\E |G|^p\right)^{1/p} \leq \left(2^{2-p} \E \yint \E [|\MalD_{x}G| | \eta_x ]^p \bar{\lambda}(dx)\right)^{1/p}.
	\]
	Let us now study the term
	\[
	\MalD_xG = \MalD_x \yint \MalD_yF\, \E[\MalD_yF | \eta_y] \bar{\lambda}(dy)
	\]
	and define
	\[
	h(\eta,y) := \MalD_yF\, \E[\MalD_yF | \eta_y].
	\]
	We want to  show that since $h \in L^1(\mathbf{N} \times \Y)$,
	\begin{equation}\label{eqRainbow2}
	\left|\MalD_x \yint h(\eta,y) \bar{\lambda}(dy)\right| \leq  \yint \left|\MalD_x h(\eta,y)\right| \bar{\lambda}(dy).
	\end{equation}
	and to do so we need an argument taken from the proof of \cite[Proposition~4.1]{LPS14}: Assume the RHS of \eqref{eqRainbow2} to be finite (if it is not, then the inequality \eqref{eqRainbow2} is trivially true). Then
	\[
	\yint |h(\eta + \delta_x,y)| \bar{\lambda}(dy) \leq \yint \left|\MalD_x h(\eta,y)\right| + |h(\eta,y)| \bar{\lambda}(dy) < \infty
	\]
	and hence
	\[
	\MalD_x \yint h(\eta,y) \bar{\lambda}(dy) =  \yint \MalD_x h(\eta,y) \bar{\lambda}(dy).
	\]
	The inequality \eqref{eqRainbow2} follows. We have therefore shown
	\begin{align}
		\E |G| &\leq \left(2^{2-p} \E \yint \E \left[\left. \yint |\MalD_x h(\eta,y)| \bar{\lambda}(dy) \right| \eta_x \right]^p \bar{\lambda}(dx)\right)^{1/p} \notag\\
		&= \left(2^{2-p}  \yint \E\left(\yint \E \big[|\MalD_x h(\eta,y)|\big| \eta_x \big] \bar{\lambda}(dy) \right)^p \bar{\lambda}(dx)\right)^{1/p}, \label{Rainbow11}
	\end{align}
	where the second line follows from Tonelli's theorem. By Minkowski's integral inequality,
	\begin{equation}\label{Rainbow12}
	\E\left(\yint \E \big[|\MalD_x h(\eta,y)|\big| \eta_x \big] \bar{\lambda}(dy) \right)^p%
	\leq \left(\yint \E\big[\E \big[|\MalD_x h(\eta,y)|\big| \eta_x \big]^p\big]^{1/p} \bar{\lambda}(dy) \right)^p.
	\end{equation}
	By the formula \eqref{eqProductFormula} for products,
	\begin{align}
		\MalD_x h(\eta,y) &= \MalD_x \left(\MalD_y F\, \E[\MalD_y F | \eta_y]\right) \notag\\
		&= \DD_{x,y} F \cdot \E[\MalD_yF | \eta_y] + \MalD_yF \cdot \MalD_x \E[\MalD_yF | \eta_y] + \DD_{x,y} F \cdot \MalD_x \E[\MalD_yF | \eta_y].
	\end{align}
	Since $\E[\MalD_yF | \eta_y]$ depends on $\eta$ only through $\eta_y$, it follows that $\MalD_x \E[\MalD_yF | \eta_y] = 0$ whenever $x \geq y$. Moreover, since $F \in L^2(\p_\eta)$ implies that $\MalD_yF, \DD_{x,y} F \in L^1(\p_\eta)$ by \cite[Theorem~1.1]{LPen1}, if $x<y$, we can put the add-one cost operator inside the conditional expectation. Therefore
	\[
	\MalD_x \E[\MalD_yF | \eta_y] = \1_{\{x<y\}} \E[ \DD_{x,y} F | \eta_y].
	\]
	By Minkowski's norm inequality, it follows that
	\begin{align}
		\E\left[\E \big[|\MalD_x h(\eta,y)|\big| \eta_x \big]^p\right]^{1/p}%
		&\leq \E\left[\E \big[|\DD_{x,y} F \cdot \E[\MalD_yF | \eta_y]|\big| \eta_x \big]^p\right]^{1/p}\notag\\
		&\phantom{\leq}+ \E\left[\E \big[|\MalD_yF \cdot \1_{\{x<y\}} \E[ \DD_{x,y} F | \eta_y]|\big| \eta_x \big]^p\right]^{1/p}\notag\\
		&\phantom{\leq}+ \E\left[\E \big[|\DD_{x,y} F \cdot \1_{\{x<y\}} \E[ \DD_{x,y} F | \eta_y]|\big| \eta_x \big]^p\right]^{1/p}.\label{eqThing}
	\end{align}
	By the properties of conditional expectations and splitting the first term into two parts, \eqref{eqThing} is bounded by
	\begin{align}
		&\ind{y \leq x} \E\left[|\E[|\DD_{x,y} F || \eta_x ]^p \cdot \E[\MalD_yF | \eta_y]|^p\right]^{1/p} \notag\\
		&+\ind{x<y} \E\left[\E \big[\E[|\DD_{x,y} F| |\eta_y] \cdot |\E[\MalD_yF | \eta_y]|\big| \eta_x \big]^p\right]^{1/p} \notag\\
		&+\ind{x<y} \E\left[\E \big[\E[|\MalD_yF||\eta_y] \cdot |\E[ \DD_{x,y} F | \eta_y]|\big| \eta_x \big]^p\right]^{1/p}\notag\\
		&+\ind{x<y} \E\left[\E \big[\E[|\DD_{x,y} F||\eta_y] \cdot |\E[ \DD_{x,y} F | \eta_y]|\big| \eta_x \big]^p\right]^{1/p}. \label{eqThing2}
	\end{align}
	By an application of Jensen's inequality \eqref{eqThing2} is now bounded by
	\begin{align}
		&\E\left[|\E[\MalD_yF | \eta_y]|^p \cdot \E\big[|\DD_{x,y} F |\big| \eta_{x \vee y} \big]^p\right]^{1/p} \notag\\
		&+\ind{x<y} \E\left[\E\big[|\MalD_yF|\big|\eta_y\big]^p \cdot \big|\E\big[ \DD_{x,y} F \big| \eta_y\big]\big|^p\right]^{1/p}\notag\\
		&+\ind{x<y} \E\left[\E\big[|\DD_{x,y} F|\big|\eta_y\big]^{2p} \right]^{1/p}. \label{Rainbow13}
	\end{align}
	Plugging the conclusion from \eqref{Rainbow12} - \eqref{Rainbow13} into \eqref{Rainbow11} and applying Minkowski's norm inequality again yields
	\begin{align}
		\E |G|%
		&\leq2^{2/p-1} \left( \yint \left(\yint \E\left[\big|\E[\MalD_yF | \eta_y]\big|^p \cdot \E\big[|\DD_{x,y} F |\big| \eta_{x \vee y} \big]^p\right]^{1/p} \bar{\lambda}(dy) \right)^p \bar{\lambda}(dx)\right)^{1/p} \notag\\
		&\phantom{\leq} + 2^{2/p-1} \left( \yint \left(\yint \ind{x<y} \E\left[\E\big[|\MalD_yF|\big|\eta_y\big]^p \cdot \big|\E[ \DD_{x,y} F | \eta_y]\big|^p\right]^{1/p} \bar{\lambda}(dy) \right)^p \bar{\lambda}(dx)\right)^{1/p}\notag\\
		&\phantom{\leq} + 2^{2/p-1} \left( \yint \left(\yint \ind{x<y} \E\left[\E\big[|\DD_{x,y} F|\big|\eta_y\big]^{2p} \right]^{1/p} \bar{\lambda}(dy) \right)^p \bar{\lambda}(dx)\right)^{1/p}
	\end{align}
	which is in turn bounded by
	\begin{align}
		&2^{2/p} \left( \yint \left(\yint \E\left[\E\big[|\MalD_yF|\big|\eta_y\big]^p \cdot \E\big[|\DD_{x,y} F |\big| \eta_{x \vee y} \big]^p\right]^{1/p} \bar{\lambda}(dy) \right)^p \bar{\lambda}(dx)\right)^{1/p} \notag\\
		&+ 2^{2/p-1} \left( \yint \left(\yint \ind{x<y} \E\left[\E\big[|\DD_{x,y} F|\big|\eta_y\big]^{2p} \right]^{1/p} \bar{\lambda}(dy) \right)^p \bar{\lambda}(dx)\right)^{1/p}.
	\end{align}
	The result follows by an application of the Cauchy-Schwarz inequality.
	
	\noindent\textbf{Step 2.} We want to show that
	\begin{equation}\label{eqRainbow3}
	\E \yint \left| \E[\MalD_yF | \eta_y] \right| \cdot |\MalD_yF|^q \bar{\lambda}(dy) \leq \beta_3 + \beta_4.
	\end{equation}
	Again it suffices to show \eqref{eqRainbow3} for $F$ with $\E F = 0$ and $\E F^2 = 1$. Assume the $\beta_3,\beta_4$ to be finite (otherwise there is nothing to prove). Then, in particular
	\[
	\E \yint \left| \E \left[\left. \MalD_yF \right|\eta_y\right]\right|^{q+1} \bar{\lambda}(dy) < \infty
	\]
	and we can add and subtract this term to the LHS of \eqref{eqRainbow3}, yielding $\beta_3$ and the following rest term:
	\begin{align}
		&\left|\E \yint \left| \E[\MalD_yF | \eta_y] \right| \cdot |\MalD_yF|^q \bar{\lambda}(dy) - \E \yint \left| \E \left[\left. \MalD_yF \right|\eta_y\right]\right|^{q+1} \Bar{\lambda}(dy)\right|\notag\\
		&=\left|\E \yint \left| \E[\MalD_yF | \eta_y] \right| \cdot \E[|\MalD_yF|^q | \eta_y] \bar{\lambda}(dy) - \E \yint \left| \E \left[\left. \MalD_yF \right|\eta_y\right]\right|^{q+1} \Bar{\lambda}(dy)\right|\notag\\
		&\leq \E \yint|\E[\MalD_yF | \eta_y]| \cdot \big|\E[|\MalD_yF|^q | \eta_y] - \left| \E \left[\left. \MalD_yF \right|\eta_y\right]\right|^{q}\big| \Bar{\lambda}(dy), \label{iP4Bound1}
	\end{align}
	where the equality is justified by the fact that $\E [G|\eta_y]$ is defined for all $G \in L^0(\p_\eta)$ which are integrable or non-negative (see the definition \eqref{eqCondExp}).

	\noindent The term $\E[|\MalD_yF|^q | \eta_y] - \left| \E \left[\left. \MalD_yF \right|\eta_y\right]\right|^{q}$ can be rewritten as
	\[
	\tilde{\E} |g_y(\chi)|^q - |\tilde{\E} g_y(\chi)|^q,
	\]
	where $g_y(\chi) = \MalD_yF(\eta_y + \chi)$ and $\tilde{\E}$ is the expectation with respect to $\chi$, a Poisson measure on the space $\{x \in \Y: x \geq y\}$ with intensity $\lambda \otimes ds$, independent of $\eta$. As $F \in \dom \MalD$, it can be shown that $g_y \in L^2(\p_\chi)$ for a.e.~$y \in \Y$.
	
	\noindent By \eqref{eqPPoin2} in Corollary~\ref{corpPoinGen},
	\begin{equation}\label{Rainbow4}
	\tilde{\E} |g_y(\chi)|^q - |\tilde{\E} g_y(\chi)|^q \leq 2^{2-q} \tilde{\E} \yint \ind{y \leq x}  \big|\tilde{\E}[\MalD_xg_y(\chi)|\chi_x]\big|^q \bar{\lambda}(dx).
	\end{equation}
	Plugging \eqref{Rainbow4} into \eqref{iP4Bound1}, we deduce that the RHS of \eqref{iP4Bound1} is upper bounded by
	\[
	2^{2-q}\E \yint|\E[\MalD_yF | \eta_y]| \yint \ind{y \leq x} \tilde{\E} \big|\tilde{\E}[\MalD_xg_y(\chi)|\chi_x]\big|^q \bar{\lambda}(dx) \Bar{\lambda}(dy).
	\]
	Since $\chi$ and $\eta$ are independent and have the same intensity measure,
	\[
	\ind{y \leq x} |\E[\MalD_yF | \eta_y]|\cdot \tilde{\E} \big|\tilde{\E}[\MalD_xg_y(\chi)|\chi_x]\big|^q
	\]
	has the same law as
	\[
	\ind{y \leq x} |\E[\MalD_yF | \eta_y]|\cdot \E\big[|\E[\DD_{x,y} F|\eta_x]|^p\big|\eta_y\big].
	\]
	This implies finally that the RHS of \eqref{iP4Bound1} is upper bounded by
	\[
	2^{2-q}\E \yint\yint\ind{y \leq x}|\E[\MalD_yF | \eta_y]| \cdot |\E[\DD_{x,y} F|\eta_x]|^q \bar{\lambda}(dx) \Bar{\lambda}(dy)
	\]
	and the result follows by the Cauchy-Schwarz inequality.
\end{proof}

\begin{proof}[Proof of Theorem~\ref{thmWasserNon-Cond.}]
	It suffices to prove this result for $F$ such that $\E F =0$ and $\sigma = 1$. If Theorem~\ref{thmWasserNon-Cond.} holds for all $\sigma$-finite spaces $\X$, it holds in particular for the space $\X \times [0,1]$. In fact, it suffices to prove Theorem~\ref{thmWasserNon-Cond.} for functionals $F \in L^2(\p_\eta) \cap \dom \MalD$, where $\eta$ is a $(\X \times [0,1],\lambda \otimes ds)$-Poisson measure. Indeed, as discussed in Section~\ref{secFrame}, we can regard any functional $F \in L^2(\p_\chi)$, with the $(\X,\lambda)$-Poisson measure $\chi$, as a functional of $\eta$ without changing the law of $F$ or its add-one costs. If we replace $\X$ by $\X \times [0,1]$ and $\chi$ by $\eta$ in the terms $\gamma_1,...,\gamma_7$, the expressions do not change either since for any $F \in L^2(\p_\chi) \cap \dom \MalD$, the integrands do not depend on time. For the rest of this proof, we let thus $F \in L^2(\p_\eta) \cap \dom \MalD$ and recall the notation from Remark~\ref{remNotatWK} where $\Y := \X \times [0,1]$ and $\bar{\lambda}:=\lambda \otimes ds$ are explicitly defined.
	
	We divide this proof into two steps: first, we discuss the bound on the Wasserstein distance, second, we show the bound on the Kolmogorov distance.
	
	\textbf{Step 1.} By a combination of Theorems~\ref{thmWasserKolg1}~and~\ref{thmWasserCond}, we see that
	\begin{align}
		d_W(G,N) &\leq \sqrt{\frac{2}{\pi}}\, \E \left| 1- \yint \MalD_yF\, \E[\MalD_yF | \eta_y] \bar{\lambda}(dy) \right| + 2\, \E \yint \left| \E[\MalD_yF | \eta_y] \right| \cdot |\MalD_yF|^q \bar{\lambda}(dy) \notag\\
		&\leq \beta_1 + \beta_2 + 2\, \E \yint \left| \E[\MalD_yF | \eta_y] \right| \cdot |\MalD_yF|^q \bar{\lambda}(dy).
	\end{align}
	To bound $\beta_1$ and $\beta_2$, we simply apply Jensen's inequality and bound $\ind{x<y}$ by $1$. This gives
	\[
	\beta_1 + \beta_2 \leq \gamma_1 + \gamma_2.
	\]
	For the second term, apply Hölder's inequality to deduce
	\[
	\E \yint \left| \E[\MalD_yF | \eta_y] \right| \cdot |\MalD_yF|^q \bar{\lambda}(dy) \leq \yint \left(\E\left| \E[\MalD_yF | \eta_y] \right|^{q+1}\right)^{1/(q+1)} \cdot \left(\E|\MalD_yF|^{q+1}\right)^{1-1/(q+1)} \bar{\lambda}(dy).
	\]
	A further application of Jensen's inequality yields the result.
	
	\textbf{Step 2.} We start from inequality \eqref{eqK2} in Theorem~\ref{thmWasserKolg1}. The first term was dealt with in Step~1 of this proof. Let us deal with the second term. Define
	\[
	h(\eta,y):= \MalD_yF \cdot |\E[\MalD_yF|\eta_y]|
	\]
	and for $z \in \R$,
	\[
	Z_z:= F f_z(F) + \ind{F>z}.
	\]
	The term we want to bound is thus given by
	\[
	\sup_{z \in \R} \E \yint h(\eta,y) \MalD_y Z_z \bar{\lambda}(dy).
	\]
	Since $F \in \dom \MalD$, it follows by Cauchy-Schwarz inequality that $h \in L^1(\mathbf{N} \times \Y)$. Moreover, by the properties of $f_z$, we have $|Z_z| \leq 2$ for all $z \in \R$. Hence we can apply Corollary~\ref{corKolgIneq} and get for any $p \in [1,2]$
	\begin{align}
		\E \yint h(\eta,y) \MalD_y Z_z \bar{\lambda}(dy)%
		&\leq 2 \left(2^{2-p} \E \yint |h(\eta,y)|^p \bar{\lambda}(dy) \right. \notag\\
		&+ p2^{2-p} \E\yint \yint \left|\MalD_{y}h(\eta,x)\right|^p \bar{\lambda}(dy) \bar{\lambda}(dx) \notag\\
		&+ \left. p2^{3-p} \yint \yint \ind{x<y} \left(\E\left|\MalD_{y}h(\eta,x)\right|^p\right)^{1/p} \cdot \left(\E\left|h(\eta,y)\right|^{p}\right)^{1-1/p} \bar{\lambda}(dy)\bar{\lambda}(dx)\right)^{1/p}\notag\\
		&= (I_1 + I_2 + I_3)^{1/p} \notag\\
		&\leq (I_1)^{1/p} + (I_2)^{1/p} + (I_3)^{1/p},
	\end{align}
	since $|a+b|^{1/p} \leq |a|^{1/p} + |b|^{1/p}$.
	
	Let us first look at $I_1$. By applying the Cauchy-Schwarz and Jensen inequalities, it follows that
	\begin{align}
		I_1 &= 4\, \E \yint |\MalD_yF|^p \cdot |\E[\MalD_yF|\eta_y]|^p \bar{\lambda}(dy) \notag\\
		&\leq 4  \yint \E\left[|\MalD_yF|^{2p}\right] \bar{\lambda}(dy).
	\end{align}
	Now we deal with $I_2$. By the formula \eqref{eqProductFormula},
	\[
	\MalD_yh(\eta,x) = \DD_{x,y} F \cdot \MalD_y |\E[\MalD_xF|\eta_x]| + \DD_{x,y} F \cdot |\E[\MalD_xF|\eta_x]| + \MalD_x F \cdot \MalD_y|\E[\MalD_xF|\eta_x]|.
	\]
	By Minkowski's norm inequality,
	\begin{align}
		\left(\E\yint \yint \left|\MalD_{y}h(\eta,x)\right|^p \bar{\lambda}(dy) \bar{\lambda}(dx)\right)^{1/p}%
		&\leq \left(\E\yint \yint \left|\DD_{x,y} F \cdot \MalD_y |\E[\MalD_xF|\eta_x]|\right|^p \bar{\lambda}(dy) \bar{\lambda}(dx)\right)^{1/p} \notag\\
		&+ \left(\E\yint \yint \left|\DD_{x,y} F \cdot |\E[\MalD_xF|\eta_x]|\right|^p \bar{\lambda}(dy) \bar{\lambda}(dx)\right)^{1/p} \notag\\
		&+ \left(\E\yint \yint \big|\MalD_x F \cdot \MalD_y|\E[\MalD_xF|\eta_x]|\big|^p \bar{\lambda}(dy) \bar{\lambda}(dx)\right)^{1/p}.
	\end{align}
	By the triangle inequality, and as in the proof of Theorem~\ref{thmWasserCond},
	\[
	\big|\MalD_y |\E[\MalD_xF|\eta_x]|\big| \leq \ind{y<x} |\MalD_y\E[\MalD_xF|\eta_x]| \leq |\E[\DD_{x,y} F|\eta_x]|.
	\]
	By the Cauchy-Schwarz and Jensen inequalities,
	\begin{align}
		(I_2)^{1/p}%
		&\leq (4p)^{1/p} \left(\yint \yint \E\left[|\DD_{x,y} F|^{2p}\right] \bar{\lambda}(dy) \bar{\lambda}(dx)\right)^{1/p} \notag\\
		&+ 2^{2/p+1}p^{1/p} \left(\yint \yint \E\left[|\DD_{x,y} F|^{2p}\right]^{1/2} \cdot \E\left[|\MalD_{x} F|^{2p}\right]^{1/2} \bar{\lambda}(dy) \bar{\lambda}(dx)\right)^{1/p}.
	\end{align}
	Lastly, we look at the term $I_3$. Since $\E[\MalD_xF|\eta_x]$ depends only on $\eta_x$, we get
	\[
	\ind{x<y} \MalD_y |\E[\MalD_xF|\eta_x]| = 0.
	\]
	Hence
	\[
	\ind{x<y} \big(\E\left|\MalD_{y}h(\eta,x)\right|^p\big)^{1/p} = \ind{x<y} \left(\E\left|\DD_{x,y} F \cdot  \E[\MalD_xF|\eta_x]\right|^p\right)^{1/p}.
	\]
	Applying Cauchy-Schwarz and Jensen again yields
	\[
	I_3 \leq 8p \yint \yint \left(\E|\DD_{x,y} F|^{2p}\right)^{\frac{1}{2p}} \cdot  \left(\E|\MalD_xF|^{2p}\right)^{\frac{1}{2p}} \cdot \left(\E\left|\MalD_yF\right|^{2p}\right)^{1-1/p} \bar{\lambda}(dy)\bar{\lambda}(dx).
	\]
	This inequality concludes the proof.
\end{proof}

\section{Online Nearest Neighbour Graph}\label{secPONNG}
The technical proofs in this section sometimes use ideas and techniques close to \cite[Section~3.4]{Pen05} and \cite{Wade2009} (and previous papers). We will discuss the connections with this part of the literature as the proofs unfold.

Throughout this section, we work under the setting of Section~\ref{secONNG}. We will adopt the following notation: for a measurable subset $A \subset \R^d$, we will write $\eta_A$ for $\eta_{|A \times [0,1]}$. Moreover, we need to adapt the definition of `generic' to sets in the marked space $\R^d \times [0,1]$. In this section only, we call a finite set $\mu \subset \R^d \times [0,1]$ \textbf{generic} if
\begin{itemize}
	\item all projections of $\mu$ onto $\R^d$ are distinct;
	\item all pairwise distances between projections of $\mu$ onto $\R^d$ are distinct;
	\item all projections of $\mu$ onto $[0,1]$ are distinct.
\end{itemize}
Note that for any convex body $D_0 \subset \R^d$, the restriction of an $(\R^d \times [0,1],dx \otimes ds)$-Poisson measure $\eta$ to $D_0 \times [0,1]$ has generic support. We call $\mu$ \textbf{generic with respect to} a point (or points) $(x,s),(y,u) \in \R^d \times [0,1]$ if $\mu \cup \{(x,s),(y,u)\}$ is generic.

We start with a short discussion of the properties of the space $H$ defined in Section~\ref{secONNG}. Since $H$ has non-empty interior, there exist $\delta > 0$ and $y_0 \in H$ such that $B^d(y_0,\delta) \subset H$. Fix these $\delta$ and $y_0$ throughout this section. For $\epsilon > 0$ define
\[
H_\epsilon := \{x \in H: \dist(x,\partial H) > \epsilon\},
\]
where $\dist$ denotes the Euclidean distance and $\partial H$ the boundary of $H$. For small $\epsilon$, this set is non-empty. By \cite[(3.19)]{HLS16}, one has
\[
|H\setminus H_\epsilon| \leq |H + B^d(0,\epsilon)| - |H|,
\]
where the sum is the Minkowski sum of sets. By Steiner's formula (cf. \cite[(14.5)]{StochIntGeo}), it follows that there is a constant $\beta_H>0$ such that
\begin{equation}\label{eqHcond}
	|H\setminus H_\epsilon| \leq \beta_H\, \epsilon.
\end{equation}

\subsection{Moment estimates of conditional expectations of add-one costs}\label{subsecCondExp}

The goal of this subsection is to prove the following proposition:
\begin{prop}\label{propONNGcondBounds}
	Assume the conditions in the statement of Theorem~\ref{thmONNG}. Let $\alpha>0$ and $r \geq 1$. Then for all $t\geq 1$ and all $(x,s),(y,u) \in tH\times(0,1]$ with $s<u$, it holds that
	\begin{align}
		&\E\left[ \E\big[ |\MalD_{(x,s)}\fat| \big| \eta_{|tH \times [0,s)} \big]^r \right]^{1/r} \leq c_1  (s^{-\alpha/d} \wedge t^\alpha) \label{eqONNGdxBound} \\
		&\E\left[ \E\big[ |\DD_{(x,s),(y,u)}\fat| \big| \eta_{|tH \times [0,u)} \big]^r \right]^{1/r} \leq c_1 (u^{-\alpha/d}\wedge t^\alpha). \label{eqONNGdxybound}\\
		&\p\left( \E\big[ |\DD_{(x,s),(y,u)}\fat| \big| \eta_{|tH \times [0,u)} \big] \neq 0 \right) \leq C_2 \exp(-c_2 u|x-y|^d). \label{eqONNGpBound}
	\end{align}
	where $c_1>0$ is a constant depending on $\alpha,r$ and $\epsilon$ and $c_2,C_2>0$ depend only on $d$ and $H$.
\end{prop}
\begin{rem}
	The bound \eqref{eqONNGdxBound} can be compared to Lemma~4.2 in \cite{Wade2009}, with the difference that we work in a Poisson setting and consider general moments $r\geq 1$, whereas \cite{Wade2009} considers uniform random variables and $r=2$.
\end{rem}
We start with two technical lemmas.

\begin{lem}\label{lemGeo1}
	Let $\epsilon, \ell>0$. There exists $0<q<1$ depending only on $\ell, \epsilon$ and $d$ such that for every convex body $D_0 \subset \R^d$ satisfying $\diam(D_0)\leq \ell$ and $B^d(y_0,\epsilon) \subset D_0$ for some $y_0 \in D_0$, for every $w \in D_0$ and $0<s\leq \ell$, there is a point $w' \in D_0$ with
	\[
	B^d(w',qs) \subset D_0 \cap B^d(w,s).
	\]
	In particular, this implies that
	\begin{equation}\label{AppelsandPears}
	\inf_{\substack{w \in D_0 \\ 0<s\leq \ell}} \frac{|D_0 \cap B^d(w,s)|}{s^d} \geq c_1(\epsilon,\ell,d)
	\end{equation}
	with $c_1(\epsilon,\ell,d)>0$ a constant depending on $\epsilon,\ell$ and $d$.
\end{lem}
In words, this lemma means that there is a constant rate $q$ such that for any not-too-big ball intersecting $D_0$, we can find a smaller ball shrunken by the factor $q$ within the intersection. The ratio $q$ at which the ball is shrunk depends only on a lower bound for the in-radius of $D_0$ and an upper bound for the diameter of $D_0$. This will become important in the proof of Proposition~\ref{propONNGcondBounds}, where $D_0$ is a probabilistic object, but has deterministic upper and lower bounds.

\begin{rem}
	While geometric lemmas in the spirit of Lemma~\ref{lemGeo1} are already shown in \cite{Pen05,Wade2009} and appear in some form in many papers on stochastic geometry, we need a version that allows for fine control over the constants, in order to be able to state our results in the full generality of convex sets.
\end{rem}
\begin{proof}
	Since the ball $B^d(y_0,\epsilon)$ is a subset of $D_0$ and $D_0$ is closed and convex, for any $w \in D_0$, the convex hull of $w \cup \bb^d(y_0,\epsilon)$ is inside $D_0$. If $d \geq 2$ and for $w \notin \bb^d(y_0,\epsilon)$, this is a cone-like structure with angular radius $\omega_w = \arcsin(\epsilon/|w-y_0|)$. Since $\diam(D_0) \leq \ell$, it follows that
	\[
	\frac{\pi}{4} \wedge \inf_{w \in D_0} \omega_w \geq \frac{\pi}{4} \wedge \arcsin\left(\frac{\epsilon}{\ell}\right) =: \theta.
	\]
	
	For every $w \in D_0$, let $C(w)$ be the closed cone centred at $w$, of angular radius $\theta$ and central axis the half-line from $w$ through $y_0$ (if $d=1$, take $C(w)$ to be the closed half-line starting at $w$ and containing $y_0$). The set $A(w,\epsilon):= C(w) \cap \bb^d(w,\epsilon)$ is now included in $D_0$. This is clear if $d=1$. If $d \geq 2$, then if $w \notin \bb^d(y_0,\epsilon)$, it follows that $|y_0-w| > \epsilon$ and since $\theta \leq \omega_w \wedge \frac{\pi}{4}$, one deduces that $A(w,\epsilon)$ is inside the convex hull of $w \cup \bb^d(y_0,\epsilon)$ (here we need $\theta \leq \frac{\pi}{4}$). If $w \in \bb^d(y_0,\epsilon)$, then by construction $A(w,\epsilon)$ is also inside $\bb^d(y_0,\epsilon)$ (here we need $\theta \leq \frac{\pi}{3}$).
	
	Since $s,\epsilon \leq \ell$, it follows that 
	\[
	A\left(w,\frac{\epsilon s}{\ell}\right) \subset A(w,\epsilon) \cap \bb^d(w,s) \subset D_0 \cap \bb^d(w,s).
	\]
	The set $A(0,1) = \bb^d(0,1) \cap C(0)$ is convex and of non-zero volume, therefore it contains an open ball of some radius $r>0$, centred at some point $z \in A(0,1)$. By translation and scaling,
	\[
	B^d\left(\frac{\epsilon s}{\ell}\cdot z + w, \frac{\epsilon sr}{\ell}  \right) \subset A\left(w,\frac{\epsilon s}{\ell}\right).
	\]
	Thus with $q = \frac{\epsilon r}{\ell} $ and $w' = \frac{\epsilon s}{\ell}\cdot z + w$, we achieve
	\[
	B^d(w',qs) \subset A\left(w,\frac{\epsilon s}{\ell}\right) \subset D_0 \cap \bb^d(w,s).
	\]
	Since $B^d(w',qs)$ is open, we have $B^d(w',qs) \subset B^d(w,s)$.
	The statement \eqref{AppelsandPears} follows easily by taking $c_1(\epsilon,\ell,d):= \kappa_d q^d$.
\end{proof}

For the next part, we need some additional technical definitions. Let $C_1(0),...,C_K(0)$ be a collection of closed cones centred at $0$ and of angular radius $\frac{\pi}{12}$ such that $\bigcup_{i=1}^K C_i(0) = \R^d$ (for $d=1$, take $C_1(0):= (-\infty,0]$ and $C_2(0):= [0,\infty)$). Define $C_i(x):=C_i(0)+x$ for $x \in \R^d$ and let $C_1^+(x),...,C_K^+(x)$ be the cones centred at $x$ of angular radius $\frac{\pi}{6}$ such that $C_i^+(x)$ has the same central half-axis as $C_i(x)$ for each $i$ (for $d=1$, take $C_i^+(0)=C_i(0)$).

\noindent For $A \subset \R^d$ a convex body, $x \in A$ and $\mu \subset \R^d \times [0,1]$ a finite set which is generic with respect to $x$, define
\begin{itemize}
	\item $s_i(x,A) := \diam(C_i(x) \cap A)$;
	\item $R_{i,\theta}(x,A,\mu) := \inf\{|x-y|:\ (y,u) \in \left(A \cap C_i^+(x) \times [0,\theta)\right)\cap \mu \} \wedge s_i(x,A)$ for $\theta \in [0,1]$ and where $\inf \emptyset = \infty$;
	\item $R_\theta(x,A,\mu) := \max_{i=1,...,K} R_{i,\theta}(x,A,\mu)$.
\end{itemize}
In practice, when it is clear what $A$ and $\mu$ are, we will omit them from the notation and write $s_i(x),R_{i,\theta}(x),R_\theta(x)$. In words, $s_i(x,A)$ is the maximal distance within $A$ from $x$ to any point in the cone $C_i(x)$. The quantity $R_{i,\theta}(x,A,\mu)$ is the distance to the closest point of $\mu$ of mark smaller then $\theta$ within $A$ and the larger cone $C_i^+(x)$. The quantity $R_\theta(x,A,\mu)$ is such that the ball $B^d(x,R_\theta(x,A,\mu))$ either contains a point of mark less than $\theta$, or if it doesn't, then it contains all of $A$.

\begin{rem}
	The use of cones and the construction of the radii $R_\theta$ follows the ideas given in \cite{Pen05,Wade2009}, again with the difference that we work in a Poisson setting without passing through uniform random vectors.
\end{rem}

\begin{lem} \label{lemGeo2}
	Let $\epsilon,\ell>0$ and $D_0 \subset \R^d$ be as in Lemma~\ref{lemGeo1}. Take $\lambda > 0$ and $x \in \lambda D_0$, and let $0< r \leq s_i(x,\lambda D_0)$ for some $i \in \{1,...,K\}$. Then there is a constant $c_2(\epsilon,\ell,d)>0$ such that
	\[
	|\lambda D_0 \cap C_i^+(x) \cap B^d(x,r)| \geq c_2(\epsilon,\ell,d) r^d.
	\]
\end{lem}
\begin{rem}
	A similar result was shown in the proof of \cite[Lemma~3.2]{Wade2009}, with a constant depending on $D_0$.
\end{rem}
\begin{proof}
	First, we show that there is a point $z \in \lambda D_0$ such that $B^d(z, \frac{r}{8}) \subset C_i^+(x) \cap B^d(x,r)$. Indeed, as $r \leq s_i(x,\lambda D_0)$, there is by convexity a point $z \in \lambda D_0 \cap C_i(x)$ such that $|x-z| = \frac{r}{2}$. Now consider any point $y \in B^d(z,\frac{r}{8})$. For $d=1$, one has that $y \in C_i^+(x) \cap B^1(x,r)$. For $d \geq 2$, the angle $\angle zxy$ will be largest when the line $(xy)$ is tangent to $B^d(z,\frac{r}{8})$, in which case
	\[
	\angle zxy \leq \arcsin\left(\frac{r/8}{r/2}\right) = \arcsin\left(\frac{1}{4}\right) < \frac{\pi}{12}.
	\]
	
	Since $z \in C_i(x)$, this implies that $y \in C_i^+(x)$ and we clearly also have $|x-y| \leq |x-z| + |z-y| \leq \frac{5}{8}r < r$. We infer that for all $d \geq 1$,
	\[
	|\lambda D_0 \cap C_i^+(x) \cap B^d(x,r)| \geq |\lambda D_0 \cap B^d(z,\tfrac{r}{8})|.
	\]
	
	Now $\frac{1}{\lambda}z \in D_0$ and $\frac{1}{\lambda} \frac{r}{8} \leq \diam(D_0)\leq \ell$, therefore Lemma~\ref{lemGeo1} implies
	\[
	\left|\lambda D_0 \cap B^d\left(z,\tfrac{r}{8}\right)\right| = \lambda^d \left| D_0 \cap B^d\left(\tfrac{1}{\lambda}z,\tfrac{1}{\lambda}\tfrac{r}{8}\right)\right| \geq \lambda^d c_1(\epsilon,\ell,d) \left(\tfrac{r}{8\lambda}\right)^d = 8^{-d} c_1(\epsilon,\ell,d) r^d,
	\]
	which yields the desired bound.
\end{proof}

\begin{lem}\label{lemONNGradius}
	Let $\epsilon,\ell>0$ and $D_0 \subset \R^d$ be as in Lemma~\ref{lemGeo1}. Let $\lambda>0$, $x \in \lambda D_0$ and $\theta \in (0,1)$. Then for all $\beta > 0$, there is a constant $c_3(\epsilon,d,\ell,\beta)>0$ such that
	\[
	\E\left[R_\theta(x,\lambda D_0, \eta_{\lambda D_0} )^\beta\right] \leq c_3(\epsilon,d,\ell,\beta) \left(\theta^{-\beta/d} \wedge \lambda^\beta\right).
	\]
\end{lem}
\begin{rem}
	This result is comparable to \cite[Lemma~3.2]{Wade2009} and an argument in the proof of \cite[Lemma~3.3]{Pen05}, both of which worked with uniform random vectors.
\end{rem}
\begin{proof}
	It is clear that by construction $R_\theta(x,\lambda D_0, \eta_{\lambda D_0}) \leq \diam(\lambda D_0) \leq \ell \lambda$, thus we need only show the bound by $\theta^{-\beta/d}$. To simplify the notation, write $R_\theta(x)$ for $R_\theta(x,\lambda D_0, \eta_{\lambda D_0})$. We are going to study the probability $\p(\R_\theta(x)>r)$. If for some $r>0$, we have $R_\theta(x)>r$, then $\max_{i=1,...,K} R_{i,\theta}(x)>r$, implying that there is an $i \in \{1,...,K\}$ such that $R_{i,\theta}(x) > r$. Since $R_{i,\theta}(x) \leq s_i(x)$, this implies in turn that $r<s_i(x)$ and that
	\[
	\eta \left(\left(\lambda D_0 \cap B^d(x,r) \cap C_i^+(x)\right)\times [0,\theta)\right) = 0.
	\]
	It follows that
	\begin{align}
		\p(\R_\theta(x)>r)%
		 &\leq \p \left( \bigcup_{i=1}^K \left\{ \eta \left( (\lambda D_0 \cap B^d(x,r) \cap C_i^+(x)) \times [0,\theta) \right) = 0,\ r<s_i(x) \right\} \right) \notag\\
		 &\leq \sum_{i=1}^K \ind{r<s_i(x)} \p\left(\eta \left( (\lambda D_0 \cap B^d(x,r) \cap C_i^+(x)) \times [0,\theta) \right) = 0\right) \notag\\
		 &= \sum_{i=1}^K \ind{r<s_i(x)} \exp\left(-\theta |\lambda D_0 \cap B^d(x,r) \cap C_i^+(x)|\right).
	\end{align}
	By Lemma~\ref{lemGeo2}, there is a constant $C:=c_2(\epsilon,\ell,d)>0$ such that
	\[
	|\lambda D_0 \cap B^d(x,r) \cap C_i^+(x)| \geq C r^d.
	\]
	Hence
	\[
	\exp\left(-\theta |\lambda D_0 \cap B^d(x,r) \cap C_i^+(x)|\right) \leq \exp(-\theta C r^d)
	\]
	and
	\[
	\p(\R_\theta(x)>r) \leq K \exp(-\theta C r^d).
	\]
	We can now estimate for all $\beta>0$:
	\begin{align}
		\E [R_\theta(x)^\beta] &= \int_0^\infty \p(R_\theta(x)>r^{1/\beta}) dr\notag\\
		&\leq \int_0^\infty K \exp(-\theta C r^{d/\beta}) dr \notag\\
		&= K \frac{\beta}{d} \theta^{-\beta/d} C^{-\beta/d} \int_0^\infty u^{\beta/d-1} \exp(-u) du \notag\\
		&= K C^{-\beta/d} \Gamma\left(\frac{\beta}{d}+1\right) \theta^{-\beta/d},
	\end{align}
	by the properties of the Gamma function $\Gamma$ (see e.g. \cite[6.1.1]{Hand}).
\end{proof}

Before we pass to the next lemma, we note some useful properties of the add-one cost and introduce the quantity $\opL_{(x,s)}\Fa(\mu)$. Let $(x,s) \in \R^d \times [0,1]$ and $\mu \subset \R^d\times[0,1]$ be a generic set with respect to $(x,s)$. Note that
\begin{equation}\label{eqDexpr}
\MalD_{(x,s)}\Fa(\mu) = e(x,s,\mu)^\alpha + \sum_{\substack{(y,u) \in \\ \mu \cap (\R^d \times (s,1])}} \left( e(y,u,\mu + \delta_{(x,s)})^\alpha - e(y,u,\mu)^\alpha \right).
\end{equation}
Define the quantity
\[
\opL_{(x,s)} \Fa(\mu) := \sum_{\substack{(y,u) \in \\ \mu \cap (\R^d \times (s,1])}} e(y,u,\mu)^\alpha \ind{(y,u) \rightarrow (x,s) \text{ in } \mu + \delta_{(x,s)}},
\]
where we use (as before) the notation $(y,u) \rightarrow (x,s) \text{ in } \mu + \delta_{(x,s)}$ in order to indicate that the point $(y,u)$ connects to $(x,s)$ in the collection of points $\mu + \delta_{(x,s)}$, as was explained in Section~\ref{secONNG}.

We claim the following: For any convex body $A \subset \R^d$ such that the projection onto $\R^d$ of $\mu$ is included in the interior of $A$, it holds that

\begin{equation}
	\left| \MalD_{(x,s)}\Fa(\mu) \right| \leq R_s(x,A,\mu)^\alpha + \opL_{(x,s)} \Fa(\mu). \label{eqDbyL}
\end{equation}
To show the claim, we start by noting the following: if a point $(y,u) \in \mu \cap (\R^d \times (s,1])$ has an online nearest neighbour in $\mu$, then $e(y,u,\mu) \neq 0$. There are now two possibilities:
\begin{enumerate}
	\item The point $(y,u)$ connects to $(x,s)$ in $\mu + \delta_{(x,s)}$. This implies that $e(y,u,\mu + \delta_{(x,s)}) < e(y,u,\mu)$.
	\item The point $(y,u)$ does not connect to $(x,s)$ in $\mu + \delta_{(x,s)}$. Then $e(y,u,\mu + \delta_{(x,s)}) = e(y,u,\mu)$.
\end{enumerate}
In both cases, it holds that
\begin{equation}\label{eqonn}
|e(y,u,\mu + \delta_{(x,s)})^\alpha - e(y,u,\mu)^\alpha| \leq e(y,u,\mu)^\alpha \ind{(y,u) \rightarrow (x,s) \text{ in } \mu + \delta_{(x,s)}}.
\end{equation}
As a next step, consider three scenarios:
\begin{enumerate}
	\item If $\mu \cap {\R^d \times [0,s)} \neq \emptyset$, then any point $(y,u) \in \mu$ with $u>s$ has an online nearest neighbour in $\mu$ and \eqref{eqonn} holds. Moreover, the point $(x,s)$ has an online nearest neighbour in $\mu$ and by construction, $e(x,s,\mu) \leq R_s(x,A,\mu)$. Combining this with \eqref{eqDexpr} and \eqref{eqonn} implies that \eqref{eqDbyL} holds.
	\item If $\mu \cap {\R^d \times [0,s)} = \emptyset$ but $\mu \cap {\R^d \times (s,1]} \neq \emptyset$, then the point $(x,s)$ does not have an online nearest neighbour in $\mu$ and $e(x,s,\mu)=0$. However, there is now a point $(y_0,u_0) \in \mu \cap {\R^d \times (s,1]}$ which is the point of lowest mark in $\mu$ and it does not have an online nearest neighbour in $\mu$. Hence $e(y_0,u_0,\mu)=0$ and $e(y_0,u_0,\mu + \delta_{(x,s)})=|x-y_0|$ since this point will connect to $(x,s)$ in $\mu + \delta_{(x,s)}$. Since $\mu \cap {\R^d \times [0,s)} = \emptyset$, we have that $A \subset \overbar{B}^d(x,R_s(x,A,\mu))$ and hence $|x-y_0| \leq R_s(x,A,\mu)$. Any point $(y,u) \in \mu$ different from $(y_0,u_0)$ must have an online nearest neighbour in $\mu$, since $(y_0,u_0)$ is a potential neighbour. Thus for such $(y,u)$, the inequality \eqref{eqonn} holds and we deduce
	\begin{align}
	\left| \MalD_{(x,s)}\Fa(\mu) \right| &= \bigg| |x-y_0|^\alpha + \sum_{\substack{(y_0,u_0) \neq (y,u) \in \\ \mu \cap (\R^d \times (s,1])}} \left( e(y,u,\mu + \delta_{(x,s)})^\alpha - e(y,u,\mu)^\alpha \right) \bigg| \\
	&\leq R_s(x,A,\mu)^\alpha + \opL_{(x,s)} \Fa(\mu),
	\end{align}
	thus inequality \eqref{eqDbyL} holds.
	\item If $\mu = \emptyset$, then $\MalD_{(x,s)}\Fa(\mu) = 0$ and inequality \eqref{eqDbyL} trivially holds.
\end{enumerate}
This concludes the proof of the claim. In the next lemma, we now provide a bound for the quantity $\opL_{(x,s)}\Fa(\eta_{\lambda D_0})$.

\begin{lem}\label{lemONNGFirstBounds}
	Let $\epsilon,\ell>0$ and $D_0 \subset \R^d$ be as in Lemma~\ref{lemGeo1}. Let $\lambda>0$ and $(x,s) \in \lambda D_0 \times [0,1]$. Then there is a constant $c_3(\epsilon,\ell,d,\alpha)>0$ such that
	\begin{equation}\label{eqONNGRestBound}
		\E [\opL_{(x,s)}\Fa(\eta_{\lambda D_0})] \leq c_4(\epsilon,\ell,d,\alpha) (s^{-\alpha/d} \wedge \lambda^\alpha).
	\end{equation}
\end{lem}
\begin{proof} 
	To prove \eqref{eqONNGRestBound}, let $(y,u) \in \eta_{\lambda D_0}$ be a point connecting to $(x,s)$ in $\eta_{\lambda D_0} + \delta_{(x,s)}$. We will proceed in three steps.
	
	\textbf{Step 1.} We claim that $y \in \bb^d(x,\R_\theta(x,\lambda D_0,\eta_{\lambda D_0}))$ for any $\theta < u$.
	
	As the $C_i(x)$, $i=1,...,K$, cover $\R^d$, there is an $i$ such that $y \in C_i(x)$. Assume that $|x-y| > R_\theta(x)$. Since $|x-y| \leq s_i(x,\lambda D_0)$, we must have 
	\[
	R_{i,\theta}(x) \leq R_\theta(x) < s_i(x,\lambda D_0)
	\]
	and hence
	\[
	\eta \cap ((\lambda D_0 \cap C_i^+(x) \cap \bb^d(x,R_\theta(x)))\times [0,\theta)) \neq \emptyset.
	\]
	So there is a point $(z,v)$ within this set and in particular $v < \theta < u$. We now have:
	\begin{itemize}
		\item $|x-y| > R_\theta(x) \geq |x-z|$
		\item $z \in C_i^+(x)$
		\item $y \in C_i(x)$.
	\end{itemize}
	By the geometric properties of the cones $C_i(x)$ and $C_i^+(x)$, shown in \cite[Lemma~3.3]{Pen05}, this implies that $|z-y| \leq |x-y|$. Since $|z-y| \neq |x-y|$ a.s., we deduce that the point $y$ will connect to $z$ rather than to $x$, which is a contradiction. We must thus have $|x-y| \leq R_\theta(x)$.
	
	\textbf{Step 2.} We will now derive the first part of the bound in \eqref{eqONNGRestBound}.
	
	For $0 \leq \theta_1 < \theta_2 \leq 1$, define
	\[
	H(\theta_1,\theta_2) := \sum_{\substack{(y,u) \in \\ \eta \cap (\lambda D_0 \times (\theta_1,\theta_2])}} e(y,u,\eta_{\lambda D_0})^\alpha \ind{(y,u) \rightarrow (x,s) \text{ in } \eta_{\lambda D_0} + \delta_{(x,s)}}.
	\]
	This is the contribution to the sum $\opL_{(x,s)}\Fa(\eta_{\lambda D_0})$ of points with marks in the interval $(\theta_1,\theta_2]$, so that
	\begin{equation} \label{iP5Lsplit}
	\opL_{(x,s)}\Fa(\eta_{\lambda D_0}) = H(0,1) = H(s,1) = \sum_{i=1}^n H(\theta_{i-1},\theta_i),
	\end{equation}
	for any partition $s=\theta_0 < \theta_1 < ... < \theta_n = 1$. 
	
	Let $(y,u) \in \eta \cap \left(\lambda D_0 \times (\theta_{i-1},\theta_i]\right)$ be a point connecting to $x$. By Step 1, we have $y \in \bb(x,R_{\theta_{i-1}}(x))$. By construction, either $\lambda D_0 \cap B^d(x,R_{\theta_{i-1}}(x))$ contains a point $(w,v)$ of $\eta$ with mark less than $\theta_{i-1}$, in which case $e(y,u,\eta_{\lambda D_0}) \leq |y-w| \leq |x-y| + |x-w|$, or $\lambda D_0 \subset B^d(x,R_{\theta_{i-1}}(x))$. In both cases,
	\begin{equation}\label{iP5Edge}
	e(y,u,\eta_{\lambda D_0}) \leq |x-y| + R_{\theta_{i-1}}(x) \leq 2 R_{\theta_{i-1}}(x).
	\end{equation}
	Moreover, the number of points in $\lambda D_0 \times (\theta_{i-1},\theta_i]$ connecting to $x$ is upper bounded by the total number of points in $\eta \cap \left((\lambda D_0 \cap \bb^d(x,R_{\theta_{i-1}}(x))) \times (\theta_{i-1},\theta_i]\right)$.
	From this we conclude that
	\[
	H(\theta_{i-1},\theta_i) \leq 2^\alpha R_{\theta_{i-1}}(x)^\alpha \eta \left((\lambda D_0 \cap \bb^d(x,R_{\theta_{i-1}}(x))) \times (\theta_{i-1},\theta_i]\right).
	\]
	To upper bound the expectation of $H(\theta_{i-1},\theta_i)$, we are going to use the fact that $\eta_{|\lambda D_0 \times [0,\theta_{i-1}]}$ and $\eta_{|\lambda D_0 \times (\theta_{i-1},\theta_i]}$ are independent and $R_{\theta_{i-1}}(x)$ is measurable with respect to $\eta_{|\lambda D_0 \times [0,\theta_{i-1}]}$. We can thus calculate
	\begin{align}
		\E \left[ H(\theta_{i-1},\theta_i) \right] &\leq 2^\alpha \E \left[R_{\theta_{i-1}}(x)^\alpha \eta \left((\lambda D_0 \cap \bb^d(x,R_{\theta_{i-1}}(x))) \times (\theta_{i-1},\theta_i]\right)\right] \notag\\
		&= 2^\alpha \E \left[R_{\theta_{i-1}}(x)^\alpha \E \left[\left.\eta \left((\lambda D_0 \cap \bb^d(x,R_{\theta_{i-1}}(x))) \times (\theta_{i-1},\theta_i]\right)\right|\eta_{|\lambda D_0 \times [0,\theta_{i-1}]}\right]\right] \notag\\
		&= 2^\alpha \E \left[ R_{\theta_{i-1}}(x)^\alpha \big|\lambda D_0 \cap \bb^d(x,R_{\theta_{i-1}}(x))\big| (\theta_i-\theta_{i-1}) \right]. \label{Rainbow14}
	\end{align}
	Upper bounding the volume of the intersection by $\kappa_d R_{\theta_{i-1}}(x)^d$ and applying Lemma~\ref{lemONNGradius} yields that \eqref{Rainbow14} is bounded by
	\[
	2^\alpha \kappa_d (\theta_i-\theta_{i-1}) \E \left[R_{\theta_{i-1}}(x)^{\alpha + d}\right] \leq 2^\alpha \kappa_d (\theta_i-\theta_{i-1}) c\, \theta_{i-1}^{-\alpha/d-1},
	\]
	with $c = c_3(\epsilon,d,\ell,\alpha + d)$. Combining this with \eqref{iP5Lsplit}, we infer that
	\[
	\E \left[ \opL_{(x,s)}\Fa(\eta_{\lambda D_0})\right] \leq 2^\alpha \kappa_d c \sum_{i=1}^n \theta_{i-1}^{-\alpha/d-1} (\theta_i-\theta_{i-1})
	\]
	for any partition $s=\theta_0< \theta_1 < ... < \theta_n=1$. Letting $n \rightarrow \infty$ and the mesh of the partition tend to $0$, we get
	\begin{align}
		\E \left[ \opL_{(x,s)}\Fa(\eta_{\lambda D_0})\right] &\leq 2^\alpha \kappa_d c \int_s^1 \theta^{-\alpha/d-1} d\theta \notag\\
		&= 2^\alpha \kappa_d c \frac{d}{\alpha} \left(s^{-\alpha/d}-1\right) \notag\\
		&\leq c' s^{-\alpha/d},
	\end{align}
	with $c':= 2^\alpha \kappa_d c_3(\epsilon,d,\ell,\alpha + d) \frac{d}{\alpha}$, a constant independent of $x$, $s$ and $\lambda$.
	
	\textbf{Step 3.} To show the second part of \eqref{eqONNGRestBound}, note that
	\begin{equation}\label{E35}
	\opL_{(x,s)}\Fa(\eta_{\lambda D_0}) \leq \opL_{(x,0)}\Fa(\eta_{\lambda D_0}).
	\end{equation}
	Indeed, $e(y,u,\eta_{\lambda D_0})$ is independent of $(x,s)$ and hence remains unchanged for any $(y,u)\in \eta_{\lambda D_0}$ if we reduce the arrival time of $(x,s)$ to zero. Any point connecting to $(x,s)$ will also connect to $(x,0)$, hence no terms are deleted from the sum. Some positive terms might be added, since there may be points connecting to $(x,0)$ but not to $(x,s)$.
	
	For any $0<\theta_0<1$,
	\[
	\opL_{(x,0)}\Fa(\eta_{\lambda D_0}) = H(0,\theta_0) + \opL_{(x,\theta_0)}\Fa(\eta_{\lambda D_0}).
	\]
	A crude upper bound for the first term on the RHS is as follows:
	\[
	H(0,\theta_0) \leq \diam(\lambda D_0)^\alpha \eta(B^d(x,\lambda \diam(D_0)) \times [0,\theta_0)).
	\]
	Taking expectation, we get
	\[
	\E H(0,\theta_0) \leq \kappa_d\diam(D_0)^{\alpha+d} \lambda^{\alpha + d} \theta_0.
	\]
	On the other hand, by Step 2 we have
	\begin{equation}\label{E39}
	\E \left[ \opL_{(x,\theta_0)}\Fa(\eta_{\lambda D_0})\right] \leq c' \theta_0^{-\alpha/d}.
	\end{equation}
	Combining \eqref{E35}-\eqref{E39} yields for any $0<\theta_0<1$ that
	\[
	\E \left[\opL_{(x,s)}\Fa(\eta_{\lambda D_0})\right] \leq \kappa_d\diam(D_0)^{\alpha+d} \lambda^{\alpha + d} \theta_0 + c' \theta_0^{-\alpha/d}.
	\]
	Choosing $\theta_0 := \lambda^{-d}$, we deduce
	\[
	\E \opL_{(x,s)}\Fa(\eta_{\lambda D_0}) \leq \left(\kappa_d\diam(D_0)^{\alpha+d} + c'\right) \lambda^{\alpha}.
	\]
	The bound in \eqref{eqONNGRestBound} follows with $c_4(\epsilon,\ell,d,\alpha) := \kappa_d\diam(D_0)^{\alpha+d} + c'$.
\end{proof}

\begin{proof}[Proof of Prop.~\ref{propONNGcondBounds}]
	\textbf{Proof of \eqref{eqONNGdxBound}.} We will show that there is a constant $C_1>0$ such that for all $(x,s)\in tH\times[0,1]$ and $t\geq 1$
	\begin{equation}\label{iP6bound1}
		\E \big[|\MalD_{(x,s)}\fat| \big| \eta_{|tH \times [0,s)}\big] \leq C_1 R_s(x,tH,\eta_{tH})^\alpha,
	\end{equation}
	where $C_1$ does not depend on $x,\ s$ or $t$. The claim \eqref{eqONNGdxBound} then follows by Lemma~\ref{lemONNGradius}.
	
	From here on, write $R_s(x)$ for $R_s(x,tH,\eta_{tH})$.
	
	By \eqref{eqDbyL} and since $e(x,s,\eta_{tH}) \leq R_s(x)$, which is measurable with respect to $\eta_{|tH \times [0,s)}$, it is enough to show that
	\begin{equation}\label{iP6Lbound1}
		\E \big[|\opL_{(x,s)}\fat| \big| \eta_{|tH \times [0,s)}\big] \leq C_2 R_s(x)^\alpha,
	\end{equation}
	for some constant $C_2>0$.
	
	As established in the proof of Lemma~\ref{lemONNGFirstBounds}, all points connecting to $(x,s)$ are points of $\eta$ inside the set $A(t,x,s):=(tH \cap \bb^d(x,R_s(x)))\times (s,1]$. Since this is included in $tH$, a set of finite measure, the total number of points in $\eta \cap A(t,x,s)$ is almost surely finite. Denote the points connecting to $(x,s)$ by $(y_1,s_1),...,(y_m,s_m)$, ordered by increasing mark, i.e. $s_1<s_2<...<s_m$.
	
	With this notation, we have
	\begin{equation}\label{iP6Lbound2}
	\opL_{(x,s)}\fat = \sum_{i=1}^m e(y_i,s_i,\eta_{tH})^\alpha.
	\end{equation}
	When removing all points of $\eta$ outside of $A(t,x,s)$, i.e. all points further than $R_s(x)$ from $x$ and all points of mark less than $s$, then for each of the points $y_2,...,y_m$, the edge-length to its online nearest neighbour in $\eta_{tH}$ will either stay constant or increase. It will never be zero because $y_1$ remains as a potential nearest neighbour. In formulas:
	\[
	e(y_i,s_i,\eta_{tH}) \leq e(y_i,s_i,\eta_{|A(t,x,s)}), \qquad \text{for } i=2,...,m.
	\]
	For $i=1$, we know as in the proof of Lemma~\ref{lemONNGFirstBounds} that
	\[
	e(y_1,s_i,\eta_{tH}) \leq 2 R_s(x).
	\]
	Combining \eqref{iP6Lbound2} with the above discussion, we get
	\[
	\opL_{(x,s)}\fat \leq 2^\alpha R_s(x)^\alpha + \sum_{i=2}^m e(y_i,s_i,\eta_{|A(t,x,s)}).
	\]
	Using that edge-lengths are non-negative, we have
	\begin{equation}\label{iP6Lbound3}
		\opL_{(x,s)}\fat \leq 2^\alpha R_s(x)^\alpha + \sum_{\substack{(y,u) \in \\ \eta \cap A(t,x,s)}} e(y,u,\eta_{|A(t,x,s)})^\alpha \ind{(y,u) \rightarrow (x,s) \text{ in } \eta_{|A(t,x,s)} + \delta_{(x,s)}}.
	\end{equation}
	Here we have added the edge contributions from points that would not connect to $(x,s)$ in $\eta_{tH} + \delta_{(x,s)}$, but that do connect to $(x,s)$ in the smaller set $A(t,x,s) + \delta_{(x,s)}$. The sum on the RHS on \eqref{iP6Lbound3} still includes the points $(y_i,s_i)$ for $i \geq 2$.
	
	Observe that $R_s(x)>0$ a.s. and define
	\[
	D_0 := R_s(x)^{-1} \left(tH \cap \bb^d(x,R_s(x))\right).
	\]
	Now we have $A(t,x,s) = R_s(x)D_0 \times (s,1]$ and we can write the second term in \eqref{iP6Lbound3} as
	\begin{equation}\label{iP6sumsum}
	\sum_{\substack{(y,u) \in \\ \eta \cap (R_s(x)D_0 \times (s,1])}} e(y,u,\eta_{|R_s(x)D_0 \times (s,1]})^\alpha \ind{(y,u) \rightarrow (x,s) \text{ in } \eta_{|R_s(x)D_0 \times (s,1]} + \delta_{(x,s)}}.
	\end{equation}
	By the properties of Poisson measures, $\eta_{|tH \times [0,s)}$ (and hence also $R_s(x)$) is independent of $\eta_{|tH \times (s,1]}$. Conditioning on $\eta_{|tH \times [0,s)}$, the sum \eqref{iP6sumsum} is equal in law to $\opL_{(x,0)}\Fa(\eta_{|\lambda D_0 \times [0,1-s)}')$, where $\lambda=R_s(x)$ and $\eta'$ is an independent copy of $\eta$.
	
	We clearly have that
	\[
	\opL_{(x,0)}\Fa(\eta_{|\lambda D_0 \times [0,1-s)}') \leq \opL_{(x,0)}\Fa(\eta_{|\lambda D_0 \times [0,1]}') = \opL_{(x,0)}\Fa(\eta_{\lambda D_0}')
	\]
	since adding points of higher mark only adds more edges to the graph.
	
	By inequality \eqref{eqONNGRestBound} of Lemma~\ref{lemONNGFirstBounds},
	\[
	\E_{\eta'}\left[\opL_{(x,0)}\Fa(\eta_{\lambda D_0}')\right] \leq c_4(\epsilon,\ell,d,\alpha) \lambda^\alpha = c_4(\epsilon,\ell,d,\alpha) R_s(x)^\alpha,
	\]
	where $\ell$ is an upper bound for $\diam(D_0)$ and $\epsilon$ is such that there is a point $x_0 \in D_0$ with $B^d(x_0,\epsilon) \subset D_0$. Our claim \eqref{iP6Lbound1} is proven if we can find such deterministic $\ell$ and $\epsilon$ that do not depend on $x,\ s$ or $t$.
	
	We have that
	\[
	D_0 = R_s(x)^{-1}tH \cap \bb^d(R_s(x)^{-1}x,1),
	\]
	therefore $\diam(D_0)\leq 2=:\ell$. To find an $\epsilon$, let $w:= t^{-1}x \in H$ and $r := t^{-1}R_s(x)$. Since $R_s(x)\leq t\diam(H)$, we have $r \leq \diam(H)$ and by Lemma~\ref{lemGeo1}, there exists a $0<q<1$ depending on $H$ and its properties such that there is a point $w' \in H$ with
	\[
	B^d(w',qr) \subset H \cap B^d(w,r).
	\]
	Multiplying by $tR_s(x)^{-1}$ yields
	\[
	B^d(tR_s(x)^{-1}w',q) \subset R_s(x)^{-1} \left(tH \cap B^d(x,R_s(x))\right) = D_0.
	\]
	We can thus pick $\epsilon := q$ and the result is shown.
	
	\textbf{Proof of \eqref{eqONNGdxybound}.} By Lemma~\ref{lemONNGradius}, it suffices to show that
	\[
	\E\big[|\DD_{(x,s),(y,u)}\fat| \big|\eta_{|tH \times [0,u)}\big] \leq c R_u(y)^\alpha,
	\]
	for some constant $c>0$.
	
	By the triangle inequality,
	\[
	|\DD_{(x,s),(y,u)}\fat| \leq |\MalD_{(y,u)} \Fa(\eta_{tH})| + |\MalD_{(y,u)} \Fa(\eta_{tH}+\delta_{(x,s)})|.
	\]
	The inequality \eqref{iP6bound1} deals with the first term on the RHS. For the second term, we use \eqref{eqDbyL} to say
	\[
	|\MalD_{(y,u)} \Fa(\eta_{tH}+\delta_{(x,s)})| \leq R_u(y,tH,\eta_{tH}+\delta_{(x,s)})^\alpha + \opL_{(y,u)} \Fa (\eta_{tH}+\delta_{(x,s)}).
	\]
	Reusing arguments from the proof of Lemma~\ref{lemONNGFirstBounds}, one sees that
	\[
	R_u(y,tH,\eta_{tH}+\delta_{(x,s)}) \leq R_u(y,tH,\eta_{tH}),
	\]
	which yields the required bound for this term. On the other hand, if we remove a point of mark lower than the one of $y$, more points might connect to $y$, and no points already connected to $y$ will change neighbour. Therefore
	\[
	\opL_{(y,u)} \Fa (\eta_{tH}+\delta_{(x,s)}) \leq \opL_{(y,u)} \Fa (\eta_{tH}).
	\]
	The result now follows using \eqref{iP6Lbound1}.
	
	\textbf{Proof of \eqref{eqONNGpBound}.} Using ideas explained in \cite[p.673]{LPS14}, we will show that for any convex body $A \subset \R^d$, for all $(x,s),(y,u) \in A \times [0,1]$ with $s<u$ and any finite set $\mu \subset A \times [0,1]$ generic with respect to $(x,s)$ and $(y,u)$, the condition 
	\begin{equation}\label{iP6cond}
		|x-y| > 3R_u(y,A,\mu),
	\end{equation}
	implies that $\DD_{(x,s),(y,u)}\Fa(\mu) = 0$. This property is related to the theory of stabilisation, see \cite[p.673]{LPS14} for a discussion.
	Recall that
	\begin{multline}
	\MalD_{(y,u)}\Fa(\mu)\\ = e(y,u,\mu)^\alpha + \sum_{\substack{(z,v) \in \\ \mu \cap (\R^d \times (u,1])}} \left( e(z,v,\mu + \delta_{(y,u)})^\alpha - e(z,v,\mu)^\alpha \right)\ind{(z,v) \rightarrow (y,u) \text{ in } \mu + \delta_{(y,u)}}
	\end{multline}
	and
	\[
	\DD_{(x,s),(y,u)}\Fa(\mu) = \MalD_{(y,u)}\Fa(\mu + \delta_{(x,s)}) - \MalD_{(y,u)}\Fa(\mu).
	\]
	Condition \eqref{iP6cond} implies that $A \not\subset \bb^d(y,R_u(y,A,\mu))$ and hence that $(y,u)$ has an online nearest neighbour in $\mu$ and $e(y,u,\mu) \neq 0$. Thus $e(y,u,\mu+\delta_{(x,s)}) = |x-y| \wedge e(y,u,\mu)$, but at the same time
	\[
	e(y,u,\mu) \leq R_u(y,A,\mu) < |x-y|.
	\]
	It follows that $e(y,u,\mu+\delta_{(x,s)}) = e(y,u,\mu)$. Let $(z,v)$ be a point that connects to $(y,u)$ in $\mu + \delta_{(y,u)}$. Then as shown previously,
	\begin{equation}
	|z-y| \leq R_u(y,A,\mu) \quad \text{and} \quad e(z,v,\mu) \leq 2R_u(y,A,\mu).\label{Rainbow15}
	\end{equation}
	By conditions \eqref{iP6cond} and \eqref{Rainbow15},
	\[
	|x-z| \geq |x-y| - |z-y| > 2 R_u(y,A,\mu) \geq \max\{|z-y| , e(z,v,\mu)\}
	\]
	and as before $e(z,v,\mu)\neq 0$, since $(y,u)$ has an online nearest neighbour in $\mu$. Put together, this implies that $(z,v)$ will not connect to $(x,s)$, neither in $\mu+\delta_{(x,s)}$, nor in $\mu + \delta_{(y,u)} + \delta_{(x,s)}$. It follows that
	\[
	e(z,v,\mu + \delta_{(y,u)}) = e(z,v,\mu + \delta_{(y,u)} + \delta_{(x,s)}) \quad \text{and} \quad e(z,v,\mu) = e(z,v,\mu + \delta_{(x,s)}).
	\]
	This means that the addition of $(x,s)$ does not induce any changes to $\MalD_{(y,u)}\Fa(\mu)$ and we deduce that
	\[
	\MalD_{(y,u)}\Fa(\mu + \delta_{(x,s)}) = \MalD_{(y,u)}\Fa(\mu)
	\]
	and thus $\DD_{(x,s),(y,u)}\Fa(\mu) = 0$.
	
	Note that $R_u(y,A,\mu)$ and the reasoning above only depend on $\mu_{|A \times [0,u)}$ and thus for any finite $\chi \subset A \times (u,1]$, by the same reasoning one concludes that $R_u(y,A,\mu_{|A \times [0,u)} \cup \chi) = R_u(y,A,\mu_{|A \times [0,u)}) = R_u(y,A,\mu)$ and
	\[
	\DD_{(x,s),(y,u)}\Fa(\mu_{|A \times [0,u)} \cup \chi) = 0 \qquad \text{if } |x-y| > 3 R_u(y,A,\mu).
	\]
	In particular, this implies that for $(x,s), (y,u) \in tH\times[0,1]$ with $s<u$ and $|x-y| > 3R_u(y,tH,\eta_{tH})$,
	\[
	\E \big[\DD_{(x,s),(y,u)} \fat \big|\eta_{|tH \times [0,u)}\big] = \int \DD_{(x,s),(y,u)} \fat(\eta_{|tH \times [0,u)} + \chi) \Pi_{u}(d\chi) = 0,
	\]
	where $\Pi_{u}$ is the law of $\eta_{|tH \times (u,1]}$.
	
	We conclude that
	\[
	\p\left(\E \big[\DD_{(x,s),(y,u)} \fat \big|\eta_{|tH \times [0,u)}\big] \neq 0 \right) \leq \p\left(|x-y| < 3 R_u(y,tH,\eta_{tH})\right).
	\]
	As in the proof of Lemma~\ref{lemONNGradius}, there is a constant $c>0$ such that
	\[
	\p\left(\tfrac{1}{3}|x-y| < R_u(y,tH,\eta_{tH})\right) \leq K \exp(-uc|x-y|^d),
	\]
	which concludes the proof.
\end{proof}

\begin{rem}
	The proofs of Lemma~\ref{lemONNGFirstBounds} and the inequality \eqref{eqONNGdxBound} given in Proposition~\ref{propONNGcondBounds} build on and extend ideas used in \cite{Wade2009}. In particular, the proof of \eqref{eqONNGdxBound} adapts a rescaling argument from the proof of \cite[Lemma~3.2]{Wade2009} and extends it to the Poisson setting and to arbitrary convex bodies. As already discussed, in our proof we need the fine control over constants introduced in Lemma~\ref{lemGeo1}.
\end{rem}

\subsection{Moment estimates of add-one costs}
To find the speed of convergence in the Kolmogorov distance, we need bounds on quantities of the type $\E |\MalD_{(x,s)}\fat|^r$, with $r \geq 1$.
\begin{prop}\label{propONNGBounds}
	We work under the conditions of Theorem~\ref{thmONNG}. Let $\alpha>0$ and $r \geq 1$. For every $\epsilon>0$ and  for all $t\geq 1$ and all $(x,s),(y,u) \in tH\times(0,1]$ with $s<u$,
	\begin{align}
		&\E\left[ |\MalD_{(x,s)}\fat|^r \right]^{1/r} \leq c_1 s^{-\alpha/d-\epsilon} \label{eqONNGdxBound2} \\
		&\E\left[ |\DD_{(x,s),(y,u)}\fat|^r \right]^{1/r} \leq c_1 u^{-\alpha/d-\epsilon} \label{eqONNGdxybound2} \\
		&\p\left( \DD_{(x,s),(y,u)}\fat \neq 0 \right) \leq C_2 \exp(-c_2 u|x-y|^d),\label{eqONNGpBound2}
	\end{align}
	where $c_1>0$ is a constant depending on $\alpha,r$ and $\epsilon$ and $c_2,C_2>0$ are absolute.
\end{prop}
\begin{rem}
	The bound \eqref{eqONNGdxBound2} is an extension to arbitrary exponents $r \geq 1$ of what was shown in \cite[Lemma~3.4]{Pen05}. The proof below builds on the same arguments, but without using uniform random vectors.
\end{rem}
\begin{proof}
	\textbf{Proof of \eqref{eqONNGdxBound2}.} By \eqref{eqDbyL} and \eqref{iP5Edge}, it can be seen that
	\begin{equation}\label{iP9dx}
	|\MalD_{(x,s)}\fat|^r \leq R_s(x)^{r \alpha} 2^{r-1} \left(1 + 2^{r\alpha}  |A(x,s)|^r \right).
	\end{equation}
	where $A(x,s) := \{(y,u) \in \eta_{|tH \times (s,1]}:\ (y,u) \rightarrow (x,s) \text{ in } \eta_{tH} + \delta_{(x,s)}\}$ is the set of points in $\eta_{tH}$ that will connect to $(x,s)$ upon addition of this point. Recall that any point connecting to $(x,s)$ must be inside $\bb^d(x,R_s(x))$. For each $i=1,...,K$, define
	\begin{multline}
	A_i(x,s) := \{ (y,u) \in \eta_{tH} \cap (C_i(x) \cap \bb^d(x,R_s(x)) \times (s,1]) : \\ |x-y| < |x-z|\ \forall (z,v) \in \eta_{tH} \cap (C_i(x) \cap \bb^d(x,R_s(x)) \times (s,u)) \}.
	\end{multline}
	This is the set of points $(y,u)$ in $\eta_{tH}$ inside the intersection of the cone $C_i(x)$ with $\bb^d(x,R_s(x))$ that are closer to $x$ than any other point of mark between $s$ and $u$ within this cone. Any point $(y,u)$ connecting to $(x,s)$ must be within such a set for some $i$, else there is a point closer to $y$ than $x$ and of lower mark, i.e. a potential neighbour. Hence
	\[
	|A(x,s)|^r \leq \left(\sum_{i=1}^K |A_i(x,s)|\right)^r \leq K^{r-1} \left(\sum_{i=1}^K |A_i(x,s)|^r\right).
	\]
	Define $m:= \ceil{r}$ and fix $i \in \{1,...,K\}$. Since $|A_i(x,s)|$ is a non-negative integer and $\frac{a}{m} \leq \frac{a-k}{m-k}$ for $0\leq k \leq m-1$ and $a \geq m$, we have
	\begin{equation}\label{iP9bound1}
		|A_i(x,s)|^r \leq |A_i(x,s)|^m \leq m^m {|A_i(x,s)| \choose m} + (m-1)^m.
	\end{equation}
	
	Our goal is now to estimate ${|A_i(x,s)| \choose m}$. First, let
	\[
	G_i := \eta_{tH} \cap \left( C_i(x) \cap B^d(x,R_s(x)) \times (s,1] \right)
	\]
	be the set of points in $\eta_{tH}$ that are closer than $R_s(x)$ to $x$, within the cone $C_i(x)$ and of mark higher than $s$. Any point connecting to $x$ must be within this set. Let $N_i:= |G_i|$ be the random number of points inside $G_i$ and note that $N_i$ is almost surely finite. Given $N_i$, the points in $G_i$ can be denoted by the random coordinates $\{(y_1,s_1),...,(y_{N_i},s_{N_i})\}$, where $y_1,...,y_{N_i}$ are the spatial coordinates of the points and $s_1,...,s_{N_i}$ are the marks of the points.

	As $N_i$ is almost surely finite, we can assume w.l.o.g. that the points $y_1,...,y_{N_i}$ are ordered by increasing distance to $x$. We now condition on the event $N_i=n\neq 0$ and, given this conditioning, we also take conditional expectation with respect to the $\sigma$-algebra generated by the random coordinates $y_j$. Since $A_i(x,s) \subset G_i$, we have
	\begin{align}
	 	\E \left[\left. {|A_i(x,s)| \choose m} \right| N_i=n,y_1,...,y_n \right] &= \sum_{\substack{\{z_1,...,z_m\} \subset \{y_1,...,y_n\} \\ \text{distinct}}} \E \left[\left.\ind{z_1,...,z_m \in A_i(x,s) }\right| N_i=n,y_1,...,y_n\right] \notag\\
	 		&= \sum_{1\leq j_1<j_2<...<j_m\leq n} \p \left(\{y_{j_1},...,y_{j_m}\} \subset A_i(x,s) | N_i=n,y_1,...,y_n\right)
	\end{align}
 	Note that this expression is zero if $n<m$.
 	
 	Conditional on the event $N_i=n$, the marks $s_1,...,s_n$ are independent of the spatial coordinates $y_1,...,y_n$ and i.i.d. uniformly distributed in $(s,1]$. Hence any ordering of the marks is equally likely. Conditional on the spatial coordinates $y_1,...,y_n$, the event $\{y_{j_1},...,y_{j_m}\} \subset A_i(x,s)$ happens if and only if $y_{j_1}$ has the smallest mark among the points $y_1,y_2,...,y_{j_1}$, and $y_{j_2}$ has the smallest mark among the points $y_1,y_2,...,y_{j_2}$ etc. The probability that this happens is exactly given by $(j_1j_2...j_m)^{-1}$. Thus
 	\[
 	\E \left[\left. {|A_i(x,s)| \choose m} \right| N_i=n,y_1,...,y_n \right] = \sum_{1\leq j_1<j_2<...<j_m\leq n} (j_1j_2...j_m)^{-1} \leq \frac{1}{m!} \left(\sum_{j=1}^n \frac{1}{j}\right)^m.
 	\]
 	It holds that
 	\begin{equation}
 	\sum_{j=1}^n \frac{1}{j} \leq \int_1^n \frac{1}{x} dx +1 = \log(n)+1. \label{Rainbow16}
 	\end{equation}
 	The bounds developed in \eqref{iP9bound1}-\eqref{Rainbow16} yield for $n\neq 0$:
 	\[
 	\E \left[|A_i(x,s)|^r | N_i=n\right] \leq m^m \frac{1}{m!} (\log(n)+1)^m+(m-1)^m.
 	\]
 	For any $\epsilon' > 0$, the function $g:[1,+\infty) \rightarrow \R : x \mapsto \left(m^m \frac{1}{m!} (\log(x)+1)^m+(m-1)^m\right)x^{-\epsilon'}$ is bounded by a constant $c>0$ (dependent on $r$ and $\epsilon'$), implying that
 	\begin{equation}\label{Rainbow5}
 	\E \left[|A_i(x,s)|^r | N_i=n\right] \leq c n^{\epsilon'}.
 	\end{equation}
 	If $n=0$, then $|A_i(x,s)|=0$, hence the bound \eqref{Rainbow5} continues to hold. Therefore
 	\[
 	\E \left[|A_i(x,s)|^r | N_i\right] \leq c N_i^{\epsilon'}.
 	\]
 	Conditional on $\eta_{|tH \times [0,s)}$, the random variable $N_i$ is equal in law to $\eta'(tH \cap C_i(x) \cap \bb^d(x,R_s(x)) \times (s,1])$, where $\eta'$ is an independent copy of $\eta$. This quantity in turn is upper bounded by $\eta'(\bb^d(x,R_s(x)) \times [0,1])$. Hence
 	\begin{align}
 		\E \big[|A_i(x,s)|^r \big| \eta_{|tH \times [0,s)} \big] &\leq c \E \big[ \eta'(\bb^d(x,R_s(x)) \times [0,1])^{\epsilon'} \big| \eta_{|tH \times [0,s)} \big] \notag\\
 		&\leq c(\kappa_d R_s(x)^d)^{\epsilon'},
 	\end{align}
 	where we applied Jensen's inequality to pass to the second inequality. Plugging this bound into \eqref{iP9bound1}, we deduce
 	\begin{equation}\label{Rainbow6}
 	\E \left[ |A(x,s)|^r | \eta_{|tH \times [0,s)} \right] \leq cK^r \kappa_d^{\epsilon'} R_s(x)^{\epsilon' d}.
 	\end{equation}
 	Combining \eqref{Rainbow6} with \eqref{iP9dx} leads to
 	\[
 	\E\left[|\MalD_{(x,s)}\fat|^r\right] \leq c_0 \E R_s(x)^{r\alpha} (1+R_s(x)^{\epsilon' d}).
 	\]
 	for some constant $c_0>0$. By Lemma~\ref{lemONNGradius}, the RHS is bounded by $c_1 s^{-r\alpha/d-\epsilon'}$ for some constant $c_1>0$. For $\epsilon'=\epsilon r$, we get
 	\[
 	\E\left[|\MalD_{(x,s)}\fat|^r\right]^{1/r} \leq c_1^{1/r} s^{-\alpha/d-\epsilon}.
 	\]
 	
 	\textbf{Proof of \eqref{eqONNGdxybound2}.} As in the proof of \eqref{eqONNGdxybound}, we have
 	\[
 	|\MalD_{(x,s),(y,u)}\fat| \leq |\MalD_{(y,u)} \Fa(\eta_{tH})| + R_u(y)^\alpha + \opL_{(y,u)} \Fa (\eta_{tH})
 	\]
 	and the result follows by the proof of \eqref{eqONNGdxBound2}.
 	
 	\textbf{Proof of \eqref{eqONNGpBound2}.} As in the proof of \eqref{eqONNGpBound}, one see that if $|x-y|>3R_u(y)$, then
 	\[
 	\MalD_{(x,s),(y,u)}\fat = 0.
 	\]
 	Hence for some constants $c_1,c_2>0$,
 	\[
 	\p(\MalD_{(x,s),(y,u)}\fat \neq 0) \leq \p(|x-y|>3R_u(y)) \leq c_1 \exp(-c_2u|x-y|^d),
 	\]
 	as seen in Lemma~\ref{lemONNGradius}.
\end{proof}

\subsection{The orders of the variances}\label{subsecVar}

The goal of this subsection is to show the orders of the variances given in \eqref{eqONNGVarNC} and \eqref{eqONNGVarC}. For the sake of legibility, we will split the proof into several propositions.
\begin{prop}\label{propONNGVarUpper}
	There are constants $c_1,c_2>0$ such that for $0<\alpha<\frac{d}{2}$,
	\[
	\var (\fat) \leq c_1 t^d
	\]
	and
	\[
	\var(\fad) \leq c_2 t^d \log t.
	\]
\end{prop}
\begin{rem}
	The bounds in this proposition were already shown in \cite[Theorem~2.1]{Wade2009} via proving the result for uniform vectors and subsequent Poissonisation. We include a proof for the sake of completeness. It is similar in spirit to the one given in \cite{Wade2009}, but works in a purely Poisson setting.
\end{rem}
\begin{proof}
	Let $0<\alpha \leq \frac{d}{2}$. Since $\fat \leq (\diam(H)t)^\alpha \eta(tH \times [0,1])$, it is clear that $\fat \in L^2(\p_\eta)$. Hence by Lemma~\ref{lemCondCov}, we have
	\begin{equation}
	\var\left(\fat\right) = \E \int_{tH} \int_0^1 \, \E\big[\MalD_{(y,u)} \fat \big| \eta_{|tH \times [0,u)} \big]^2 \,dydu. \label{Rainbow17}
	\end{equation}
	Applying \eqref{iP6bound1} to the RHS of \eqref{Rainbow17} yields
	\begin{equation}
	\var\left(\fat\right) \leq c\ \E \int_{tH} \int_0^1 \, R_u(y,tH,\eta_tH)^{2\alpha}\, dydu \label{Rainbow18}
	\end{equation}
	for some constant $c>0$. Now applying Lemma~\ref{lemONNGradius} to the RHS of \eqref{Rainbow18} gives
	\[
	\var\left(\fat\right) \leq c' \int_{tH} \int_0^1 \, \left(t^{2\alpha} \wedge u^{-2\alpha/d}\right) dydu
	\]
	for some other constant $c'>0$. Integrating now yields the result.
\end{proof}
\begin{prop}\label{propONNGlowerVar1}
	For $\alpha>0$, there is a constant $c>0$ such that
	\begin{equation}\label{Rainbow19}
	c t^d \leq \var\left(\fat\right).
	\end{equation}
\end{prop}
\noindent The proof will make use of \cite[Theorem~5.2]{LPS14}.
\begin{proof}[Proof of Prop. \ref{propONNGlowerVar1}]
	Recall that there are $y_0 \in H$ and $\delta>0$ such that $B^d(y_0,\delta) \subset H$, as defined at the beginning of this Section~\ref{secPONNG}. Also recall the definition of the cones $C_1^+(0),...,C_K^+(0)$ from section~\ref{subsecCondExp}. For $i\in\{1,...,K\}$, define
	\[
	V_i:= C_i^+(0) \cap B^d(0,1) \setminus B^d\big(0,\tfrac{1}{2}\big).
	\]
	Now for $0<\tau<\frac{\delta}{2}$, define
	\begin{multline}
	U_\tau := \bigg\{ (x,s,x+\tau z_1,s_1,...,x+\tau z_K,s_K) : (x,s) \in tB^d\left(y_0,\tfrac{\delta}{2}\right)\times \left[\tfrac{1}{2},1\right],\\ \text{ and } (z_i,s_i) \in V_i \times \big[0,\tfrac{1}{2}\big),\ i=1,...,K \bigg\}.
	\end{multline}
	Note that $U_\tau \subset (tH\times [0,1])^{K+1}$. For \cite[Theorem~5.2]{LPS14} to yield a lower bound as in \eqref{Rainbow19}, we need to show that for a suitably chosen $\tau$ to be defined later,
	\begin{enumerate}
		\item There is a constant $c_1>0$ such that for all \label{iP7C1} $(x,s,\tilde{z}_1,s_1,...,\tilde{z}_K,s_K) \in U_\tau$,
		\[
		\left| \E \Fa\left(\eta_{tH} + \delta_{(x,s)} + \sum_{i=1}^K \delta_{(\tilde{z}_i,s_i)}\right) - \Fa\left(\eta_{tH} + \sum_{i=1}^K \delta_{(\tilde{z}_i,s_i)}\right) \right| \geq c_1.
		\]
		\item There is a constant $c_2>0$ such that \label{iP7C2}
		\[
		\min_{\emptyset \neq J \subset \{1,...,K+1\}} \inf_{\substack{V \subset U_\tau \\ \lambda^{(K+1)}(V)\geq \lambda^{(K+1)}(U_\tau)/2^{K+2}}} \lambda^{(|J|)}\left(\Pi_J(V)\right) \geq c_2t^d.
		\]
	\end{enumerate}
	\textbf{Proof of \ref{iP7C1}.} We use a construction similar to the one in \cite[Lemma~7.1]{LPS14} for the $k$-Nearest Neighbour graph. Let $(x,s) \in tB^d\left(y_0,\frac{\delta}{2}\right)\times \left[\frac{1}{2},1\right]$ and $(z_i,s_i) \in V_i \times \left[\left.0,\frac{1}{2}\right)\right.$ for all $i=1,...,K$. Define
	\[
	\mathcal{A}_\tau := \eta_{tH} + \sum_{i=1}^K \delta_{(x+\tau z_i,s_i)}.
	\]
	By the choice of the points $z_1,...,z_K$, we infer that
	\[
	R_s(x,tH,\mathcal{A}_\tau) \leq \tau
	\]
	and no point outside $B^d(x,\tau)$ will connect to $x$, since there is always a point $z_i$ which is closer (by Lemma~3.3 in \cite{Pen05}).
	
	If $\eta(\bb^d(x,\tau) \times [0,1]) = 0$, then no point at all will connect to $x$ and the only change upon addition of $x$ is the addition of the edge from $x$ to its online nearest neighbour. Since there is no point of $\eta$ in $\bb^d(x,\tau)$, the online nearest neighbour of $x$ must be one of the points $z_1,...,z_K$. But $|x-z_i|\geq \frac{\tau}{2}$ for all $i=1,...,K$ and we deduce
	\begin{equation}\label{iP7Dx1}
	\MalD_{(x,s)} \Fa (\mathcal{A}_\tau) \geq 2^{-\alpha}\tau^\alpha.
	\end{equation}
	If $\eta(\bb^d(x,\tau) \times [0,1]) \neq 0$, we use that by \eqref{eqDbyL} and the proof of Lemma~\ref{lemONNGFirstBounds},
	\begin{align}
		|\MalD_{(x,s)} \Fa (\mathcal{A}_\tau)| &\leq e(x,s,\mathcal{A}_\tau)^\alpha + \opL_{(x,s)}\Fa(\mathcal{A}_\tau) \notag\\
		&\leq R_s(x,tH,\mathcal{A}_\tau)^\alpha + (2R_s(x,tH,\mathcal{A}_\tau))^\alpha \eta(\bb^d(x,\tau) \times [0,1]) \notag\\
		&\leq \tau^\alpha + (2\tau)^\alpha \eta(\bb^d(x,\tau) \times [0,1])
	\end{align}
	to find the following bound:
	\begin{align}
	\left|\E \ind{\eta(\bb^d(x,\tau) \times [0,1]) \neq 0} \MalD_{(x,s)} \Fa (\mathcal{A}_\tau) \right| &\leq \p(\eta(\bb^d(x,\tau) \times [0,1]) \neq 0) \tau^\alpha + (2\tau)^\alpha \E \eta(\bb^d(x,\tau) \times [0,1]) \notag\\
	&= \tau^\alpha (1-\exp(-\kappa_d \tau^d)) + (2\tau)^\alpha \kappa_d \tau^d. \label{iP7Dx2}
	\end{align}
	Combining \eqref{iP7Dx1} and \eqref{iP7Dx2}, we find
	\begin{align}
	&\left|\E \MalD_{(x,s)} \Fa (\mathcal{A}_\tau) \right|\notag\\
	&\hspace{1cm}\geq \E \ind{\eta(B^d(x,\tau) \times [0,1]) = 0} \MalD_{(x,s)} \Fa (\mathcal{A}_\tau) - \left|\E \ind{\eta(B^d(x,\tau) \times [0,1]) \neq 0} \MalD_{(x,s)} \Fa (\mathcal{A}_\tau) \right| \notag\\
	&\hspace{1cm}\geq 2^{-\alpha}\tau^\alpha \exp(-\kappa_d \tau^d) - \tau^\alpha (1-\exp(-\kappa_d \tau^d)) - (2\tau)^\alpha \kappa_d \tau^d\notag\\
	&\hspace{1cm} = \tau^\alpha \left(\exp(-\kappa_d \tau^d) 2^{-\alpha} - (1-\exp(-\kappa_d \tau^d)) - 2^\alpha \kappa_d \tau^d\right)=:c_\tau.
	\end{align}
	But we have
	\[
	\lim_{\tau \rightarrow 0} \exp(-\kappa_d \tau^d) 2^{-\alpha} - (1-\exp(-\kappa_d \tau^d)) - 2^\alpha \kappa_d \tau^d = 2^{-\alpha} > 0,
	\]
	which means that we can choose $\tau > 0$ small enough such that $c_\tau>0$. This choice of $\tau$ depends only on $\alpha$ and $d$. We fix this $\tau$ for the rest of the proof.
	
	\textbf{Proof of \ref{iP7C2}.} We follow the same type of reasoning as was used in the proof of \cite[Theorem~5.3]{LPS14}. First, note that
	\[
	|U_\tau| = \kappa_d 2^{-d-K-1} \delta^d |V_1|^K t^d.
	\]
	Let $\emptyset \neq J = \{i_1,...,i_{|J|}\} \subset \{1,...,K+1\}$. For any $(y,u)=(y_1,u_1,...,y_{K+1},u_{K+1}) \in \left(\R^d \times [0,1]\right)^{K+1}$, write $(y,u)_J = (y_{i_1},u_{i_1},...,y_{i_{|J|}},u_{i_{|J|}})$. If $(y,u) \in U$, then for any $i,j \in \{1,...,K+1\}$, we have that $y_j \in B^d(y_i,2\tau)$ and $u_j \in \left[\left.0,\frac{1}{2}\right)\right.$. This implies that for any $(y,u)_J \in \left(\R^d \times [0,1]\right)^{|J|}$,
	\[
	\lambda^{(K+1-|J|)}\left((y,u)_{J^c} \in \left(\R^d \times [0,1]\right)^{K+1-|J|} : \left((y,u)_J,(y,u)_{J^c}\right) \in U \right) \leq \left(2^{d-1}\tau^d\kappa_d\right)^{K+1-|J|},
	\]
	where we use $\lambda$ for Lebesgue measure. Thus for any $V \subset U$,
	\begin{align}
		\lambda^{(K+1)}(V) &\leq \int_{\left(\R^d\times[0,1]\right)^{K+1}} \ind{(y,u)_J \in \Pi_J(V)} \ind{((y,u)_J,(y,u)_{J^c}) \in U} \lambda^{(K+1)}(d(y,u)) \notag\\
		&\leq \left(2^{d-1} \tau^d \kappa_d\right)^{K+1-|J|} \lambda^{(|J|)}\left(\Pi_J(V)\right).
	\end{align}
	Hence any $V \subset U$ with $\lambda^{(K+1)}(V) \geq \lambda^{(K+1)}(U)2^{-(K+2)}$ satisfies
	\[
	\lambda^{(|J|)}\left(\Pi_J(V)\right) \geq \left(2^{d-1} \tau^d \kappa_d\right)^{-(K+1-|J|)} 2^{-(K+2)} \kappa_d 2^{-d-K-1} \delta^d |V_1|^K t^d.
	\]
	This in turn is lower bounded by $c_2t^d$, where
	\[
	c_2:= \min\left\{1,\left(2^{d-1} \tau^d \kappa_d\right)^{-K}\right\} \kappa_d 2^{-d-2K-3} \delta^d |V_1|^K.
	\]
	This concludes the proof.
\end{proof}

\begin{prop}\label{propONNGlowerVar2}
	There are constants $c>0$ and $T_0 \geq 1$ such that for all $t\geq T_0$,
	\[
	ct^d \log(t) \leq \var \left(\fad\right).
	\]
\end{prop}
\begin{rem}
	By inspection of the arguments in Lemmas~\ref{lemSplitVar}-\ref{lemI_2'}, one sees that
	\begin{equation}\label{Rainbow35}
	d\left(1 - \tfrac{\pi}{2} + c(d)\right) \leq \liminf_{t \rightarrow \infty} \frac{\var \left(\fad\right)}{t^d \log(t)},
	\end{equation}
	where $c(d)$ is a positive constant such that $c(d) > \frac{\pi}{2}-1$ for $d \geq 1$. We believe this bound to correspond to the exact asymptotic order of $\var \left(\fad\right)$. We can however only provide a closed form of $c(d)$ in dimension $d=1$. For dimension $d=2$, we numerically estimate $c(2)$ and for dimensions $d \geq 3$, we use a lower bound which is smaller than the LHS of \eqref{Rainbow35}:
	\[
	d\left(1-\frac{\pi}{4}-\tilde{c}(d)\right) \leq \liminf_{t \rightarrow \infty} \frac{\var \left(\fad\right)}{t^d \log(t)},
	\]
	where $0<\tilde{c}(d)<1-\frac{\pi}{4}$ for $d \geq 3$ and $\tilde{c}(d) \rightarrow 0$ as $d \rightarrow \infty$.
\end{rem}
\begin{lem}\label{lemSplitVar}
	Let $\alpha>0$ and $t\geq 1$. Let $\ell:=\diam(H)$. Then
	\begin{equation}\label{eqvarONNG1}
		\var(\fat) = I_1(t) + I_2(t) - I_3(t) + I_4(t) - I_5(t),
	\end{equation}
	where the terms are defined as follows:
	\begin{align}
		I_1(t) &= \int_{tH}dy \int_0^1 dv \int_0^{(t\ell)^{2\alpha}} ds \left(\exp(-v|tH\cap B^d(y,s^{1/(2\alpha)})|)-\exp(-v|tH|)\right) \notag\\
		I_2(t) &= 2 \int_{tH}dx \int_{tH}dy \int_0^1dv \int_0^1du \int_0^{(t\ell)^\alpha}ds \int_0^{(t\ell)^\alpha}dr\ \ind{u<v} \ind{s<|x-y|^\alpha}\notag\\
		&\hspace{2cm}\exp\left(-u|tH\cap B^d(x,r^{1/\alpha})|\right) \exp\left(-v|tH\cap B^d(y,s^{1/\alpha})|\right)\notag\\ &\hspace{4cm}\left[\exp\left(u|tH\cap B^d(x,r^{1/\alpha})\cap B^d(y,s^{1/\alpha})|\right)-1\right]\notag\\
		I_3(t) &= 2 \int_{tH}dx \int_{tH}dy \int_0^1dv \int_0^1du \int_0^{(t\ell)^\alpha}ds \int_0^{(t\ell)^\alpha}dr\ \ind{u<v} \ind{s>|x-y|^\alpha}\notag\\
		&\hspace{2cm} \exp\left(-u|tH\cap B^d(x,r^{1/\alpha})|\right) \exp\left(-v|tH\cap B^d(y,s^{1/\alpha})|\right)\notag\\
		I_4(t) &= 2 \int_{tH}dx \int_{tH}dy \int_0^1dv \int_0^1du \int_0^{(t\ell)^\alpha}ds \int_0^{(t\ell)^\alpha}dr\ \ind{u+v\geq 1}\notag\\
		&\hspace{2cm} \exp\left(-u|tH \cap B^d(x,r^{1/\alpha})|\right) \exp\left(-v|tH|\right)\notag\\
		I_5(t) &= \int_{tH}dx \int_{tH}dy \int_0^1dv \int_0^1du \int_0^{(t\ell)^\alpha}ds \int_0^{(t\ell)^\alpha}dr \exp\left(-u|tH|\right) \exp\left(-v|tH|\right) \notag
	\end{align}
\end{lem}
\begin{proof}
	We start by pointing out a few identities. The functional $\fat$ can be written
	\begin{equation}\label{iP8fat}
		\fat = \int_{tH \times [0,1]} e(y,v,\eta)^\alpha \eta(dy,dv).
	\end{equation}
	Moreover, for $(y,v),(x,u) \in tH\times[0,1]$ with $u<v$, we have
	\begin{equation}\label{iP8edge}
		e(y,v,\eta)^\alpha = \ind{\eta(tH\times[0,v))\neq0} \int_0^{(t\ell)^\alpha}ds\ \ind{\eta(tH\cap B^d(y,s^{1/\alpha})\times[0,v))=0}
	\end{equation}
	and
	\begin{multline}\label{iP8edgeMore}
		e\left(y,v,\eta+\delta_{(x,u)}\right)^\alpha = |x-y|^\alpha \ind{\eta(tH\times[0,v))=0} \\+ \ind{\eta(tH\times[0,v))\neq0} \int_0^{|x-y|^\alpha}ds\ \ind{\eta(tH\cap B^d(y,s^{1/\alpha})\times[0,v))=0}.
	\end{multline}
	Note also that $e(y,v,\eta)=e(y,v,\eta_{|tH \times [0,v)})$, a fact we are going to use repeatedly.
	
	Combining \eqref{iP8fat} and \eqref{iP8edge} with Mecke equation \eqref{eqMecke} allows us now to calculate $\E \fat$:
	\begin{align}
		\E \fat &= \int_{tH}dy \int_0^1 dv \int_0^{(t\ell)^\alpha} dr\ \E \ind{\eta(tH\times[0,v))\neq0} \ind{\eta(tH\cap B^d(y,r^{1/\alpha})\times[0,v))=0}\notag\\
		&= \int_{tH}dy \int_0^1 dv \int_0^{(t\ell)^\alpha} dr\ \left(\exp\big(-v|tH \cap B^d(y,r^{1/\alpha})|\big)-\exp(-v|tH|)\right).\label{iP8expFat}
	\end{align}
	To calculate $\E \big[(\fat)^2\big]$, note first that
	\begin{multline}\label{iP8v}
		\E \left(\int_{tH \times [0,1]} e(y,v,\eta)^\alpha \eta(dy,dv)\right)^2 \\= \int_{tH}dy\int_0^1dv\ \E\ e(y,v,\eta)^{2\alpha} + \int_{tH}dx\int_{tH}dy\int_0^1du\int_0^1dv\ \E\ e(y,v,\eta+\delta_{(x,u)})^\alpha e(x,u,\eta+\delta_{(y,v)})^\alpha,
	\end{multline}
	as can be seen by applying Mecke equation twice or other means. The first term on the RHS of \eqref{iP8v} is equal to $\E \big[F_t^{(2\alpha)}\big]$, which by \eqref{iP8expFat} is equal to $I_1(t)$.

	The second term on the RHS of \eqref{iP8v} can by symmetry be written as
	\[
	2\int_{tH}dx\int_{tH}dy\int_0^1du\int_0^1dv\ \ind{u<v}\E\ e(y,v,\eta+\delta_{(x,u)})^\alpha e(x,u,\eta)^\alpha.
	\]
	Plugging in \eqref{iP8edge} and \eqref{iP8edgeMore} and using that the product $\ind{\eta(tH\times[0,u))\neq0} \ind{\eta(tH\times[0,v))=0}$ is zero, leads to
	\begin{multline}
		2\E\int_{tH}dx\int_{tH}dy\int_0^1du\int_0^1dv\int_0^{(t\ell)^\alpha}dr\int_0^{(t\ell)^\alpha}ds\ \ind{u<v} \ind{s<|x-y|^\alpha}\\  \ind{\eta(tH\times[0,u))\neq0} \ind{\eta(tH\cap B^d(x,r^{1/\alpha})\times[0,u))=0} \ind{\eta(tH\cap B^d(y,s^{1/\alpha})\times[0,v))=0}. \label{Rainbow20}
	\end{multline}
	Writing
	\[
	\ind{\eta(tH\cap B^d(y,s^{1/\alpha})\times[0,v))=0} = \ind{\eta(tH\cap B^d(y,s^{1/\alpha})\times[0,u))=0} \ind{\eta(tH\cap B^d(y,s^{1/\alpha})\times[u,v))=0}
	\] and taking expectation, using that $\eta_{|tH \times [0,u)}$ and $\eta_{|tH \times [u,v)}$ are independent, \eqref{Rainbow20} is equal to
	\begin{multline}
		2\int_{tH}dx\int_{tH}dy\int_0^1du\int_0^1dv\int_0^{(t\ell)^\alpha}dr\int_0^{(t\ell)^\alpha}ds\ \ind{u<v} \ind{s<|x-y|^\alpha} \exp\left(-(v-u)|tH \cap B^d(y,s^{1/\alpha})|\right) \\ \left(\exp\big(-u|tH \cap (B^d(x,r^{1/\alpha}) \cup B^d(y,s^{1/\alpha}))|\big)-\exp(-u|tH|)\right). \label{Rainbow21}
	\end{multline}
	Combining \eqref{Rainbow21} with \eqref{iP8expFat}, we have
	\[
	\var\left(\fat\right) = I_1(t) + A(t) + B(t) + C(t)
	\]
	with
	\begin{align*}
		&A(t) = 2 \int_{tH}dx \int_{tH}dy \int_0^1du \int_0^1dv \int_0^{(t\ell)^\alpha}dr \int_0^{(t\ell)^\alpha}ds\ \ind{u<v} \ind{s<|x-y|^\alpha}\\
		&\hspace{1cm}\exp\left(-(v-u)|tH \cap B^d(y,s^{1/\alpha})|\right)\exp\big(-u|tH \cap (B^d(x,r^{1/\alpha}) \cup B^d(y,s^{1/\alpha}))|\big) \\
		&B(t) = - 2 \int_{tH}dx \int_{tH}dy \int_0^1du \int_0^1dv \int_0^{(t\ell)^\alpha}dr \int_0^{(t\ell)^\alpha}ds\ \ind{u<v} \ind{s<|x-y|^\alpha}\notag\\  &\hspace{1cm}\exp\left(-(v-u)|tH \cap B^d(y,s^{1/\alpha})|\right)\exp(-u|tH|)\\
		&C(t) = -\left(\int_{tH}dy \int_0^1 dv \int_0^{(t\ell)^\alpha} ds\ \left(\exp\big(-v|tH \cap B^d(y,s^{1/\alpha})|\big)-\exp(-v|tH|)\right)\right)^2 \\
	\end{align*}
	It can easily be seen that
	\[
	C(t) = C_1(t) + C_2(t) - I_3(t) - I_5(t)
	\]
	with
	\begin{align*}
		C_1(t) &= -2\int_{tH}dx\int_{tH}dy\int_0^1du\int_0^1dv\int_0^{(t\ell)^\alpha}dr\int_0^{(t\ell)^\alpha}ds\ \ind{u<v} \ind{s<|x-y|^\alpha} \notag\\
		&\hspace{1cm} \exp(-v|tH \cap B^d(y,s^{1/\alpha})|) \exp(-u|tH \cap B^d(x,r^{1/\alpha})|)\\
		C_2(t) &= 2\int_{tH}dx\int_{tH}dy\int_0^1du\int_0^1dv\int_0^{(t\ell)^\alpha}dr\int_0^{(t\ell)^\alpha}ds\ \exp(-v|tH \cap B^d(y,s^{1/\alpha})|) \exp(-u|tH|)\\
	\end{align*}
	It remains to show that $I_2(t) + I_4(t) = A(t) + B(t) + C_1(t) + C_2(t)$. Let us first show that $A(t) + C_1(t) = I_2(t)$. Indeed, denote $R:= B^d(x,r^{1/\alpha})$ and $S:=B^d(y,s^{1/\alpha})$. Then
	\begin{multline}
	\exp(-(v-u)|tH \cap S|)\exp(-u|tH\cap(S\cup R)|)\\ = \exp(-v|tH \cap S|)\exp(-u|tH \cap R|)\exp(u|tH\cap S\cap R|)
	\end{multline}
	and hence
	\begin{multline}
		A(t) + C_1(t) = 2\int_{tH}dx\int_{tH}dy\int_0^1du\int_0^1dv\int_0^{(t\ell)^\alpha}dr\int_0^{(t\ell)^\alpha}ds\ \ind{u<v} \ind{s<|x-y|^\alpha} \\
		\exp(-v|tH \cap S|)\exp(-u|tH \cap R|)\left(\exp(u|tH\cap S\cap R|)-1\right) = I_2(t).
	\end{multline}
	To see that $B(t) + C_2(t) = I_4(t)$, it suffices to perform a change of variables in $B(t)$ by taking $\tilde{v}:=v-u$, $d\tilde{v}=dv$.
\end{proof}

In the following we are going to deal with the different terms in Lemma~\ref{lemSplitVar} in the case $\alpha=\frac{d}{2}$.
\begin{lem}\label{lemI_1}
	In the notation of Lemma~\ref{lemSplitVar} and with $\alpha=\frac{d}{2}$,
	\[
	I_1(t) \geq |H|\kappa_d^{-1} t^d\log(t^d) + O(t^d).
	\]
\end{lem}
\begin{proof}
	First, we show that the second part of the expression $I_1(t)$ is $O(t^d)$. Indeed,
	\[
	\int_{tH}dy \int_0^1 dv \int_0^{(t\ell)^{d}} ds \exp(-v|tH|) = t^d \ell^d (1-\exp(-t^d|H|)) = O(t^d).
	\]
	Note that since $|tH \cap B^d(y,s^{1/d})| \leq \kappa_d s$, the first part of the expression $I_1(t)$ is lower bounded by
	\[
	\int_{tH} dy \int_{t^{-d}}^1 dv \int_0^{(t\ell)^d}ds\  \exp(-v\kappa_d s),
	\]
	which after integrating in $y$ and $s$ and the change of variables $u=t^dv$, is equal to
	\[
	t^d|H|\kappa_d^{-1} \left(\log(t^d) - \int_1^{t^d} du\ u^{-1} \exp(-u\kappa_d \ell^d)\right).
	\]
	Since $\int_1^{\infty} du\ u^{-1} \exp(-u\kappa_d \ell^d) < \infty$, this term is equal to $t^d|H|\kappa_d^{-1} \log(t^d) + O(t^d)$, concluding the proof.
\end{proof}

\begin{lem}\label{lemI_3}
	In the notation of Lemma~\ref{lemSplitVar} and with $\alpha=\frac{d}{2}$, there is a constant $T_0\geq 1$ such that for all $t \geq T_0$, we have
	\[
	I_3(t) \leq |H|\kappa_d^{-1} \frac{\pi}{2} t^d\log(t^d)+O(t^d).
	\]
\end{lem}
See the proof of \cite[Lemma~3.6]{Wade2009} for a similar computation to the one done below.
\begin{proof}
	Start with a change of variables:
	\[
	\begin{cases}
		\tilde{x} = t^{-1}x,\ dx = t^d d\tilde{x} \\
		\tilde{y} = t^{-1}y,\ dy = t^d d\tilde{y} \\
		\tilde{s} = t^{-d/2}s,\ ds = t^{d/2} d\tilde{s} \\
		\tilde{r} = t^{-d/2}r,\ dr = t^{d/2} d\tilde{r} \\
		\tilde{u} = t^du,\ du = t^{-d} d\tilde{u} \\
		\tilde{v} = t^dv, \ dv = t^{-d} d\tilde{v},
	\end{cases}
	\]
	which leads to
	\begin{multline}
		I_3(t) = 2t^d \int_{H}dx \int_{H}dy \int_0^{t^d}dv \int_0^vdu \int_0^{\ell^{d/2}}ds \int_0^{\ell^{d/2}}dr\ \ind{s>|x-y|^{d/2}}\\
		\exp\left(-u|H\cap B^d(x,r^{2/d})|\right) \exp\left(-v|H\cap B^d(y,s^{2/d})|\right).
	\end{multline}
	We can reduce the integration interval of the variable $v$ to $[\tau_0,t^d]$, for some large constant $\tau_0$ and large enough $t$, since the rest term is clearly $O(t^d)$. For any $v \geq \tau_0$, we then have $|H \setminus H_{2v^{-1/(2d)}}| \leq \beta_H 2 v^{-1/(2d)}$ by the discussion at the start of Section~\ref{secONNG}.
	
	We can now split the above expression as follows:
	\begin{align}
		I_3(t) &=%
		2t^d \int_0^{\tau_0}dv \int_0^vdu \int_{H}dx \int_{H}dy \int_0^{\ell^{d/2}}ds \int_0^{\ell^{d/2}}dr\ \ind{s>|x-y|^{d/2}}\notag\\
		&\hspace{2cm}\exp\left(-u|H\cap B^d(x,r^{2/d})|\right) \exp\left(-v|H\cap B^d(y,s^{2/d})|\right) \notag\\
		%
		&+2t^d \int_{\tau_0}^{t^d}dv \int_0^vdu \int_{H \setminus H_{2v^{-1/(2d)}}}dx \int_{H}dy \int_0^{\ell^{d/2}}ds \int_0^{\ell^{d/2}}dr\ \ind{s>|x-y|^{d/2}}\notag\\
		&\hspace{2cm}\exp\left(-u|H\cap B^d(x,r^{2/d})|\right) \exp\left(-v|H\cap B^d(y,s^{2/d})|\right)\notag\\
		%
		&+2t^d \int_{\tau_0}^{t^d}dv \int_0^vdu \int_{H_{2v^{-1/(2d)}}}dx \int_{H}dy \int_{v^{-1/4}}^{\ell^{d/2}}ds \int_0^{\ell^{d/2}}dr \ \ind{s>|x-y|^{d/2}}\notag\\
		&\hspace{2cm}\exp\left(-u|H\cap B^d(x,r^{2/d})|\right) \exp\left(-v|H\cap B^d(y,s^{2/d})|\right)\notag\\
		%
		&+2t^d \int_{\tau_0}^{t^d}dv \int_0^vdu \int_{H_{2v^{-1/(2d)}}}dx \int_{H}dy \int_0^{v^{-1/4}}ds \int_{v^{-1/4}}^{\ell^{d/2}}dr \ \ind{s>|x-y|^{d/2}}\notag\\
		&\hspace{2cm}\exp\left(-u|H\cap B^d(x,r^{2/d})|\right) \exp\left(-v|H\cap B^d(y,s^{2/d})|\right) \notag\\
		%
		&+2t^d \int_{\tau_0}^{t^d}dv \int_0^vdu \int_{H_{2v^{-1/(2d)}}}dx \int_{H}dy \int_0^{v^{-1/4}}ds \int_0^{v^{-1/4}}dr \ \ind{s>|x-y|^{d/2}}\notag\\
		&\hspace{2cm}\exp\left(-u|H\cap B^d(x,r^{2/d})|\right) \exp\left(-v|H\cap B^d(y,s^{2/d})|\right)\notag\\
		&= O(t^d) + R_1(t) + R_2(t) + R_3(t) + I_3'(t)
	\end{align}
	Note that in $I_3'(t)$, due to the choice of sets we are integrating over, we have $B^d(x,r^{2/d}) \subset H$ and $B^d(y,s^{2/d}) \subset H$. Hence $I_3'(t)$ is upper bounded by
	\[
	2t^d \int_{1}^{t^d}dv \int_0^vdu \int_0^{\infty}ds \int_0^{\infty}dr \int_{H}dx \int_{B^d(x,s^{2/d})}dy\ \exp\left(-u\kappa_d r^2\right) \exp\left(-v\kappa_d r^2 \right)
	\]
	which is equal to $|H|\kappa_d^{-1} \frac{\pi}{2} t^d\log(t^d)$. Hence
	\[
	I_3'(t) \leq |H|\kappa_d^{-1} \tfrac{\pi}{2} t^d\log(t^d).
	\]
	It remains to show that the three rest terms $R_1(t),R_2(t),R_3(t)$ are $O(t^d)$. To this end, recall from Lemma~\ref{lemGeo1} that there is a constant $c_H>0$ such that $|H \cap B^d(w,s)| \geq c_H s^d$ for any $w \in H$ and $s \leq \ell$. This implies that
	\[
	R_1(t) \leq 2t^d \int_{\tau_0}^{\infty}dv \int_0^vdu \int_0^{\infty}ds \int_0^{\infty}dr \int_{H \setminus H_{2v^{-1/(2d)}}}dx \int_{B^d(x,s^{2/d})}dy\ \exp\left(-uc_H r^2\right) \exp\left(-vc_H s^2\right)
	\]
	which in turn is upper bounded by 
	\[
	t^d \pi \kappa_d c_H^{-2} \beta_H \int_{\tau_0}^\infty v^{-1-1/(2d)} = O(t^d).
	\]
	By the same reasoning, we get
	\begin{align}
	R_2(t) \leq 2|H| \kappa_d t^d \int_{\tau_0}^{\infty}dv \int_0^vdu \int_{v^{-1/4}}^{\infty}ds \int_0^{\infty}dr\ s^2 \exp\left(-uc_H r^2\right) \exp\left(-vc_H s^2\right)
	\intertext{and}
	R_3(t) \leq 2|H| \kappa_d t^d \int_{\tau_0}^{\infty}dv \int_0^vdu \int_0^{\infty}ds \int_{v^{-1/4}}^{\infty}dr\ s^2 \exp\left(-uc_H r^2\right) \exp\left(-vc_H s^2\right).
	\end{align}
	both of which can be shown to be $O(t^d)$.
\end{proof}
\begin{lem}\label{lemI4I5}
	In the notation of Lemma~\ref{lemSplitVar} and with $\alpha=\frac{d}{2}$, one has $I_4(t) = O(t^d)$ and $I_5(t) = O(t^d)$. 
\end{lem}
\begin{rem}
	Since $I_4(t) \geq 0$, it is not necessary to calculate the order of this term to find a lower bound of $\var (\fad)$. However, we include the proof for completeness.
\end{rem}
\begin{proof}
	Perform the same change of variables and upper bound as in the proof of Lemma~\ref{lemI_3} to find
	\[
		I_4(t) \leq 2 t^d \int_{H}dx \int_{H}dy \int_0^{t^d}dv \int_0^{t^d}du \int_0^{\ell^{d/2}}ds \int_0^{\ell^{d/2}}dr\ \ind{u+v\geq t^d}\\
		\exp\left(-uc_Hr^2\right) \exp\left(-v|H|\right).
	\]
	Integrating over $x$, $y$, $s$ and $v$ yields the upper bound
	\[
	I_4(t) \leq 2 |H| \ell^{d/2} t^d \int_0^{t^d} du \int_0^{\ell^{d/2}} dr\ \exp(-uc_Hr^2) \left(\exp(-(t^d-u)|H|)-\exp(-t^d|H|)\right).
	\]
	The integrand is bounded by $\exp(-(t^d-u)|H|)$. Introducing the change of variable $\tilde{u} = t^d-u$ yields the upper bound
	\[
	I_4(t) \leq 2 |H| \ell^{d/2} t^d \int_0^{t^d} du \int_0^{\ell^{d/2}} dr\ \exp(-u|H|) \leq 2 \ell^d t^d = O(t^d).
	\]
	As for $I_5(t)$, integrating over all variables yields that
	\[
	I_5(t) = t^d \ell^d (1-\exp(-t|H|))^2 = O(t^d).
	\]
\end{proof}

\begin{lem}[Joint work with Pierre Perruchaud]\label{lemI_2}
	In the notation of Lemma~\ref{lemSplitVar} and with $\alpha=\frac{d}{2}$, for every $\epsilon > 0$, there is a constant $T_0\geq 1$ such that for all $t \geq T_0$,
	\[
	I_2(t) \geq (|H|\kappa_d^{-1}c(d)-\epsilon) t^d\log(t^d) + O(t^d),
	\]
	where
	\[
	c(d) = \int_{\R^d} dz \int_0^\infty dr\ \ind{|z| \geq 1}\\	\left(\left( |B^d(0,r^{2/d}) \cup B^d(z,1)|\right)^{-1} - \kappa_d^{-1} (r^2 + 1)^{-1}\right).
	\]
\end{lem}

\begin{proof}
	Fix $\tilde{\epsilon}>0$ such that $(|H|-\tilde{\epsilon})(c(d)-\tilde{\epsilon}) \geq |H|c(d) - \epsilon$ and let $\tilde{\delta} > 0$ be such that $\beta_H \tilde{\delta} < \tilde{\epsilon}$ and $|H_{\tilde{\delta}}| \geq |H|-\tilde{\epsilon}$ (which is possible by property~\eqref{eqHcond}).
	
	Assume that $x \in tH_{\tilde{\delta}}$, $y\in B^d(x,\frac{t\tilde{\delta}}{2})$, $r \leq (t\tilde{\delta})^{d/2}$ and $s\leq |x-y|^{d/2}$. Then $B^d(x,t\tilde{\delta}) \subset tH$, $B^d(x,r^{2/d}) \subset tH$ and $B^d(y,s^{2/d}) \subset tH$. In the integrals making up $I_2(t)$, we can reduce the intervals of integration to the ones stated here and hence $I_2(t)$ is lower bounded by
	\begin{multline}
		2\int_{tH_{\tilde{\delta}}} dx \int_{B^d(x,\frac{t\tilde{\delta}}{2})} dy \int_0^{(t\tilde{\delta})^{d/2}} dr \int_0^{|x-y|^{d/2}} ds \int_0^1 dv \int_0^v du \\
		\exp(-u\kappa_d r^2) \exp(-v\kappa_d s^2) \left(\exp(u|B^d(0,r^{2/d}) \cap B^d(y-x,s^{2/d})|) - 1 \right).
	\end{multline}
	Now carry out the following changes of variables:
	\[
	\begin{cases}
		\tilde{u} = v^{-1}u,\ d\tilde{u} = v^{-1} du\\
		z = v^{1/d} (y-x),\ dz= v dy\\
		\tilde{s} = v^{1/2} s,\ d\tilde{s} = v^{1/2} ds \\
		\tilde{r} = v^{1/2} r,\ d\tilde{r} = v^{1/2} dr
	\end{cases}
	\]
	and deduce that the above expression is equal to
	\begin{multline}
		2\int_0^1 dv \int_0^1 d\tilde{u} \int_{tH_{\tilde{\delta}}} dx \int_{B^d(0,tv^{1/d}\tilde{\delta}/2)} dz \int_0^{(tv^{1/d}\tilde{\delta})^{d/2}}d\tilde{r} \int_0^{|z|^{d/2}} d\tilde{s} \\
		v^{-1} \exp(-\tilde{u}\kappa_d \tilde{r}^2) \exp(-\kappa_d \tilde{s}^2) \left(\exp(\tilde{u}|B^d(0,\tilde{r}^{2/d}) \cap B^d(z,\tilde{s}^{2/d})|) - 1 \right)
	\end{multline}
	Assume now that $t \geq \tilde{\delta}^{-1}\tau_0=:T_0$ for some $\tau_0>0$ and lower bound this expression by reducing to the integration interval where $v^{1/d} \geq t^{-1}\tilde{\delta}^{-1}\tau_0$. Integrating additionally over $x$, we get the lower bound
	\begin{multline}
		t^d|H_{\tilde{\delta}}|\left(\int_{(t^{-1}\tilde{\delta}^{-1}\tau_0)^d}^1 dv\ v^{-1}\right) 2\int_0^1 du \int_{B^d(0,\tau_0/2)} dz \int_0^{\tau_0^{d/2}} dr \int_0^{|z|^{d/2}} ds \\
		\exp(-u\kappa_d r^2) \exp(-\kappa_d s^2)\left(\exp(u|B^d(0,r^{2/d}) \cap B^d(z,s^{2/d})|) - 1 \right).
	\end{multline}
	For $\tau_0$ large enough, this is lower bounded by
	\[
	t^d (|H|-\tilde{\epsilon})(\log(t^d)-d\log(\tilde{\delta}^{-1}\tau_0)) (c_0(d)-\tilde{\epsilon}) \geq (|H|c_0(d)-\epsilon) t^d \log(t^d) + O(t^d)
	\]
	with
	\begin{multline}
		c_0(d) := 2\int_0^1 du \int_{\R^d} dz \int_0^{\infty} dr \int_0^{|z|^{d/2}} ds \exp(-u\kappa_d r^2) \exp(-\kappa_d s^2) \\ \left(\exp(u|B^d(0,r^{2/d}) \cap B^d(z,s^{2/d})|) - 1 \right).
	\end{multline}
	The last thing that remains to do is to show that $c_0(d) = \kappa_d^{-1} c(d)$. Perform the successive change of variables $\tilde{r} = u^{1/2} r,\ \tilde{z} = u^{1/d} z,\ \tilde{s} = u^{1/2} s$ and $\tilde{u}=u^{-1}$, then integrate over $\tilde{u}$ to get
	\begin{multline}
		c_0(d) = 2\int_{\R^d} dz \int_0^\infty dr \int_0^{|z|^{d/2}} ds\ \kappa_d^{-1} s^{-2}\exp(-\kappa_d s^2) \exp(-\kappa_d r^2) \\   \left(\exp(|B^d(0,r^{2/d}) \cap B^d(z,s^{2/d})|) - 1 \right).
	\end{multline}
	Then change variables $\tilde{z} = s^{-2/d}z$ and $\tilde{r} = s^{-1} r$ and integrate over $s$ to find $c_0(d) = \kappa_d^{-1} c(d)$.
\end{proof}

Combining Lemmas~\ref{lemI_1}, \ref{lemI_3}, \ref{lemI4I5} and \ref{lemI_2}, we have shown that for every $\epsilon>0$, there is a $T_0\geq 1$ such that for all $t\geq T_0$,
\[
\var \left(\fad\right) \geq t^d \log(t^d) \left[|H|\kappa_d^{-1} \left(1+c(d)-\frac{\pi}{2}\right) - \epsilon \right] + O(t^d).
\]
It remains thus to show that for $d \in \N$,
\[
c(d)>\frac{\pi}{2}-1 \approx 0.57,
\]
which will be shown in the following lemma.

\begin{lem}[Joint work with Pierre Perruchaud] \label{lemI_2'}
	In the notation of Lemma~\ref{lemI_2}, for $d \in \N$,
	\[
	c(d)>\frac{\pi}{2}-1 \approx 0.57.
	\]
\end{lem}

\begin{proof}

	We start by showing a lower bound on $c(d)$. Indeed, for all $x \in (0,1)$,
	\[
	\frac{1}{1-x} - 1 = \frac{x}{1-x} \geq x.
	\]
	Hence
	\begin{align}
	&\frac{1}{|B^d(0,r^{2/d}) \cup B^d(z,1)|} - \frac{1}{\kappa_d(r^2+1)}\notag\\
	&= \kappa_d^{-1}(r^2+1)^{-1} \left(\frac{1}{1-\kappa_d^{-1}(r^2+1)^{-1}|B^d(0,r^{2/d}) \cap B^d(z,1)|} - 1\right) \notag\\
	&\geq \kappa_d^{-2}(r^2+1)^{-2} |B^d(0,r^{2/d}) \cap B^d(z,1)|.
	\end{align}
	A lower bound for $c(d)$ is thus given by
	\begin{equation}\label{iP9cdLower}
	 \int_0^\infty dr\ \kappa_d^{-2} (r^2 + 1)^{-2} \int_{\R^d} dz\ \ind{|z| \geq 1} |B^d(0,r^{2/d}) \cap B^d(z,1)|.
	\end{equation}
	Looking only at the inner integral over $z$, note that we can rewrite it as follows:
	\begin{align}
		\int_{\R^d} dz\ \ind{|z| \geq 1} |B^d(0,r^{2/d}) \cap B^d(z,1)| &= \int_{\R^d} dz \int_{\R^d} dx\ \ind{|z| \geq 1} \ind{|x| \leq r^{2/d}} \ind{|x-z| \leq 1} \notag\\
		&= \int_{\R^d} dx\ \ind{|x| \leq r^{2/d}} \int_{\R^d} dz\ \left(\ind{|x-z| \leq 1} - \ind{|x-z| \leq 1} \ind{|z| \leq 1}\right) \notag\\
		&= \int_{B^d(0,r^{2/d})} dx \ \left( \kappa_d - |B^d(0,1) \cap B^d(x,1)| \right).
	\end{align}
	Plugging this into the lower bound \eqref{iP9cdLower}, we obtain
	\[
	\int_0^\infty dr\ \kappa_d^{-2} (r^2 + 1)^{-2} \int_{B^d(0,r^{2/d})} dx \ \left( \kappa_d - |B^d(0,1) \cap B^d(x,1)| \right),
	\]
	which can be written as
	\[
	\int_0^\infty dr\ (r^2 + 1)^{-2} r^2 - \tilde{c}(d) = \frac{\pi}{4} - \tilde{c}(d)
	\]
	with
	\[
	\tilde{c}(d) := \int_0^\infty dr\ \kappa_d^{-2} (r^2 + 1)^{-2} \int_{B^d(0,r^{2/d})} dx \ |B^d(0,1) \cap B^d(x,1)|.
	\]
	Hence $c(d) \geq \frac{\pi}{4} - \tilde{c}(d)$ and to show that $c(d) \geq \frac{\pi}{2} - 1 $, it suffices to show that $\tilde{c}(d) \leq 1-\frac{\pi}{4} \approx 0.215$. First, note that
	\[
	|B^d(0,1) \cap B^d(x,1)| = 2 \kappa_{d-1} \ind{|x| \leq 2} \int_{\frac{|x|}{2}}^1 dy\ (1-y^2)^{\frac{d-1}{2}}.
	\]
	which can be seen by integrating over the $d-1$-dimensional hyperspheres making up the spherical caps of the intersection, or by following the development in \cite{SphericalCaps}.
	Hence $\tilde{c}(d)$ can be written
	\begin{align}
	\tilde{c}(d) &= 2 \kappa_{d-1} \kappa_d^{-2} \int_0^1 dy \int_0^\infty dr \int_{\R^d} dx\ \ind{|x| \leq r^{2/d} \wedge 2y} (r^2 + 1)^{-2} (1-y^2)^{\frac{d-1}{2}} \notag\\
	&= 2 \kappa_{d-1} \kappa_d^{-1} \int_0^1 dy \int_0^\infty dr\ \left(r^{2} \wedge (2y)^d\right) (r^2 + 1)^{-2} (1-y^2)^{\frac{d-1}{2}},
	\end{align}
	where we integrated over $x$ in the second line. For any $u \geq 0$, define
	\[
	g(u) := 2 \int_0^\infty dr\ (r^2 \wedge u^2) (r^2 + 1)^{-2} = \frac{\pi}{2}u^2 + (1-u^2) \arctan(u) - u.
	\]
	Then it can be verified by standard methods that $g(0)=0$, as well as $\lim_{u \rightarrow \infty} g(u) = \frac{\pi}{2}$ and the function is strictly increasing. Moreover, $g'(u) \leq \pi u $, therefore $g(x) \leq \frac{\pi}{2} x^2$.
	
	We now have
	\[
	\tilde{c}(d) = \kappa_{d-1} \kappa_d^{-1} \int_0^1 dy\ (1-y^2)^{\frac{d-1}{2}} g\left((2y)^{d/2}\right).
	\]
	Change variables $u = \frac{1-y}{2}$ to get
	\[
	\tilde{c}(d) = 2^d \kappa_{d-1} \kappa_d^{-1} \int_0^{1/2} du\ u^{\frac{d-1}{2}} (1-u)^{\frac{d-1}{2}} g\left(2^{d/2}(1-2u)^{d/2}\right)
	\]
	and note that
	\begin{equation}\label{Rainbow22}
	2^d \kappa_{d-1} \kappa_d^{-1} \int_0^{x} du\ u^{\frac{d-1}{2}} (1-u)^{\frac{d-1}{2}} = I_x\left(\frac{d+1}{2},\frac{d+1}{2}\right),
	\end{equation}
	where $I_x(a,b)$ is the regularized incomplete beta function (see \cite{RegBeta} for details).
	
	Let us now deal with dimensions $d \in \{3,...,9\}$. Split the integration interval $[0,\frac{1}{2}]$ into intervals $[\frac{i-1}{40},\frac{i}{40}]$ for $i \in \{1,2,...,20\}$ and upper bound $g\left(2^{d/2}(1-2u)^{d/2}\right)$ by $g\left(2^{d/2}(1-\frac{i-1}{10})^{d/2}\right)$ for $u \in (\frac{i-1}{40},\frac{i}{40}]$. We deduce the following bound:
	\[
	\tilde{c}(d) \leq \sum_{i=1}^{20} g\left(\left(2-\frac{i-1}{10}\right)^{d/2}\right) \left(I_{\frac{i}{40}}\left(\frac{d+1}{2},\frac{d+1}{2}\right) - I_{\frac{i-1}{40}}\left(\frac{d+1}{2},\frac{d+1}{2}\right)\right) =: \beta_1(d)
	\]
	This results in the values in Table~\ref{tabBeta1}, all of which are smaller than $1-\frac{\pi}{4} \approx 0.215$.
	\begin{table}
		\centering
	\begin{tabular}{|c|c|}
	\hline
	$d$ & $\beta_1(d)$ \\
	\hline
	3 & 0.203 \\
	4 & 0.175 \\
	5 & 0.150 \\
	6 & 0.128 \\
	7 & 0.110 \\
	8 & 0.094 \\
	9 & 0.081 \\
	\hline
	\end{tabular}	
	\caption{Values of $\beta_1(d)$}	\label{tabBeta1}
	\end{table}

	For $d \geq 10$, take $\theta = 0.32$ and split $\tilde{c}(d)$ as follows:
	\begin{align}
	\tilde{c}(d) &\begin{multlined}[t] {}= 2^d \kappa_{d-1} \kappa_d^{-1} \bigg(\int_0^{\theta} du\ u^{\frac{d-1}{2}} (1-u)^{\frac{d-1}{2}} \underbrace{g\left(2^{d/2}(1-2u)^{d/2}\right)}_{\leq \frac{\pi}{2}} \\ + \int_\theta^{1/2} du\ u^{\frac{d-1}{2}} (1-u)^{\frac{d-1}{2}} \underbrace{g\left(2^{d/2}(1-2u)^{d/2}\right)}_{\leq \pi 2^{d-1}(1-2\theta)^d} \bigg)
	\end{multlined} \notag\\
	&\leq \frac{\pi}{2} I_\theta\left(\frac{d+1}{2},\frac{d+1}{2}\right) + \pi 2^{d-1} (1-2\theta)^d I_{1/2}\left(\frac{d+1}{2},\frac{d+1}{2}\right),
	\end{align}
	where the second line follows by \eqref{Rainbow22}. Since $I_{1/2}\left(\frac{d+1}{2},\frac{d+1}{2}\right) = \frac{1}{2}$, it follows that $\tilde{c}(d)$ is upper bounded by
	\[
	\beta_2(d):= \frac{\pi}{2} \left(I_\theta\left(\frac{d+1}{2},\frac{d+1}{2}\right) + 2^{d-1}(1-2\theta)^d\right).
	\]
	Now $\beta_2(10) \approx 0.208 < 1-\frac{\pi}{4}$. By Proposition~4 in \cite{LatticeBeta}, for any $x \in (0,\frac{1}{2})$, the function $\alpha \mapsto I_x(\alpha,\alpha)$ is decreasing. Hence $\beta_2(d)$ is decreasing in $d$ and thus for all $d \geq 10$, one has $\beta_2(d) \leq \beta_2(10)$.
	
	For dimension $d=1$, it is possible to show that $c(1) = \left(\frac{5}{2}-\sqrt{2}\right) \pi - 2 \sqrt{2} \approx 0.583 > \frac{\pi}{2}-1$. For dimension $d=2$, one can numerically estimate that $c(2) \approx 0.606>\frac{\pi}{2}-1$.

The following python code was used to carry out this estimation with an approximate error of $1.5 \times 10^{-8}$. It can be used to estimate $c(d)$ at other small values of $d$.

\begin{lstlisting}[language=Python]
	import math
	from math import gamma
	import scipy as sc
	import scipy.special
	import scipy.integrate
	import numpy as np
	
	# volume(spherical cap)/volume(unit ball)
	# r: radius, a: base distance to cap, d: dimension
	def capvol(r,a,d): 
		if a>=0:
			return 1/2*(r**d)*sc.special.betainc((d+1)/2,1/2,1-a**2/r**2) # smaller cap
		else:
			return r**d-capvol(r,-a,d) #larger cap
	
	# volume(intersection of balls)/volume(unit ball)
	# x: distance between centres of balls, r1,r2: radii, d: dimension
	def vol(x,r1,r2,d):
		if x >= r1+r2: #no intersection
			return 0
		elif x <= abs(r1-r2): #one ball within the other
			return min(r1,r2)**d
		else:
			c1 = (x**2+r1**2-r2**2)/(2*x) #distances to bases of caps
			c2 = (x**2-r1**2+r2**2)/(2*x)
			return capvol(r1,c1,d)+capvol(r2,c2,d) #sum of both spherical caps
	
	# the integrand
	# r,a: variables, d: dimension
	def integr(r,a,d):
		q1 = 1+a**d-vol(r,1,a,d)
		q2 = 1 + a**d
		return 2*d**2/4*r**(d-1)/a**(d/2+1)*(1/q1-1/q2)
	
	# the constant c(d)
	def cst(d):
		options={'limit':200}
		res = sc.integrate.nquad(lambda a,r:integr(r,a,d),\
			[lambda r:[0,r],[0,np.inf]],opts=[options,options])
		return res #returns the result and the maximal error made
\end{lstlisting}
\end{proof}
\begin{proof}[Proof of Proposition~\ref{propONNGlowerVar2}]
	Combine Lemmas~\ref{lemSplitVar} - \ref{lemI_2'}.
\end{proof}

\subsection{Proof of Theorem~\ref{thmONNG}}

The inequalities~\eqref{eqONNGVarNC} and \eqref{eqONNGVarC} are shown in Propositions~\ref{propONNGVarUpper}, \ref{propONNGlowerVar1} and \ref{propONNGlowerVar2}. In the following, we will write $c$ to indicate the presence of a constant. The value of $c$ might change from line to line, or indeed, within one line.

\noindent We have for $0<\alpha\leq \frac{d}{2}$:
\[
|\fat| \leq (\diam(H)t)^{\alpha} \eta(tH \times [0,1])
\]
and using \eqref{eqDbyL},
\[
|\MalD_{(x,s)} \fat| \leq (\diam(H)t)^{\alpha} (\eta(tH \times [0,1])+1),
\]
hence $\fat \in L^2(\p_\eta) \cap \dom\MalD$.

\subsubsection{Wasserstein distance when $\alpha=\frac{d}{2}$}
{\setlength\parindent{0pt}
We use the bound given in Theorem~\ref{thmWasserCond} with $(\Y,\bar{\lambda}) = (tH \times [0,1],dx \otimes ds)$ and $p=q=2$.

Recall that, combining Propositions~\ref{propONNGcondBounds} and \ref{propONNGlowerVar2}, we have for $r \geq 1$ the following bounds:
\begin{align}
	&\E \left[ \E\big[ |\MalD_{(x,s)} \fad| \big| \eta_{|tH \times [0,s)} \big]^r \right]^{1/r} \lesssim s^{-1/2} \wedge t^{d/2} \label{R1} \\
	&\E \left[ \E \big[ |\DD_{(x,s),(y,u)} \fad| \big| \eta_{|tH \times [0,s \vee u)} \big]^4 \right]^{1/4} \lesssim (s \vee u)^{-1/2} \exp(-c(s \vee u)|x-y|^d),\label{R2}\\
	&\var\left(\fad\right) \gtrsim t^d \log(t), \label{R3}
\end{align}
where for the second line we used the Cauchy-Schwarz inequality to get
\begin{multline}
\E \left[ \E \big[ |\DD_{(x,s),(y,u)} \fad| \big| \eta_{|tH \times [0,s \vee u)} \big]^4 \right]^{1/4} \\ \leq \E \left[ \E \big[ |\DD_{(x,s),(y,u)} \fad| \big| \eta_{|tH \times [0,s \vee u)} \big]^8 \right]^{1/8} \p(\E \big[ |\DD_{(x,s),(y,u)} \fad| \big| \eta_{|tH \times [0,s \vee u)} \big] \neq 0)^{1/2}.
\end{multline}
We start by plugging the bounds \eqref{R1}, \eqref{R2}, \eqref{R3} into $\beta_1$ from Theorem~\ref{thmWasserCond}. Then

\[
\beta_1 \lesssim \left(t^d \log(t)\right)^{-1} \left( \int_{tH}dx \int_0^1 ds \left( \int_{tH}dy \int_0^1 du\ u^{-1/2} (u\vee s)^{-1/2} \exp\left(-c(u\vee s)|x-y|^d\right) \right)^2 \right)^{1/2}.
\]

We now change variables as follows:
\[
\begin{cases}
	\tilde{x} = t^{-1}x,\ dx = t^d d\tilde{x}\\
	\tilde{y} = t^{-1}y,\ dy = t^d d\tilde{y}\\
	\tilde{s} = t^ds,\ ds = t^{-d} d\tilde{s} \\
	\tilde{u} = t^du,\ du = t^{-d} d\tilde{u}
\end{cases}
\]
and deduce that $\beta_1$ is bounded by
\[
c \log(t)^{-1} \left( \int_{H}dx \int_0^\infty ds \left( \int_{H}dy \int_0^\infty du\ u^{-1/2} (u\vee s)^{-1/2} \exp(-c(u\vee s)|x-y|^d) \right)^2 \right)^{1/2}.
\]
This is $O(\log(t)^{-1})$ since the integral is finite.
Indeed, changing variables in $y$ and integrating over $x$, the integral is bounded by
\[
|H| \int_0^\infty ds \left( \int_{B^d(0,2\diam(H))} dz \int_0^\infty du\ u^{-1/2} (u\vee s)^{-1/2} \exp(-c(u\vee s)|z|^d) \right)^2.
\]
Integrating over $z$, this is equal to
\begin{equation}\label{Rainbow23}
c\int_0^\infty ds \left( \int_0^\infty du\ u^{-1/2} (u\vee s)^{-3/2} (1-\exp(-c(u\vee s))) \right)^2.
\end{equation}
Now we use that $1-\exp(-c(u\vee s)) \lesssim 1 \wedge (u\vee s)$ to bound \eqref{Rainbow23} by
\begin{equation}\label{Rainbow24}
c\int_0^\infty ds \left( \int_0^\infty du\ u^{-1/2} (u\vee s)^{-3/2} (1 \wedge (u\vee s)) \right)^2.
\end{equation}
Splitting the integral over $s$ and integrating over $u$, yields that \eqref{Rainbow24} is equal to
\[
c \int_0^1 ds \ (3-\log(s))^2 +c \int_1^\infty ds\ s^{-2} < \infty.
\]
For $\beta_2$, we have
\[
\beta_2 \lesssim (t^d \log(t))^{-1} \left( \int_{tH}dx \int_0^1 ds \left( \int_{tH}dy \int_s^1 du\ u^{-1} \exp(-cu|x-y|^d) \right)^2 \right)^{1/2}.
\]
This term can be dealt with in the same way as with $\beta_1$ and it is $O(\log(t)^{-1})$.

As for $\beta_3$, we get
\[
\beta_3 \lesssim (t^d \log(t))^{-3/2} \int_{tH} dy \int_0^1 du\ (u^{-1/2} \wedge t^{d/2})^3.
\]
Integrating over $u$ and $y$ gives that this is $O(\log(t)^{-3/2})$.

The last term is given by
\[
\beta_4 \lesssim (t^d \log(t))^{-3/2} \int_{tH} dx \int_0^1 ds \int_{tH} dy \int_0^s du\ u^{-1/2} s^{-1} \exp(-cs|x-y|^d).
\]
Proceeding in the same way we dealt with $\beta_1$ yields that $\beta_4=O(\log(t)^{-3/2})$.

Combining our estimates above, we conclude that $\beta_1 + \beta_2 + \beta_3 + \beta_4 = O(\log(t)^{-1})$. \qed}

\subsubsection{Wasserstein and Kolmogorov distance when $\alpha<\frac{d}{2}$}
{\setlength\parindent{0pt}
We use Theorem~\ref{thmWasserNon-Cond.} with $p,q \in (1,2]$ and $\epsilon>0$ such that $2p(\alpha + \epsilon d)<d$ and $(q+1)(\alpha + \epsilon d) < d$. In the following we will show that all terms $\gamma_1,...,\gamma_7$ are $O(t^{d(1/p-1)})$. By Propositions~\ref{propONNGlowerVar1}, \ref{propONNGBounds} and Cauchy-Schwarz inequality, we have for all $r\geq 1$ the following bounds:
\begin{align}
	&\E [|\MalD_{(x,s)} \fat|^r]^{1/r} \lesssim s^{-\alpha/d-\epsilon}\\
	&\E[|\DD_{(x,s),(y,u)}\fat|^r]^{1/r} \lesssim (u \vee s)^{-\alpha/d-\epsilon} \exp\left(-c(u \vee s)|x-y|^d\right)\\
	&\var(\fat) \gtrsim t^d.
\end{align}
Introducing these bounds into $\gamma_1$, we get
\begin{equation}\label{Rainbow25}
\gamma_1 \lesssim t^{-d} \left( \int_{tH} dx \int_0^1 ds \left( \int_{tH} dy \int_0^1 du\ u^{-\alpha/d-\epsilon} (u \vee s)^{-\alpha/d-\epsilon} \exp\left(-c(u \vee s)|x-y|^d\right) \right)^p \right)^{1/p}.
\end{equation}
Using a change of variables, we can bound the integral over $y$ as follows:
\begin{equation}\label{Rainbow26}
\int_{tH} dy\ \exp(-c(u \vee s)|x-y|^d) \leq \int_{\R^d} dz\ \exp(-c(u \vee s)|z|^d) \lesssim (u \vee s)^{-1}.
\end{equation}
Introducing \eqref{Rainbow26} into \eqref{Rainbow25} and integrating over $x$ gives us
\[
\gamma_1 \lesssim t^{d(1/p-1)} \left( \int_0^1 ds \left( \int_0^1 du\ u^{-\alpha/d-\epsilon} (u \vee s)^{-\alpha/d-\epsilon-1} \right)^p \right)^{1/p}.
\]
Integrating over $u$, one sees that these integrals are bounded by
\[
c\int_0^1 ds\ s^{-2p(\alpha/d+\epsilon)} <\infty.
\]
Hence $\gamma_1 = O(t^{d(1/p-1)})$.
The term $\gamma_2$ is bounded by
\[
\gamma_2 \lesssim t^{-d} \left( \int_{tH} dx \int_0^1 ds \left( \int_{tH} dy \int_0^1 du\ (u \vee s)^{-2\alpha/d-2\epsilon} \exp\left(-c(u \vee s)|x-y|^d\right) \right)^p \right)^{1/p}.
\]
This can be treated analogously to the bound on $\gamma_1$ and yields $\gamma_2=O(t^{d(1/p-1)})$.

As for $\gamma_3$, it is bounded by
\[
\gamma_3 \lesssim t^{-d(q+1)/2} \int_{tH}dy \int_0^1 du\ u^{-(q+1)(\alpha/d+\epsilon)}.
\]
This is $O(t^{d(1-q)/2})$. Choose $q = 3-\frac{2}{p}$, then $(q+1)(\alpha + \epsilon d) \leq 2p (\alpha + \epsilon d) <d$ and $q \in (1,2]$, hence the conditions are satisfied. Moreover, $\frac{1-q}{2} = -1+\frac{1}{p}$, thus we find the same rate of convergence.

The term $\gamma_4$ is bounded by
\[
\gamma_4 \lesssim t^{-d} \left( \int_{tH} dx \int_0^1 ds\ s^{-2p(\alpha/d-\epsilon)} \right)^{1/p},
\]
which is clearly $O(t^{d(1/p-1)})$.
For the term $\gamma_5$, we deduce
\[
\gamma_5 \lesssim t^{-d} \left( \int_{tH}dx \int_0^1 ds \int_{tH}dy \int_0^1du \ (u \vee s)^{-2p(\alpha/d+\epsilon)} \exp(-c(u \vee s)|x-y|^d) \right)^{1/p}.
\]
Integrating over $x$ and $y$ as before, we infer
\[
\gamma_5 \lesssim t^{d(1/p-1)} \left(\int_0^1 ds \int_0^1 du\ (u \vee s)^{-2p(\alpha/d+\epsilon)-1} \right)^{1/p},
\]
which is $O(t^{d(1/p-1)})$.
The terms $\gamma_6$ and $\gamma_7$ work similarly:
\[
\gamma_6 \lesssim t^{-d} \left( \int_{tH}dx \int_0^1 ds \int_{tH}dy \int_0^1du \ (u \vee s)^{-p(\alpha/d+\epsilon)} s^{-p(\alpha/d+\epsilon)} \exp(-c(u \vee s)|x-y|^d) \right)^{1/p},
\]
and
\begin{multline}
\gamma_7 \lesssim t^{-d} \bigg( \int_{tH}dx \int_0^1 ds \int_{tH}dy \int_0^1du \ (u \vee s)^{-(\alpha/d+\epsilon)} s^{-(\alpha/d+\epsilon)} \\ u^{-2(\alpha/d+\epsilon)(p-1)} \exp(-c(u \vee s)|x-y|^d) \bigg)^{1/p}
\end{multline}
which can be shown to be $O(t^{d(1/p-1)})$ by the same method. This concludes the proof of Theorem~\ref{thmONNG}. \qed}

\section{Gilbert Graph}\label{secPGil}
Throughout this section, we work in the framework of Section~\ref{secGil} and Theorem~\ref{thmGilbert}. We start with a technical lemma.
\begin{lem}\label{lemPoissonMoment}
	Let $Z$ be a Poisson random variable with intensity $\lambda>0$ and let $r \geq 1$. Then there is a constant $c_r>0$ such that
	\[
	\E [Z^r]^{1/r} \leq c_r (\lambda \vee \lambda^{1/r}).
	\]
\end{lem}
\begin{proof}
	Let $m \in \N$. As can be found in any standard reference (see e.g. \cite[Proposition~3.3.2]{PT11}),
	\[
	\E Z^m = \sum_{i=1}^m {m \brace i} \lambda^i,
	\]
	where ${m \brace i}$ are the Stirling numbers of second kind. This is bounded by $(\lambda \vee \lambda^m) B_m$, where $B_m = \sum_{i=1}^m {m \brace i}$ is the $m$th Bell number. Hence
	\[
	\E [Z^m]^{1/m} \leq B_m^{1/m} (\lambda \vee \lambda^{1/m}).
	\]
	Now let $r>1$, $r \notin \N$ and define $p_0 := \floor{r}$ and $p_1:=\ceil{r}$. Then take $\theta:= \frac{p_1}{r} \frac{r-p_0}{p_1-p_0} \in (0,1)$ such that we have $\frac{1}{r} = \frac{1-\theta}{p_0} + \frac{\theta}{p_1}$. By log-convexity of $L^p$ norms (see e.g.  \cite{TaoInterpolation} or \cite[Remark~2. p.93]{Bre11}) and the first part of the proof, we have
	\[
	\E [Z^r]^{1/r} \leq \norm{Z}_{p_0}^{1-\theta} \norm{Z}_{p_1}^\theta \leq B_{p_0}^{(1-\theta)/p_0} B_{p_1}^{\theta/p_1} (\lambda \vee \lambda^{1/r}),
	\]
	which provides the desired bound with $c_r := B_{p_0}^{(1-\theta)/p_0} B_{p_1}^{\theta/p_1}$.
\end{proof}
As a next step, we prove bounds on the first and second order add-one costs of $\lat$.
\begin{prop}\label{propGilBounds}
	Let $\alpha>-\frac{d}{2}$ and $r\geq 1$ such that $d+r\alpha>0$. Then $\lat \in \dom \MalD$ and there is a constant $c>0$ such that for all $x,y \in W$ and $t>0$
	\begin{align}
		&\E \left[\left(\MalD_x \lat\right)^{r}\right]^{1/r} \leq c \epsilon_t^\alpha (t\epsilon_t^d)^{1/r} \left( 1 \vee t\epsilon_t^d \right)^{1-1/r} \label{iPGilDx} \\
		&\p(\MalD_x \lat \neq 0) \leq 1 \wedge \kappa_d t \epsilon_t^d \label{iPGilPDx}\\
		&\MalD_{x,y}^{(2)} \lat = \ind{|x-y| < \epsilon_t } |x-y|^\alpha.\label{iPGilDxy}
	\end{align}
\end{prop}
\begin{proof}
	As explained in Section~\ref{secFrame}, to show that $\lat \in \dom \MalD$ it suffices to argue that $\lat \in L^1(\p_{\eta^t})$ and that \eqref{eqDomDCond} holds. By \cite[Theorem~3.1]{TSRGilbert}, it is true that $\lat \in L^1(\p_{\eta_t})$ and the fact that $\MalD_{x} \lat$ is square-integrable follows from \eqref{iPGilDx} with $r=2$ and from the fact that $W$ is bounded. Hence $\lat \in \dom \MalD$ follows once we have shown \eqref{iPGilDx}. 
	
	Since $\MalD_x \lat = \lat(\eta^t+\delta_x)-\lat(\eta^t)$, it is easy to see that
	\begin{equation}\label{iPGil1}
	\MalD_x \lat = \sum_{y \in \eta^t} \ind{|x-y| < \epsilon_t} |x-y|^\alpha.
	\end{equation}
	It now follows that
	\[
	\MalD_{x,y}^{(2)} \lat = \MalD_x \lat(\eta^t+\delta_y) - \MalD_x \lat(\eta^t) = \ind{|x-y| < \epsilon_t } |x-y|^\alpha,
	\]
	which gives \eqref{iPGilDxy}. Since all terms in the sum \eqref{iPGil1} are non-negative, we have $\MalD_x \lat \neq 0$ if and only if $\eta^t(W \cap B^d(x,\epsilon_t)) \neq 0$. Therefore
	\begin{align}
		\p(\MalD_x \lat \neq 0) = \p(\eta^t(W \cap B^d(x,\epsilon_t)) \neq 0) = 1-\exp(-t|W \cap B^d(x,\epsilon_t)|)
	\end{align}
	which is bounded by $1 \wedge t|W \cap B^d(x,\epsilon_t)| \leq 1 \wedge \kappa_d t\epsilon_t^d$. This gives \eqref{iPGilPDx}.
	
	To prove \eqref{iPGilDx}, note first that $\MalD_x \lat$ can be written
	\[
	\MalD_x \lat = \int_{W \cap B^d(x,\epsilon_t)} |x-y|^\alpha \eta^t(dy).
	\]
	This is stochastically dominated by
	\[
	\int_{B^d(x,\epsilon_t)} |x-y|^\alpha \hat{\eta}^t(dy),
	\]
	where $\hat{\eta}^t$ is an $(\R^d,t\, dx)$-Poisson measure. By translation invariance of the law of $\hat{\eta}^t$, this is equal in law to
	\[
	\int_{B^d(0,\epsilon_t)} |y|^\alpha \hat{\eta}^t(dy),
	\]
	which in turn is equal in law to
	\[
	\sum_{i=1}^{M_t}|U_i|^\alpha,
	\]
	where $M_t$ is a Poisson random variable of intensity $\kappa_d t \epsilon_t^d$ and $U_1,U_2,...$ are i.i.d. uniform random variables in $B^d(0,\epsilon_t)$ independent of $M_t$.
	
	By Jensen's inequality,
	\[
	\E \left( \sum_{i=1}^{M_t}|U_i|^\alpha \right)^{r} \leq \E M_t^{r-1} \sum_{i=1}^{M_t}|U_i|^{r\alpha}.
	\]
	By independence, this is equal to
	\[
	\E\left[ M_t^{r-1} \sum_{i=1}^{M_t} \E|U_i|^{r\alpha}\right] = \frac{d}{d+r\alpha} \epsilon_t^{r\alpha} \E M_t^{r}.
	\]
	Using Lemma~\ref{lemPoissonMoment}, we deduce
	\[
	\E \left[M_t^{r}\right]^{1/r} \leq c_{r} \left(\kappa_dt\epsilon_t^d \vee (\kappa_dt\epsilon_t^d)^{1/r}\right).
	\]
	Combining the above bounds leads to
	\[
	\E \left[\left(\MalD_x \lat\right)^{r}\right]^{1/r} \leq c_{r} (1 \vee \kappa_d) \left(\frac{d}{d+r\alpha}\right)^{1/r} \epsilon_t^\alpha \left(t\epsilon_t^d \vee (t\epsilon_t^d)^{1/r}\right),
	\]
	which shows \eqref{iPGilDx}.
\end{proof}
Theorem~3.3 in \cite{TSRGilbert} gives us the following variance asymptotics: for $\alpha > -\frac{d}{2}$,
\[
\var \left(\lat\right) = \left( \sigma_\alpha^{(1)} t^2 \epsilon_t^{2\alpha+d} + \sigma_\alpha^{(2)} t^3 \epsilon_t^{2\alpha+2d}\right) |W| (1+O(\epsilon_t)),
\]
where $\sigma_\alpha^{(1)} = \frac{d\kappa_d}{2(d+2\alpha)}$ and $\sigma_\alpha^{(2)} = \frac{d^2\kappa_d^2}{(\alpha+d)^2}$. Hence for large enough $t$, there is a constant $c>0$ such that
\begin{equation}\label{eqGilVar}
	\var \left(\lat\right) \geq c t^2 \epsilon_t^{2\alpha+d}\left( 1 \vee t \epsilon_t^d\right).
\end{equation}
We are now in a position to prove Theorem~\ref{thmGilbert}.

\begin{proof}[Proof of Theorem~\ref{thmGilbert}]
	We plug the bounds from Proposition~\ref{propGilBounds} and \eqref{eqGilVar} into the terms $\gamma_1,...,\gamma_7$ given in Theorem~\ref{thmWasserNon-Cond.}. Let $p \in (1,2]$ such that $2p\alpha +d>0$.
	
	For the choice of $r=2p$ in \eqref{iPGilDx}, the first term yields
	\[
	\gamma_1 \lesssim \left(t^2 \epsilon_t^{2\alpha+d}\left( 1 \vee t \epsilon_t^d\right)\right)^{-1} \left( \int_W \left( \int_W \epsilon_t^\alpha (t\epsilon_t^d)^{1/(2p)} \left( 1 \vee t\epsilon_t^d \right)^{1-1/(2p)} \ind{|x-y|<\epsilon_t} |x-y|^\alpha tdy \right)^p tdx \right)^{1/p}.
	\]
	Note that
	\begin{equation}\label{Rainbow27}
	\int_W \ind{|x-y|<\epsilon_t} |x-y|^\alpha dy \leq \int_{B^d(x,\epsilon_t)}  |x-y|^\alpha dy = \frac{d\kappa_d}{d+\alpha} \epsilon_t^{d+\alpha}.
	\end{equation}
	Hence $\gamma_1$ is (up to multiplication by a positive constant) bounded by
	\[
	t^{1/p-1} (t\epsilon_t^d)^{1/(2p)} (1 \vee t\epsilon_t^d)^{-1/(2p)} = t^{1/p-1} \left(1 \wedge (t\epsilon_t^d)^{1/(2p)}\right).
	\]
	As for $\gamma_2$, we have
	\[
	\gamma_2 \lesssim \left(t^2 \epsilon_t^{2\alpha+d}\left( 1 \vee t \epsilon_t^d\right)\right)^{-1} \left( \int_W \left( \int_W \ind{|x-y|<\epsilon_t} |x-y|^{2\alpha} tdy \right)^p tdx \right)^{1/p}.
	\]
	By \eqref{Rainbow27}, we deduce
	\[
	\gamma_2 \lesssim t^{1/p-1} \left( 1 \vee t \epsilon_t^d\right)^{-1}.
	\]
	For $\gamma_3$, take $q=3-\frac{2}{p}$. Then $q \in (1,2]$ and $(q+1)\alpha + d>2p\alpha+d>0$. Let $r>\frac{q+1}{2}$ such that $2r\alpha+d>0$. Then, using
	\[
	\E \left(\MalD_{x}\lat\right)^{q+1}  \leq \E \left[\left(\MalD_{x}\lat\right)^{2r}\right]^{\frac{q+1}{2r}} \p(\MalD_{x}\lat \neq 0)^{1-\frac{q+1}{2r}}
	\]
	we deduce
	\[
	\gamma_3 \lesssim \left(t^2 \epsilon_t^{2\alpha+d}\left( 1 \vee t \epsilon_t^d\right)\right)^{1/p-2} \int_W \left(\epsilon_t^\alpha (t\epsilon_t^d)^{1/(2r)} \left( 1 \vee t\epsilon_t^d \right)^{1-1/(2r)}\right)^{4-2/p} \left(1 \wedge t\epsilon_t^d\right)^{1-\frac{4-2/p}{2r}} tdx.
	\]
	Simplifying, we infer
	\[
	\gamma_3 \lesssim t^{1/p-1} (1 \vee (t\epsilon_t^d)^{1/p-1}).
	\]
	With the same method, one can establish the upper bounds
	\begin{align}
	&\gamma_4 \lesssim t^{1/p-1} \left(1 \vee (t\epsilon_t^d)^{1/p-1}\right)\\
	&\gamma_5 \lesssim t^{1/p-1} (t\epsilon_t^d)^{1/p-1} \left(1 \vee t\epsilon_t^d\right)^{-1} \\
	&\gamma_6 \lesssim t^{1/p-1} (t\epsilon_t^d)^{1/p-1} \left(1 \wedge (t\epsilon_t^d)^{1/(2p)}\right)\\
	&\gamma_7 \lesssim t^{1/p-1} (t\epsilon_t^d)^{2/p-1-1/(2p^2)} \left(1 \vee (t\epsilon_t^d)\right)^{-2/p+1+1/(2p^2)}
	\end{align}
	All of these bounds are upper bounded by $t^{1/p-1} \left(1 \vee (t\epsilon_t^d)^{1/p-1}\right)$. If $\alpha>-\frac{d}{4}$, we can choose $p=2$ and recover \eqref{eqGil1}. To show \eqref{eqGil2}, one chooses $1<p<-\frac{d}{2\alpha}$, thus concluding the proof.
\end{proof}

\section{$k$-Nearest Neighbour Graphs}\label{secPkNN}
In this section, we work in the setting of Theorem~\ref{thmkNN}. We start with a technical lemma.
\begin{lem}\label{lemPhiProp}
	Let $\phi:(0,\infty) \rightarrow (0,\infty)$ be a non-increasing function satisfying \eqref{eqkNNC1} for some $r>2$. Then for any $0 \leq q \leq r$, the following two integrals are finite:
	\begin{equation} \label{eqPhiInt}
		\int_0^1 \phi(s)^q s^{d-1} ds < \infty
	\end{equation}
	and for any constant $c>0$,
	\begin{equation}\label{eqPhiexp}
		\int_{\R^d} \phi(|x|)^q \exp(-c|x|^d) dx < \infty.
	\end{equation}
	Moreover, for $x \in \R^d$ and $\mu \subset \R^d$ a finite generic set with respect to $x$, define the following quantity:
	\[
	e(x,\mu) :=%
	\begin{cases*}
		0 & if $\mu = \emptyset$ \\
		\min\{|x-z|:z \in \mu, z\neq x\} & if $\mu \neq \emptyset$,
	\end{cases*}
	\]
	(that is, $e(x,\mu)$ is the length from $x$ to the point of $\mu$ nearest to it, or zero if $\mu$ is empty). Extend the definition of $\phi$ to $[0,\infty)$ by setting $\phi(0)=0$. Then for any $0 < q \leq r$, there is a constant $c>0$ such that for all $x \in tH$, $t\geq 1$,
	\begin{equation}\label{eqPhiE}
		\E \left[\phi(e(x,\eta_{|tH}))^q\right]^{1/q} < c.
	\end{equation}
\end{lem}
\begin{proof}
	We start by noting that $d s^{d-1}\ind{0<s<1}$ is a probability density. Hence, by Jensen's inequality,
	\[
	\left(d \int_0^1 \phi(s)^q s^{d-1} ds\right)^{1/q} \leq \left(d \int_0^1 \phi(s)^r s^{d-1} ds\right)^{1/r}.
	\]
	This is finite by virtue of condition \eqref{eqkNNC1}, thus yielding \eqref{eqPhiInt}. Using polar coordinates, the integral in \eqref{eqPhiexp} is equal to
	\begin{equation}\label{Rainbow7}
	d \kappa_d \int_0^\infty \phi(s)^q \exp(-cs^d) s^{d-1} ds.
	\end{equation}
	Now use that $\phi$ is non-increasing to infer that, if $s\geq 1$, then $\phi(s) \leq \phi(1)$. Hence \eqref{Rainbow7} is bounded by
	\[
	d \kappa_d \left(\int_0^1 \phi(s)^q s^{d-1} ds + \int_1^\infty \phi(1)^q s^{d-1} \exp(-cs^d) ds\right).
	\]
	The second integral is clearly finite and the first is finite by \eqref{eqPhiInt}, thus implying \eqref{eqPhiexp}. For the bound \eqref{eqPhiE}, note that by Jensen's inequality
	\[
	\E \left[\phi(e(x,\eta_{|tH}))^q\right]^{1/q} \leq \E \left[\phi(e(x,\eta_{|tH}))^r\right]^{1/r},
	\]
	therefore it suffices to show the bound for $q=r$.
	
	Take $x \in tH$ and let us study the distribution of $e(x,\eta_{|tH})$. For any $a \geq 0$, we have
	\[
	G(a) := \p(e(x,\eta_{|tH}) \leq a) = \p(e(x,\eta_{|tH}) = 0) + \p(0<e(x,\eta_{|tH})\leq a).
	\]
	The event $e(x,\eta_{|tH})=0$ happens if and only if $\eta(tH)=0$, whereas $0<e(x,\eta_{|tH})\leq a$ is equivalent to the event that $\eta(tH \cap \bb^d(x,a)) \neq 0$. Hence
	\[
	G(a) = \p(\eta(tH)=0) + \p(\eta(tH \cap \bb^d(x,a)) \neq 0) = \exp(-|tH|) + 1-\exp(-|tH \cap B^d(x,a)|).
	\]
	We can compute the derivative of $1-\exp(-|tH \cap B^d(x,a)|)$ as follows:
	\[
	\frac{d}{da} (1-\exp(-|tH \cap B^d(x,a)|)) = \exp(-|tH \cap B^d(x,a)|) \frac{d}{da} |tH \cap B^d(x,a)|.
	\]
	Note that, by changing variables and moving to polar coordinates, we can rewrite the volume as
	\begin{align}
		|tH \cap B^d(x,a)| &= \int_{\R^d} \ind{y \in tH} \ind{|y-x| < a} dy \notag\\
		&= \int_{\R^d} \ind{x+z \in tH} \ind{|z| < a} dz \notag\\
		&= \int_0^a du \int_{S^{d-1}} d\omega\, u^{d-1} \ind{x+u\omega \in tH},
	\end{align}
	where $S^{d-1}$ is the unit sphere in $\R^d$. Define thus
	\[
	g(a) := \exp(-|tH \cap B^d(x,a)|) \int_{S^{d-1}} a^{d-1} \ind{x+a\omega \in tH} d\omega,
	\]
	then
	\[
	G(a) = \p(e(x,\eta_{|tH}) = 0) + \int_0^a g(u) du.
	\]
	Going back to the bound we want to prove, we now have
	\begin{align}
		\E \left[ \phi(e(x,\eta_{|tH}))^r \right]^{1/r} &= \left(\phi(0) \p(e(x,\eta_{|tH})=0) + \int_0^{\infty} \phi(u)^r g(u) du\right)^{1/r}\notag\\
		&=  \left(\int_0^{\infty} \phi(u)^r g(u) du\right)^{1/r}.
	\end{align}
	We need to show that this integral is bounded by a constant independent of $t$. Since $H$ has non-empty interior, there is a ball $B^d(x_0,\delta) \subset H$, with $x_0 \in H$ and $\delta>0$. By Lemma~\ref{lemGeo1} and rescaling, there is a constant $c_H>0$ depending on $H$ and $d$ such that $|tH \cap B^d(x,a)| \geq c_H a^d$. Thus, upper bounding the indicator by $1$, we find
	\[
	g(a) \leq d \kappa_d a^{d-1} \exp(-c_H a^d).
	\]
	As a consequence
	\[
	\left(\int_0^{\infty} \phi(u)^r g(u) du\right)^{1/r} \leq \left(d \kappa_d \int_0^\infty \phi(u)^r u^{d-1} \exp\left(-c_Hu^d\right) du \right)^{1/r},
	\]
	and the RHS of this inequality is finite by \eqref{eqPhiexp}.
\end{proof}
\begin{prop}\label{propkNNBounds}
	Under the conditions of Theorem\ref{thmkNN}, there are absolute constants $C_2,c_2>0$ such that for any $t\geq1$, any $x,y \in tH$, the following bound holds:
	\begin{equation}
		\p\left(\DD_{x,y}F_t \neq 0\right) < C_2 \exp\left(-c_2|x-y|^d\right).\label{eqkNNP}
	\end{equation}
	For any $1 \leq p \leq \frac{r}{2}$ there is a constant $c_1>0$ such that for any $t\geq1$, any $x,y \in tH$, the following two bounds hold:
	\begin{align}
		&\E \left[|\MalD_x F_t|^{2p}\right]^{1/(2p)} < c_1 \label{eqkNNDx}\\
		&\E\left[|\DD_{x,y} F_t|^{2p}\right]^{1/(2p)} < c_1 \left(\phi(|x-y|)+1\right)\label{eqkNNDxy}
	\end{align}
	Moreover, $F_t \in \dom \MalD$.
\end{prop}
\noindent The proof of \eqref{eqkNNP} reuses arguments from \cite[Theorem~7.1]{LPS14}.
\begin{proof}
	\textbf{Step 1.} We start by showing that for any $x \in \R^d$ and any finite set $\mu \in \R^d$ generic with respect to $x$, we have $\MalD_xF(\mu) \geq 0$.
	
	First, note that
	\[
	F(\mu) = \sum_{e \in k\text{-NN}(\mu)} \phi(|e|),
	\] 
	where we take the sum over all edges $e$ in the $k$-NN graph. Define a `direction' on the graph in the following way:
	\begin{itemize}
		\item If a point $w$ is a nearest neighbour of the point $z$, direct the edge from $z$ to $w$ and write $z \rightarrow w$;
		\item if $z$ and $w$ are reciprocal nearest neighbours, direct the edge both ways and write $z \leftrightarrow w$.
	\end{itemize}
	Upon addition of the point $x$, any of the following scenarios can happen:
	\begin{enumerate}
		\item An edge $z \rightarrow w$ is replaced by $z \rightarrow x$ (or $z \leftrightarrow x$), in which case $|z-w|>|z-x|$;
		\item an edge $z \leftrightarrow w$ is replaced by $z \rightarrow x$ (or $z \leftrightarrow x$) and $w \rightarrow x$ (or $w \leftrightarrow x$), in which case $|z-w| > |z-x| \vee |w-x|$;
		\item an edge $z \leftrightarrow w$ becomes $w \rightarrow z$ and the edge $z \rightarrow x$ (or $z \leftrightarrow x$) is added;
		\item edges $x \rightarrow z$ are added.
	\end{enumerate}
	Let $E_a$ and $E_r$ be the sets of added and removed edges respectively. Every removed edge is replaced by at least one added edge with shorter length. All other added edges increase $\MalD_xF(\mu)$. Since $\phi$ is decreasing, we have thus
	\[
	\MalD_xF(\mu) = \sum_{e \in E_a} \phi(|e|) - \sum_{e \in E_r} \phi(|e|) \geq 0.
	\]
	
	\textbf{Step 2.} We now prove \eqref{eqkNNDx}. Fix $x \in \R^d$ and $\mu \subset \R^d$, a finite set generic with respect to $x$, and consider the $k$-NN built on $\mu$. Define $e(x,\mu)$ as in Lemma~\ref{lemPhiProp}.
	
	Suppose that $\mu \neq \emptyset$ and consider $\MalD_xF(\mu)$. There is a constant $n_{d,k}$ such that for any $k$-NN graph in $\R^d$, no vertex has degree more than $n_{d,k}$. This fact was used in the proof of  \cite[Lemma~7.2]{LPS14} and we refer to the references given therein for more details. When adding $x$ to the graph, all added edges are incident to $x$, and hence we have $|E_a| \leq n_{d,k}$. Every added edge between points $x$ and $y$ must verify $|x-y| \geq e(x,\mu)$, since $e(x,\mu)$ is the minimally possible edge-length for any edge incident to $x$. Since $\phi$ is decreasing, we conclude
	\[
	|\MalD_xF(\mu)| = \MalD_xF(\mu) \leq |E_a| \phi(e(x,\mu)) \leq n_{d,k}\phi(e(x,\mu)).
	\]
	Note that this is well-defined since $e(x,\mu)>0$ if $\mu \neq \emptyset$.
	
	If $\mu = \emptyset$, then $\MalD_xF(\mu) = 0$. Extend the definition of $\phi$ to $[0,\infty)$ by setting $\phi(0)=0$. Then in all cases
	\[
	|\MalD_xF(\mu)| \leq n_{d,k}\phi(e(x,\mu)).
	\]
	We have thus the following bound:	
	\[
	\E \left[|\MalD_xF_t|^{2p}\right]^{1/(2p)} \leq n_{d,k} \E \left[ \phi(e(x,\eta_{|tH}))^{2p}\right]^{1/(2p)},
	\]
	which is bounded by a constant by Lemma~\ref{lemPhiProp}.
	
	\textbf{Step 3.} We now show \eqref{eqkNNDxy}. By Step 2, we have
	\begin{align}
		|\DD_{x,y} F_t| &\leq |\MalD_{x}F(\eta_{|tH} + \delta_x)| + |\MalD_{x}F(\eta_{|tH})| \notag\\
		&\leq n_{d,k} \phi(e(x,\eta_{|tH})) + n_{d,k} \phi(e(x,\eta_{|tH}+\delta_y)).
	\end{align}
	For ease of notation, define the event $A_t := \{\eta(tH) \neq 0\}$. If $\eta(tH)=0$, then $e(x,\eta_{|tH}+\delta_y) = |x-y|$, else $e(x,\eta_{|tH}+\delta_y) = e(x,\eta_{|tH}) \wedge |x-y|$. Hence if we split over both events, we get
	\begin{align}
		\phi(e(x,\eta_{|tH}+\delta_y)) &= \1_{A_t^c} \phi(|x-y|) + \1_{A_t} \phi(e(x,\eta_{|tH}) \wedge |x-y|) \notag\\
		&\leq \1_{A_t^c} \phi(|x-y|) + \1_{A_t} \phi(|x-y|) + \1_{A_t} \phi(e(x,\eta_{|tH})) \notag\\
		&= \phi(|x-y|) + \1_{A_t} \phi(e(x,\eta_{|tH})).
	\end{align}
	Hence we have
	\[
	|\DD_{x,y} F_t| \leq 2 n_{d,k} \phi(e(x,\eta_{|tH})) + n_{d,k} \phi(|x-y|).
	\]
	Lemma~\ref{lemPhiProp} now yields the result.
	
	\textbf{Step 4.} The inequality \eqref{eqkNNP} follows immediately from the argument used in the proof of \cite[Theorem~7.1]{LPS14}, which relies solely on the structure of the graph, and not on the function applied to the edge-lengths.
	
	\textbf{Step 5.} Lastly, we show that $F_t \in \dom \MalD$. As was explained in Section~\ref{secFrame}, it suffices to show that $F_t$ is integrable and $\MalD_{x} F_t$ square integrable. The second fact immediately follows from ~\eqref{eqkNNC1}. By Mecke equation we also have
	\[
	\E F_t = \int_{tH} \int_{tH} \phi(|x-y|) \p(x \in N(y,\eta_{|tH}) \text{ or } y \in N(x,\eta_{|tH})) dxdy,
	\]
	where we recall that $N(y,\eta_{|tH})$ is the set of the $k$ nearest neighbours of $y$ in $\eta_{|tH}$.
	Upper bounding the probability in the integrand by $1$ and changing variables, this is upper bounded by
	\[
	t^d|H| \int_{B^d(0, t\diam(H))} \phi(|z|) dz.
	\]
	Changing to polar coordinates, this is equal to
	\[
	t^d|H| d \kappa_d \int_0^{t \diam(H)} s^{d-1} \phi(s) ds,
	\]
	which is finite by Lemma~\ref{lemPhiProp}.
\end{proof}

\begin{prop}\label{propkNNVar}
	Under the conditions of Theorem~\ref{thmkNN}, there are constants $c_1,c_2>0$ such that
	\[
	c_1 t^d \leq \var(F_t) \leq c_2 t^d.
	\]
\end{prop}

\begin{proof}
	The upper bound immediately follows from Poincaré inequality (\eqref{eqPPoinNoTime} with $p=2$) and \eqref{eqkNNDx} with $p=1$. For the lower bound, we use \cite[Theorem~5.2]{LPS14} and a reasoning similar to what was done in the proof of \cite[Lemma~7.2]{LPS14}.
	
	Assume there are points $w_1,...,w_m \in \R^d$ with $\frac{1}{2} < |w_i|<1$ for $i \in \{1,...,m\}$ such that for all $y \in \R^d$ with $|y|\geq \frac{1}{2}$,
	\begin{equation}\label{iP10Points}
		|\{i \in \{1,...,m\}: |w_i-y| < |y|\}| \geq k+1.
	\end{equation}

	This means that for any point $y$ outside $B^d\left(0,\frac{1}{2}\right)$, there are at least $k+1$ points among the $\{w_1,...,w_m\}$ which are closer to $y$ that the origin.
	
	Now assume that there is $\tau>0$ and $x\in tH$ such that $B^d(x,\tau) \in tH$. Define $\tilde{w}_i:= x + \tau w_i$. Consider the collection of points
	\[
	\mathcal{U}:= \eta_{|tH} + \sum_{i=1}^m \delta_{\tilde{w}_i}
	\]
	and let us evaluate $\MalD_{x}F(\mathcal{U})$.
	
	First, note that by \eqref{iP10Points}, all points outside $B^d(x,\tau/2)$ have at least $k$ points closer than $x$ among the $\tilde{w}_1,...,\tilde{w}_m$, including these points themselves. Therefore any points in $\mathcal{U}$ connecting to $x$ must be within $B^d(x,\tau/2)$.
	
	Assume $\eta(B^d(x,\tau))=0$. Then the only edges added when adding $x$ are those from $x$ to its $k$ nearest neighbours among $\tilde{w}_1,...,\tilde{w}_m$. Since $|x-\tilde{w}_i| \leq \tau$ for any $i$, we have:
	\[
	\ind{\eta(B^d(x,\tau))=0} \MalD_xF(\mathcal{U}) \geq \ind{\eta(B^d(x,\tau))=0} k \phi(\tau).
	\]
	As shown in Step 1 of the proof of Proposition~\ref{propkNNBounds}, we have $\MalD_xF(\mathcal{U}) \geq 0$ and
	\begin{equation}\label{iP10Dxbound}
	\left|\E \MalD_xF(\mathcal{U}) \right| \geq \E \left[\ind{\eta(B^d(x,\tau))=0} k \phi(\tau)\right] = k \phi(\tau) \exp(-\kappa_d \tau^d).
	\end{equation}
	
	We now find a set of $(x,\tilde{w}_1,...,\tilde{w}_m)$ for which this bound is true.
	
	Let $\tau>0$ be such that there is a ball $B^d(x_0,2\tau) \subset H$. For any $x \in B^d(tx_0, t\tau)$ we have $B^d(x,\tau) \subset tH$.
	
	The closure of the set $B^d(0,1)\setminus B^d\left(0,\frac{1}{2}\right)$ is compact and can be covered by balls of radius $\frac{1}{4}$. Setting $k+1$ points into the interior of each intersection of one such ball with the annulus $B^d(0,1)\setminus B^d\left(0,\frac{1}{2}\right)$ gives a collection of points $\{z_1,...,z_m\}$. We claim that this collection satisfies \eqref{iP10Points}. Indeed, any point inside $B^d(0,1)\setminus B^d\left(0,\frac{1}{2}\right)$ will be in one ball of the covering and thus there are at least $k+1$ points from $\{z_1,...,z_m\}$ at a distance of at most $\frac{1}{2}$. For any point $y$ outside $B^d(0,1)$, there is a point $z \in \partial B^d(0,1)$ such that $|y-z| = |y|-1$ and this point has $k+1$ points $z_i$ among $\{z_1,...,z_m\}$ that are less than $\frac{1}{2}$ away. For any such $z_i$, one has $|y-z_i| \leq |y|-1+\frac{1}{2}$.
	
	Given a choice of $z_1,...,z_m$, property \eqref{iP10Points} still holds if we slightly perturb the $z_i$: there is an $\epsilon > 0$ such that the collection of points $\{z_1+y_1,...,z_m+y_m\}$ satisfies \eqref{iP10Points} for any $y_1,...,y_m \in B^d(0,\epsilon)$.
	
	The bound \eqref{iP10Dxbound} is true for any $(x,\tilde{w}_1,...,\tilde{w}_m) \in U$, where
	\[
	U := \left\{ (z,z+\tau(z_1+y_1),...,z+\tau(z_m+y_m)): z \in B^d(tx_0,t\tau),y_1,...,y_m \in B^d(0,\epsilon) \right\}.
	\]
	By \cite[Theorem~5.2]{LPS14} and a development analogous to the one in the proof of \cite[Theorem~5.3]{LPS14}, we find that for some constant $c>0$,
	\[
	\var(F_t) \geq c |U| = c \kappa_d \tau^d (\kappa_d \epsilon^d)^m t^d,
	\]
	which yields the desired bound.
\end{proof}

\begin{proof}[Proof of Theorem~\ref{thmkNN}]
	For the rest of the proof, all constants will be referred to as $c$, to simplify notation. Take $p \in (1,2]$ such that $p < \frac{r}{2}$, where $r$ is given by the condition \eqref{eqkNNC1}.
	
	Let us start with a bound on $\gamma_1$. We use that
	\[
	\E \left[|\DD_{x,y} F_t |^{2p}\right]^{1/(2p)} \leq \p\left(\DD_{x,y} F_t \neq 0 \right)^{1/(2p)-1/r} \E \left[|\DD_{x,y} F_t |^{r}\right]^{1/r}
	\]
	and the bounds in \ref{propkNNBounds} and \ref{propkNNVar} to conclude that
	\[
	\gamma_1 \lesssim t^{-d} \left(\int_{tH} \left( \int_{tH} \exp(-c|x-y|^d) (\phi(|x-y|)+1) dy\right)^p dx\right)^{1/p}.
	\]
	After changing variables and extending the domain of integration to $\R^d$, the inner integral is upper bounded by
	\[
	\int_{\R^d} \exp(-c|y|^d) (\phi(|y|)+1) dy,
	\]
	which is finite by Lemma~\ref{lemPhiProp}. We deduce that
	\[
	\gamma_1 \lesssim t^{d(1/p-1)}.
	\]
	The terms $\gamma_2,\gamma_4,\gamma_5,\gamma_6,\gamma_7$ can be shown to be $O(t^{d(1/p-1)})$ by applying the same strategy.
	
	For $\gamma_3$, take $q = 3-\frac{2}{p}$. Then $q+1<2p<r$ and $\E |\MalD_y F_t|^{q+1} \leq \E \left[|\MalD_y F_t|^{r}\right]^{(q+1)/r}$ and we have
	\[
	\gamma_3 \lesssim t^{-d(2-1/p)} \int_{tH} \E \left[|\MalD_y F_t|^{r}\right]^{(q+1)/r} dx \lesssim t^{d(1/p-1)}.
	\]
	To show that in particular this bound is true for the function $\phi(x)=x^{-\alpha}$ with $0<\alpha<\frac{d}{2}$, it suffices to check condition \eqref{eqkNNC1}. Indeed, let $2<r<\frac{d}{\alpha}$. Then
	\[
	\int_0^1 s^{-\alpha r} s^{d-1} ds < \infty
	\]
	and the bound on the speed of convergence holds for any $p< \frac{r}{2}< \frac{d}{2\alpha}$, with $p \in (1,2]$.
\end{proof}

\section{Radial Spanning Tree}\label{secPRad}
In this section, we work in the setting of Theorem~\ref{thmRST} and we call finite sets $\mu \subset \R^d$ generic only if $\mu$ is generic with respect to $0$.

\begin{prop}\label{propRSTBounds}
	There are absolute constants $c_1,C_1>0$ (independent of $\phi$) such that for any $t \geq 1$ and any $x,y \in tH$, the following bound holds:
	\begin{equation}
		\p\left(\DD_{x,y} F_t \neq 0\right) \leq C_1 \exp(-c_1|x-y|^d). \label{eqRSTp}
	\end{equation}
	Moreover, for any $1 \leq p \leq \frac{r}{2}$, there are constants $c_2,C_2>0$ such that for any $t \geq 1$ and any $x,y \in tH$, the following bounds hold:
	\begin{align}
		&\E \left[ |\MalD_{x}F_t|^{2p} \right]^{1/(2p)} \leq C_2 \left(1+\phi(|x|) \exp(-c_2 |x|^d)\right) \label{eqRSTDx}\\
		&\E \left[ |\DD_{x,y} F_t |^{2p} \right]^{1/(2p)} \leq C_2 \left(\phi(|x-y|) \exp(-c_2|x-y|^d) + 1 \right) \label{eqRSTDxy}
	\end{align}
	We also have that $F_t \in \dom \MalD$.
\end{prop}
The proof relies on and reuses some arguments from the proofs of Lemmas~4.1 and 4.2 in \cite{ST14rst}.

\begin{proof}
	The bound \eqref{eqRSTp} follows immediately from the proof of Lemma~4.2 in \cite{ST14rst}. Indeed, the proof given in \cite{ST14rst} does not depend on the chosen functional, but only on the structure of the graph.
	
	Next, we establish that for any $x \in \R^d$ and $\mu \subset \R^d$ finite and generic with respect to $x$, we have $\MalD_{x} F(\mu) \geq 0$. Indeed, it is true that
	\begin{equation}\label{Rainbow8}
	\MalD_{x} F(\mu) = \phi(g(x,\mu)) + \sum_{y \in \mu} \left(\phi(g(y,\mu + \delta_x)) - \phi(g(y,\mu))\right).
	\end{equation}
	Since $0 < g(y,\mu + \delta_x) \leq g(y,\mu)$ and $\phi$ is non-increasing, the above expression is non-negative.
	
	For the next part, define $e(x,\mu)$ as we did in Lemma~\ref{lemPhiProp} and extend $\phi$ by setting $\phi(0)=0$. Then note that the summand on the RHS of \eqref{Rainbow8} is zero, unless $y$ connects to $x$, in which case $g(y,\mu +\delta_x)=|x-y|$. It follows that
	\[
	0 \leq \MalD_{x} F(\mu) \leq \phi(g(x,\mu)) + \sum_{y \in \mu} \ind{y \rightarrow x \text{ in } \mu + \delta_x} \phi(|x-y|).
	\]
	If $\mu$ is non-empty, then $|x-y| \geq e(x,\mu)$, hence the second term on the RHS of \eqref{Rainbow8} is upper bounded by
	\[
	\phi(e(x,\mu)) \sum_{y \in \mu} \ind{y \rightarrow x \text{ in } \mu + \delta_x},
	\]
	a bound that continues to hold if $\mu = \emptyset$.
	
	As for the first term on the RHS of \eqref{Rainbow8}, if $\mu \cap B^d(0,|x|) \cap B^d(x,|x|)$ is empty, then the term is equal to $\phi(|x|)$. If not, then $g(x,\mu) \geq e(x,\mu)$ and it is upper bounded by $\phi(e(x,\mu))$. Hence we deduce the bound
	\begin{equation}\label{Rainbow33}
	\MalD_{x} F(\mu) \leq \phi(|x|) \ind{\mu \cap B^d(0,|x|) \cap B^d(x,|x|) \neq \emptyset} + \phi(e(x,\mu)) \left(1 + \sum_{y \in \mu} \ind{y \rightarrow x \text{ in } \mu + \delta_x} \right).
	\end{equation}
	It follows by Hölder's and Minkowski's inequalities that
	\begin{multline}\label{iP12Bound}
	\E \left[ \left(\MalD_{x} F(\mu)\right)^{2p} \right]^{1/(2p)} \leq \phi(|x|) \p(\eta(tH \cap B^d(0,|x|) \cap B^d(x,|x|))=0)^{1/(2p)}\\%
	+ \E \left[\phi(e(x,\eta_{|tH}))^r\right]^{1/r} \E \bigg[ \bigg(1 + \sum_{y \in \mu} \ind{y \rightarrow x \text{ in } \mu + \delta_x} \bigg)^m\bigg]^{1/m},
	\end{multline}
	where $m = \ceil{\frac{2pr}{r-2p}}$.
	
	We start with the first term on the RHS of \eqref{iP12Bound}. We know that
	\[
	\p(\eta(tH \cap B^d(0,|x|) \cap B^d(x,|x|))=0) = \exp(-|tH \cap B^d(0,|x|) \cap B^d(x,|x|)|).
	\]
	Note that $B^d(\frac{x}{2},\frac{|x|}{2}) \subset B^d(0,|x|) \cap B^d(x,|x|)$, therefore
	\[
	\left|tH \cap B^d(0,|x|) \cap B^d(x,|x|)\right| \geq \left|tH \cap B^d\left(\frac{x}{2},\frac{|x|}{2}\right)\right|.
	\]
	Since $H$ is non-empty, there is a ball $B^d(x_0,\delta) \subset H$ with $x_0 \in H$ and $\delta>0$. By scaling and use of Lemma~\ref{lemGeo1}, we conclude that there is a constant $c_H>0$ depending on $H$ and $d$ such that
	\[
	|tH \cap B^d\left(\tfrac{x}{2},\tfrac{|x|}{2}\right)| \geq c_H |x|^d.
	\]
	Hence
	\begin{equation}\label{Rainbow34}
	\phi(|x|) \p\left(\eta(tH \cap B^d(0,|x|) \cap B^d(x,|x|))=0\right)^{1/(2p)} \leq \phi(|x|) \exp\left(-\tfrac{c_H}{2p}|x|^d\right).
	\end{equation}
	For the second term on the RHS of \eqref{iP12Bound}, we use that $\E \left[\phi(e(x,\eta_{|tH}))^r\right]^{1/r}$ is bounded by a constant independent of $t$ and $x$, by Lemma~\ref{lemPhiProp}. For the other part of the second term in \eqref{iP12Bound}, one can easily adapt the argument in the proof of \cite[Lemma~4.1]{ST14rst} to show that for any $m \in \N$,
	\[
	\E \left[ \left(1 + \sum_{y \in \mu} \ind{y \rightarrow x \text{ in } \mu + \delta_x} \right)^m\right]^{1/m}
	\]
	is uniformly bounded by a constant. This concludes the proof of \eqref{eqRSTDx}.
	
	Now we prove \eqref{eqRSTDxy}. For $x,y \in \R^d$ and $\mu \subset \R^d$ finite generic with respect to $x,y$, we have
	\begin{align}
		\DD_{x,y} F(\mu) &= \phi(g(x,\mu + \delta_y)) + \sum_{z \in \mu} \left(\phi(|x-z|) - \phi(g(z,\mu + \delta_y))\right) \ind{z \rightarrow x \text{ in } \mu + \delta_y + \delta_x}\notag\\
		&+ \left(\phi(|x-y|) - \phi(g(y,\mu))\right) \ind{y \rightarrow x \text{ in } \mu + \delta_y + \delta_x} \notag\\
		&-\phi(g(x,\mu)) - \sum_{z \in \mu} \left(\phi(|x-z|) - \phi(g(z,\mu))\right) \ind{z \rightarrow x \text{ in } \mu + \delta_x}. \label{Rainbow31}
	\end{align}
	This expression is in fact symmetric in $x$ and $y$ since the operator $\DD_{x,y}$ is symmetric. Assume without loss of generality that $|x| \leq |y|$. Then $x$ cannot connect to $y$ and hence $\phi(g(x,\mu + \delta_y)) = \phi(g(x,\mu))$. The point $y$ will connect to $x$ if and only if $|x-y| < |y|$ and $\mu \cap B^d(0,|y|) \cap B^d(y,|x-y|)=\emptyset$, thus
	\begin{multline}\label{Rainbow28}
	\left(\phi(|x-y|) - \phi(g(y,\mu))\right) \ind{y \rightarrow x \text{ in } \mu + \delta_y + \delta_x} \\= \phi(|x-y|) \ind{|x-y| < |y|} \ind{\mu \cap B^d(0,|y|) \cap B^d(y,|x-y|)=\emptyset}.
	\end{multline}
	Moreover, a point $z$ that connects to $x$ in $\mu + \delta_y + \delta_x$ also connects to $x$ in $\mu+\delta_x$ and in this case $\phi(|x-z|) \geq \phi(g(z,\mu+\delta_y))$. We deduce that
	\begin{equation}\label{Rainbow29}
		\sum_{z \in \mu} \left(\phi(|x-z|) - \phi(g(z,\mu + \delta_y))\right) \ind{z \rightarrow x \text{ in } \mu + \delta_y + \delta_x} \leq \sum_{z \in \mu} \phi(|x-z|) \ind{z \rightarrow x \text{ in } \mu + \delta_x}.
	\end{equation}
	Lastly, if $z$ connects to $x$ in $\mu + \delta_x$, then $\phi(|x-z|) \geq \phi(g(z,\mu))$, thus
	\begin{equation}\label{Rainbow30}
		\sum_{z \in \mu} \left(\phi(|x-z|) - \phi(g(z,\mu))\right) \ind{z \rightarrow x \text{ in } \mu + \delta_x} \leq \sum_{z \in \mu} \phi(|x-z|) \ind{z \rightarrow x \text{ in } \mu + \delta_x}.
	\end{equation}
	Combining \eqref{Rainbow28}, \eqref{Rainbow29} and \eqref{Rainbow30}, we deduce from \eqref{Rainbow31} that
	\begin{equation}\label{Rainbow32}
	\left|\DD_{x,y} F_t\right| \leq \phi(|x-y|) \ind{|x-y| < |y|} \ind{\eta(tH \cap B^d(0,|y|) \cap B^d(y,|x-y|))=0} + 2 \sum_{z \in \eta_{|tH}} \phi(|x-z|) \ind{z \rightarrow x \text{ in } \eta}.
	\end{equation}
	The proof of \eqref{eqRSTDx} shows that the $2p$th moment of the second term is uniformly bounded. When $|x-y| < |y|$, then $B^d\left(y(1-\frac{|x-y|}{2|y|}),\frac{|x-y|}{2}\right)$ is included in the intersection $B^d(0,|y|) \cap B^d(y,|x-y|)$. By Lemma~\ref{lemGeo1}, and rescaling, we have thus for some constant $c_H>0$ that
	\begin{multline}
	\p(\eta(tH \cap B^d(0,|y|) \cap B^d(y,|x-y|))=0) \\= \exp(-|tH \cap B^d(0,|y|) \cap B^d(y,|x-y|)|) \leq \exp(-c_H |x-y|^d).
	\end{multline}
	The bound on the $2p$th moment of the first term in \eqref{Rainbow32} follows.
	
	The fact that $\MalD_{x} F_t$ is square-integrable follows from the bound \eqref{eqRSTDx} and Lemma~\ref{lemPhiProp}. To show that $F_t \in \dom \MalD$, it suffices to show that $F_t \in L^1(\p_\eta)$, as was explained in Section~\ref{secFrame}. We apply Mecke equation and the fact that $g(x,\eta_{|tH} + \delta_x) = g(x,\eta_{|tH})$ to deduce that
	\[
	\E F_t = \int_{tH} \E \phi(g(x,\eta_{|tH})) dx.
	\]
	As shown in the discussion leading to \eqref{Rainbow33}, this is bounded by
	\[
	\int_{tH} \left(\phi(|x|)\p\left(\eta(tH \cap B^d(0,|x|) \cap B^d(x,|x|))=0\right) + \phi(e(x,\mu))\right) dx.
	\]
	By \eqref{Rainbow34} and Lemma~\ref{lemPhiProp}, this integral is finite.
\end{proof}

\begin{prop}\label{propRSTvar}
	Under the conditions above, there are constants $c_1,c_2>0$ such that for all $t \geq 1$ large enough
	\[
	c_1 t^d \leq \var F_t \leq c_2 t^d.
	\]
\end{prop}

\begin{proof}
	The upper bound follows from \eqref{eqRSTDx} with $p=1$ and similar integration arguments as in the proof of Proposition~\ref{propRSTBounds}. For the lower bound, we use Theorems~5.2 and 5.3 in \cite{LPS14} again.
	
	Recall that $B^d(0,\epsilon) \subset H$ and let $\delta < \frac{\epsilon}{4}$. Fix $t \geq 1$ and define the set
	\[
	U:= \{ (x,z) : x \in tH \setminus B^d(0,4\delta), z \in B^d((1-\tfrac{\delta}{|x|})x,\delta) \}.
	\]
	Note that $B^d\left(\left(1-\frac{\delta}{|x|}\right)x,\delta\right) \subset B^d(0,|x|) \cap B^d(x,|x|)$, therefore any $z \in B^d((1-\frac{\delta}{|x|})x,\delta)$ verifies $|x-z|<|x|$. Moreover, $|x-z| \leq \left|x-z-\frac{\delta}{|x|}x\right| + \left|\frac{\delta}{|x|}x\right| \leq 2\delta$.
	
	Now for $(x,z) \in U$, consider $\MalD_x F(\eta_{|tH} + \delta_z)$. As can be seen from \eqref{Rainbow8}, we have
	\[
	\MalD_x F(\eta_{|tH} + \delta_z) \geq \phi(g(x,\eta_{|tH}+\delta_z)).
	\]
	By our choice of $z$, we have $g(x,\eta_{|tH}+\delta_z) \leq |x-z| \leq 2\delta$. As $\phi$ is decreasing, it follows that $\phi(g(x,\eta_{|tH}+\delta_z)) \geq \phi(2\delta)$. By \cite[Theorem~5.2]{LPS14},
	\[
	\var(F_t) \geq \frac{\phi(2\delta)^2}{4^3 \cdot 2} \min_{\emptyset \neq J \subset \{1,2\}} \inf_{\substack{V \subset U \\ \lambda(V) \geq \lambda(U)/8}} \lambda^{(|J|)} (\Pi_J(V)),
	\]
	where $\lambda$ is the Lebesgue measure.
	
	By the reasoning in the proof of \cite[Theorem~5.3]{LPS14}, this quantity is lower bounded by $\lambda(U)$ up to multiplication by a constant. We have
	\[
	\lambda(U) = \kappa_d \delta^d |tH \setminus B^d(0,4\delta)| = \kappa_d \delta^d (t^d |H| - \kappa_d(4\delta)^d),
	\]
	which is of order $t^d$ for $t$ large enough.
\end{proof}

In the development below, most constants will be called $c$ for convenience. These constants are by no means the same and can (and will) change from line to line or indeed, within lines.

\begin{proof}[Proof of Theorem~\ref{thmRST}]
	We apply the bounds found in Propositions~\ref{propRSTBounds} and ~\ref{propRSTvar} to the terms $\gamma_1$,...,$\gamma_7$ in Theorem~\ref{thmWasserNon-Cond.}. Let $p \in (1,2]$ such that $2p<r$ and choose $s>0$ such that $2p<s<r$. Then by Hölder's inequality and \eqref{eqRSTDxy} and \eqref{eqRSTp},
	\[
	\E \left[|\DD_{x,y} F_t|^{2p}\right]^{1/(2p)} \lesssim \exp(-c|x-y|^d)\E \big[|\DD_{x,y} F_t|^{s}\big]^{1/s} \lesssim \exp(-c|x-y|^d) (1+\phi(|x-y|)).
	\]
	With this bound, we find that $\gamma_1$ is bounded by
	\[
	\gamma_1 \lesssim t^{-d} \left(\int_{tH} \left(\int_{tH} \left(1 + \phi(|x|)\exp(-c|x|^d)\right) \exp(-c|x-y|^d) \left(1+\phi(|x-y|)\right) dy\right)^p  dx\right)^{1/p}.
	\]
	Using Lemma~\ref{lemPhiProp}, we find $\gamma_1=O(t^{d(1/p-1)})$. The same way, we have $\gamma_2=O(t^{d(1/p-1)})$. For $\gamma_3$, take $q=3-\frac{2}{p}$, then $q+1 < 2p < s$ and $\E |\MalD_y F_t|^{q+1} \leq \E \left[|\MalD_y F_t|^{s}\right]^{(q+1)/s}$. It follows that
	\[
	\gamma_3 \lesssim t^{-d(2-1/p)} \int_{tH} (1+\phi(|x|)\exp(-c|x|^d))^{q+1} dx,
	\]
	which is of order $O(t^{d(1/p-1)})$. The terms $\gamma_4,\gamma_5$ and $\gamma_6$ work much the same. For $\gamma_7$ we use the simplified version proposed in Remark~\ref{remGamma7} and get
	\[
	\gamma_7 \lesssim t^{-d} \left(\int_{tH} \int_{tH} \exp(-c|x-y|^d) (\phi(|x-y|)+1) (1+\phi(|x|) \exp(-c|x|^d))^{2p-1}dxdy\right)^{1/p}.
	\]
	By the same methods as applied before, this is $O(t^{d(1/p-1)})$.
	
	The claim that the bound \eqref{eqRST} is true in particular for $\phi(s) = s^{-\alpha}$ with $0<\alpha<\frac{d}{2}$ follows from the discussion at the end of the proof of Theorem~\ref{thmkNN}.
\end{proof}

\bibliographystyle{alpha}
\bibliography{bibliography.bib}


\end{document}